\documentclass{amsart}

\usepackage{amsmath,amsfonts,amssymb,amsthm,amscd,comment,euscript}
\usepackage{etoolbox}
\usepackage{stmaryrd}
\usepackage[all]{xy}
\usepackage{graphicx}
\usepackage{enumerate, enumitem} 
\usepackage[breaklinks, colorlinks=true, linktocpage]{hyperref} 
\usepackage[usenames,dvipsnames]{xcolor}
\usepackage{tikz-cd}
\usepackage{xcolor}
\usepackage{soul}
\usepackage{stmaryrd} 

\usepackage{hyperref} 

\usepackage[capitalize]{cleveref}   

\theoremstyle{plain}
\newtheorem{lem}{Lemma}[section]
\newtheorem{cor}[lem]{Corollary}
\newtheorem{prop}[lem]{Proposition}

\newtheorem{thm}[lem]{Theorem}
\newtheorem{calc}[lem]{Calculation}
\newtheorem{assumption}[lem]{Assumption}

\theoremstyle{definition}
\newtheorem{ex}[lem]{Example}
\newtheorem{rem}[lem]{Remark}
\newtheorem{dfn}[lem]{Definition}
\newtheorem{nota}[lem]{Notation}

\newcommand{\arxiv}[1]{\href{http://arxiv.org/abs/#1}{\tt arXiv:\nolinkurl{#1}}}

\newcommand{\arxivpdf}[1]{\href{http://arxiv.org/pdf/#1}{\tt arXiv:\nolinkurl{#1}}}

\newtoggle{comments}
\newtoggle{details}
\newtoggle{detailsnote}

\toggletrue{comments}   
\togglefalse{details}   

\toggletrue{detailsnote}   

\iftoggle{comments}{%
	\newcommand{\comments}[1]{
		\ \\
		{\color{red}
			\textbf{RG:} #1
		}
		\\
	}
}{%
	\newcommand{\comments}[1]{}
}

\iftoggle{details}{%
	\newcommand{\details}[1]{
		\ \\
		{\color{OliveGreen}
			\textbf{Details:} #1
		}
		\\
	}
}{%
	\newcommand{\details}[1]{}
}

\begin{document}

\title[]{The formal affine Demazure algebra and real finite reflection groups}

\author[]{Raj Gandhi}

\address{
	Department of Mathematics and Statistics \\
	University of Ottawa
}
\email{rgand037@uottawa.ca}

\address{} 
\thanks{}

\address{}
\thanks{}



\begin{abstract}
	In this paper, we generalize the formal affine Demazure algebra of Hoffnung-Malag\'on-L\'opez-Savage-Zainoulline to all real finite reflection groups. We begin by generalizing the formal group ring of Calm\`es-Petrov-Zainoulline to all real finite reflection groups. We then define and study the formal Demazure operators that act on the formal group ring. Using these results and constructions, we define and study the formal affine Demazure algebra for all real finite reflection groups. Finally, we compute several structure coefficients that appear in a braid relation among the formal Demazure elements, and we conclude this paper by computing all structure coefficients for the reflection groups $I_2(5)$, $I_2(7)$, $H_3$, and $H_4$.
\end{abstract}

\maketitle

%

\tableofcontents

\section{Introduction}

Let $\mathrm{h}^*$ be an algebraic oriented cohomology theory in the sense of Levine-Morel \cite{LM07}. If $X$ is a smooth projective variety over a field $k$, then, given a vector bundle $E\to X$ over $X$ of rank $n$, there is a set of Chern classes $c_i(E)\in \mathrm{h}^i(X)$, $i\in\{0,\dotsc,n\}$. Any two line bundles $\mathcal{L}_1$ and $\mathcal{L}_2$ over $X$ satisfy \textit{Quillen's formula},
\[c_1(\mathcal{L}_1\otimes \mathcal{L}_2)=F(c_1(\mathcal{L}_1),c_1(\mathcal{L}_2)),\]
where $F$ is the formal group law associated to $\mathrm{h}^*$. 
Suppose $G$ is a split semisimple simply connected linear algebraic group over a field $k$, and let $T$ be a maximal split torus sitting inside a Borel subgroup $B$ of $G$, so that $G/B$ is a complete flag variety.
In \cite{CPZ}, Calm\`es-Petrov-Zainoulline constructed the \textit{formal group ring}, which we will denote $\widetilde{R\llbracket\Lambda\rrbracket}_F$, where $F$ is a one-dimensional commutative formal group law over a commutative unital ring $R$. If $\mathrm{h}^*$ is the Chow theory, then $F$ is the additive formal group law over $R=\mathbb{Z}$, and $\widetilde{R\llbracket\Lambda\rrbracket}_F$ is the completion of the symmetric algebra $S_\mathbb{Z}^*(\Lambda)$, where $\Lambda$ is the character lattice of $T$. When $R=\mathrm{h}^*(\mathrm{Spec}(k))$, the authors showed that there is a ring homomorphism, called the chararacteristic map,
\[\mathfrak{c}_{G/B}:\widetilde{R\llbracket\Lambda\rrbracket}_F\to \mathrm{h}^*(G/B).\]
Using this map, the authors constructed an algebraic model $\mathcal{H}(\Lambda)_F$ for the algebraic oriented cohomology ring $\mathrm{h}^*(G/B)$ in terms of augmented Demazure operators. Specializing $\mathrm{h}^*$ to the Chow theory, $\mathcal{H}(\Lambda)_F$ is the algebraic model for the Chow theory of the flag variety $G/B$ constructed by Demazure in \cite{D73}. 

In \cite{KK86,KK90}, Kostant-Kumar described the equivariant cohomology and $K$-theory of flag varieties using the techniques of nil-Hecke and 0-Hecke algebras. These algebras are generated by Demazure operators, which satisfy a braid relation.
Calm\`es, Hoffnung, Malag\'on-L\'opez, Savage, Zainoulline, and Zhong used the formal group ring $\widetilde{R\llbracket\Lambda\rrbracket}_F$ to generalize the nil-Hecke and $0$-Hecke algebras of Kostant-Kumar to an arbitrary algebraic oriented cohomology theory \cite{CZZ3,CZZ1,CZZ2,HMSZ}. In particular, they constructed and studied the formal affine Demazure algebra, $\mathbf{D}_{F}$, which is generated by formal Demazure elements that satisfy a $\textit{twisted}$ braid relation (see \cite[Prop. 6.8]{HMSZ}). After specializing $F$ to the additive and multiplicative formal group laws, $\mathbf{D}_{F}$ equals completion of the nil-Hecke and 0-Hecke algebra, respectively.

Let $W$ be a real finite reflection group, let $\Sigma$ be a root system of $W$, and let $\Delta$ be a simple system of $\Sigma$. The first goal of this paper is to generalize the formal group ring of \cite{CPZ} to all real finite reflection groups. We begin by defining the \textit{real root lattice} $\Lambda$ with respect to the pair $(\Sigma,\Delta)$. When $\Sigma$ is crystallographic, $\Lambda$ is the crystallographic root lattice, as described in \cite[\S 2]{CZZ1}. Using the data of a formal group law $F$ over a subring $R$ of $\mathbb{C}$ satisfying the conditions of \cref{assumption-ring}, and the real root lattice $\Lambda$, we define the \textit{formal group ring} $R\llbracket\Lambda\rrbracket_F$. When $\Sigma$ is crystallographic, $R\llbracket\Lambda\rrbracket_F$ is the formal group ring $\widetilde{R\llbracket\Lambda\rrbracket}_F$ of \cite{CPZ}. Once $R\llbracket\Lambda\rrbracket_F$ is defined, we study the subalgebra $\mathcal{D}_{(R,F)}(\Lambda)$ of the endomorphism algebra of $R\llbracket\Lambda\rrbracket_F$ generated by formal Demazure operators and by multiplication by elements in $R\llbracket\Lambda\rrbracket_F$. In particular, we show that $\mathcal{D}_{(R,F)}(\Lambda)$ is a free finitely-generated $R\llbracket\Lambda\rrbracket_F$-module. When $\Sigma$ is crystallographic, the  $R\llbracket\Lambda\rrbracket_F$-algebra $\mathcal{D}_{(R,F)}(\Lambda)$ is the endomorphism algebra studied in \cite {CPZ}. When $F$ is the additive formal group law and $R=\mathbb{C}$, the  $R\llbracket\Lambda\rrbracket_F$-algebra $\mathcal{D}_{(R,F)}(\Lambda)$ is the completion of the endomorphism algebra studied by Hiller in \cite[Ch. IV]{H82}. When $F$ is the additive formal group law, $R=\mathbb{Z}$, and $\Sigma$ is crystallographic, the $R\llbracket\Lambda\rrbracket_F$-algebra $\mathcal{D}_{(R,F)}(\Lambda)$ is the completion of the endomorphism algebra studied by Demazure in \cite{D73}.

The second goal of this paper is to extend the formal affine Demazure algebra of \cite{HMSZ} to all real finite reflection groups. We first define the formal affine Demazure algebra $\mathbf{D}_{F}$ for all real finite reflection groups using the formal group ring $R\llbracket\Lambda\rrbracket_F$. We then show that several results in \cite{CZZ1,HMSZ} regarding the crystallographic formal affine Demazure algebra extend to all real finite reflection groups. In particular, we describe $\mathbf{D}_{F}$ in terms of generators and relations in \cref{thm:basis-theorem}, as was done in \cite[Thm. 7.9]{CZZ1}. We then compute several structure coefficients that appear in a braid relation among the formal Demazure elements (see \cref{thm:linear-combo2}), and we introduce a seemingly new formula involving the formal Demazure elements in the process (see \cref{cor:linear-combo-2}).
In addition, we compute all structure coefficients for the dihedral groups $I_2(5)$ and $I_2(7)$ in \cref{ex:order-5} and \cref{ex:order-7}, and the reflection groups $H_3$ and $H_4$ in \cref{ex:H3} and \cref{ex:H4}.
The formulas for the structure coefficients that we obtain in this paper generalize several formulas that appear in the literature, such as the formulas in \cite[Prop. 6.8]{HMSZ} and \cite[Section 8]{GR13}, to all real finite reflection groups.

This paper is organized as follows. In \cref{section:generalized-root-lattice}, we recall the definition and basic properties of real finite reflection groups and their associated root systems. We then define the real root lattice $\Lambda$, generalizing the crystallographic root lattice to noncrystallographic root systems. In \cref{section:regularity}, we recall the definition and basic properties of formal groups laws, and we discuss the formal group ring constructed in \cite{CPZ} with respect to the real root lattice. In \cref{section:formal-group-algebra}, we define the formal group ring for all real finite reflection groups. In \cref{section:localization}, we define and discuss the formal Demazure operators acting on the formal group ring. In \cref{section:endomorphisms}, we define the associated graded ring with respect to the formal group ring and use it to prove a basis theorem for a subalgebra of the endomorphism algebra of the formal group ring containing the formal Demazure operators. This section closely follows \cite[\S 4]{CPZ}. In \cref{section:Demazure-elements}, we define and discuss the formal affine Demazure algebra for all real finite reflection groups. In \cref{section:presentation}, we prove a presentation for the formal affine Demazure algebra in terms of the formal Demazure elements. This section closely follows \cite[\S 6 and \S7]{CZZ1}.
In \cref{section:non-simply-laced}, we compute several structure coefficients that appear in a braid relation among the formal Demazure elements. In \cref{section:applications}, we analyze the structure coefficients derived in \cref{section:non-simply-laced} after specializing to certain formal group laws, and we compute all structure coefficients for the groups $I_2(5)$, $I_2(7)$, $H_3$, and $H_4$. In \cref{section:appendix}, we provide computations for products of up to seven formal Demazure elements for all real finite reflection groups.

\

\textbf{Acknowledgements}: The author is grateful to K.~Zainoulline for proposing the project of generalizing the formal affine Demazure algebra to all real finite reflection groups, for providing many comments and suggestions in the construction of the generalized formal group ring, and for his encouragement to carry out this project to its completion. The author would also like to thank A.~Savage for his careful reading of the paper and for his comments on earlier drafts of this work. Finally, the author would like to thank K. Besrour, J. Champaign, C. Kioulos, and S. Pilon for some useful discussions. This project was partially supported by an Undergraduate Student Research Award and a Canada Graduate Scholarship - M.Sc. from the Natural Sciences and Engineering Research Council of Canada. It was also supported by an Ontario Graduate Scholarship and by the University of Ottawa.
\section{Real finite reflection groups and the real root lattice}{\label{section:generalized-root-lattice}}
In this section, we recall the definition and basic properties of real finite reflection groups and their associated root systems. We then define the real root lattice $\Lambda$, generalizing the crystallographic root lattice to noncrystallographic root systems.

Let $V=\mathbb{R}^n$, and let $(\cdot{ , }\cdot)$ be the standard inner product on $V$. For any $\alpha \in V$, the \textit{reflection} across $\alpha$ is the linear operator $s_\alpha$ defined by the formula
\[s_\alpha(v)=v-2\tfrac{(\alpha,v)}{(\alpha,\alpha)}\alpha,\quad v\in V.\]
One can easily compute that $s_\alpha^2=1$.
\begin{dfn}{\textup{(see} \cite[pp. 6]{H90} or \cite[Ch. I, \S 3, Def. 3.1]{H82}\textup{)}}{\label{dfn:root-system}}
	A \textit{root system} $\Sigma$ in $V$ is a finite set  of nonzero vectors in $V$ satisfying the conditions:
	\begin{enumerate}
		\item $\Sigma\cap \mathbb{R}\alpha=\{\alpha,-\alpha\}$ for all $\alpha\in\Sigma$;
		\item $s_\alpha(\Sigma)=\Sigma$ for all $\alpha\in\Sigma$;
		\item The roots $\alpha\in \Sigma$ generate $V$.
	\end{enumerate}
\end{dfn}
The \textit{rank} of a root system in $V$ is the dimension of $V$ as an $\mathbb{R}$-vector space. Fix any root system $\Sigma$ in $V$, and let $\alpha\in\Sigma$ be any root. Let $\alpha^\vee:V\to \mathbb{R}$ be the linear function defined by $\alpha^\vee(\beta)=2\tfrac{(\alpha,\beta)}{(\alpha,\alpha)}$, $\beta\in V$. We call $\alpha^\vee$ the \textit{coroot} of $\alpha$. We adopt the following naming convention:
\begin{itemize}
	\item If $\alpha^\vee(\beta)\in\mathbb{Z}$ for all $\alpha,\beta\in\Sigma$, we call $\Sigma$ a \textit{crystallographic root system}. Otherwise, we call $\Sigma$ a \textit{noncrystallographic root system}.
\end{itemize}
Let $W$ be the group generated by the reflections $s_\alpha$ over all roots $\alpha\in\Sigma$. We call $W$ the \textit{real finite reflection group} of $\Sigma$. If $\Sigma$ is crystallographic, we call $W$ a \textit{Weyl group}. Let $\Delta=\{\alpha_1,\dotsc,\alpha_n\}$ be a subset of $\Sigma$. We call $\Delta$ a \textit{simple system} of $\Sigma$ if it is an $\mathbb{R}$-basis of $V$, and if every root $\alpha\in\Sigma$ can be written as an $\mathbb{R}$-linear combination of elements in $\Delta$ with all coefficients nonnpositive or all coefficients nonnegative. The elements in $\Delta$ are called \textit{simple roots} and the reflection across a simple root is called a \textit{simple reflection}. It is shown in \cite[Ch. I, Thm. 1.3]{H90} that every root system contains a simple system. If a root $\alpha$ can be written as a linear combination of the simple roots with all coefficients nonpositive (resp. nonnegative), then $\alpha$ is called a \textit{negative root} (resp. \textit{positive root}). The \textit{positive system} $\Sigma^+$ is the set of positive roots of $\Sigma$ with respect to $\Delta$, and the \textit{negative system} $\Sigma^-$ is the set of negative roots of $\Sigma$ with respect to $\Delta$. If $\alpha_i$ is a simple root, we denote a reflection along $s_{\alpha_i}$ by $s_i$, and we set $S=\{s_i\mid  \alpha_i\in\Delta\}$.

\begin{lem}{\textup{(see} \cite[Ch. I, Cor. 1.5]{H90}\textup{)}}{\label{lem:root-conjugation}}
	For any $\alpha\in\Sigma$, there is a reflection $w\in W$ and a simple root $\alpha_i\in\Delta$ such that $\alpha=w(\alpha_i).$
\end{lem}

\begin{thm}{\textup{(see} \cite[Ch. I, Thm. 1.9]{H90}\textup{)}}{\label{thm:generation-by-simple-reflections}}
	$W$ is generated by the set of simple reflections $S$, subject only to the relations:
	\[(s_is_j)^{m_{i,j}}=1,\quad \alpha_i,\alpha_j\in\Delta,\]
	where $m_{i,j}$ is the order of $s_is_j$ in $W$. 
\end{thm}

\begin{rem}
	\cref{lem:root-conjugation} and \cref{thm:generation-by-simple-reflections} tell us that, if $\Delta$ is a simple system of a root system $\Sigma$, then $\Sigma$ consists of the elements $w(\alpha_j)$, where $w=s_{i_1}\dotsb s_{i_k}$ is a product of simple reflections and $\alpha_j\in\Delta$.
\end{rem}

\begin{thm}{\textup{(see} \cite[Ch. I, \S 2]{H90}\textup{)}}
	Let $\Sigma$ be a noncrystallographic root system, and let $W$ be the real finite reflection group of $\Sigma$. Then $W$ is either the dihedral group $I_2(m)$, $m\geq 3$, of order $2m$, the full icosahedral group $H_3$, or the symmetry group $H_4$ of the hyperdodecahedron.
\end{thm}

\begin{dfn}{\textup(see \cite[Ch. I, \S 1.6 and Ch. II, \S 5.9]{H90})}
 By \cref{thm:generation-by-simple-reflections}, we can write $w\in W$ as a product $w=s_{i_1}\dotsb s_{i_r}$ of simple reflections. The \textit{length} $l(w)$ of $w$ is the minimal $r$ for which such an expression exists, and we call such an expression a \textit{reduced decomposition of $w$}. The \textit{longest word} in $W$ is the unique word in $W$ of maximal length. Write $w'\to w$ if $l(w)>l(w')$ and $w=w'u$ for some $u\in W$. We define the \textit{Bruhat order} on $W$ to be the partial ordering $<$ on $W$ defined by $w'<w$ if and only if there is a sequence $w'=w_0\to w_1\to\dotsb \to w_m=w$.
\end{dfn}

We will refer to the next few results frequently throughout this paper. 
We use the notation $s_{i,j,\dotsc}^{(k)}$ to denote the product of $k$ simple reflections $s_is_js_is_js_i\dotsb$.

\begin{prop}{\label{prop:positive-roots}}{\textup{(see} \cite[Ch. I, Prop. 3.6]{H82}\textup{)}}
	Let $v=s_{i_1}\dotsb s_{i_t}$ be a reduced decomposition of $v\in W$. Then
	\[v(\Sigma^-)\cap \Sigma^+=\{\theta_{i_j}:1\leq j\leq t\},\]
	where $\theta_{i_1}=\alpha_{i_1}$, and $\theta_{i_j}=s_{i_1}\dotsb s_{i_{j-1}}(\alpha_{i_j})$ if $j>1$.
	
	If $v$ is the longest word in $W$, then 
	\[\Sigma^+=\{\theta_{i_j}:1\leq j\leq t\}.\]
\end{prop}

\begin{lem}{\label{lem:diheral-property}}
	We have
	$\alpha_j =\begin{cases}
	s_{i,j,\dotsc}^{(m-1)}(\alpha_j),\quad \textit{m \textup{even}};\\
	s_{i,j,\dotsc}^{(m-1)}(\alpha_i),\quad \textit{m \textup{odd}}.
	\end{cases}$
\end{lem}
\begin{proof}
	Let $m_{i,j}$ be the order of $s_is_j$ in $W$, and suppose $m_{i,j}$ is odd. Since $s_{i,j,\dotsc}^{(m_{i,j})}=s_{j,i,\dotsc}^{(m_{i,j})}$, it follows from \cref{prop:positive-roots} that
	\begin{align*}
	\Sigma_{i,j}^+:&=\{\alpha_i,s_i(\alpha_j),s_i s_j(\alpha_i),\dotsc,s_{i,j,\dotsc}^{(m-1)}(\alpha_i)\}=\{\alpha_j,s_j(\alpha_i),s_j s_i(\alpha_j),\dotsc,s_{j,i,\dotsc}^{(m-1)}(\alpha_j)\}.
	\end{align*}
	Thus, 
	\begin{align*}
	s_j\left(\Sigma_{i,j}^+\right)&=\{s_j(\alpha_i),s_j s_i(\alpha_j),s_js_is_j(\alpha_i),\dotsc,s_{j,i,\dotsc}^{(m)}(\alpha_i)\}=\{s_j(\alpha_j),\alpha_i, s_i(\alpha_j),\dotsc,s_{i,j,\dotsc}^{(m-2)}(\alpha_j)\}.
	\end{align*}
	Since $s_j(\Sigma^+)\cap \Sigma^-=\{s_j(\alpha_j)\}$, we must have that $s_{i,j,\dotsc}^{(m-1)}(\alpha_i)=\alpha_j$. The proof for even $m_{i,j}$ is similar.
	\details{
		Suppose $m$ is even.
		Since $s_{\alpha\dotsc}^{(m-1)}=s_{\beta\dotsc}^{(m-1)}$, it follows from \cref{prop:positive-roots} that
		\begin{align*}
		\Sigma^+:&=\{\alpha,s_\alpha(\beta),s_\alpha s_\beta(\alpha),\dotsc,s_{\alpha\dotsc}^{(m-1)}(\beta)\}\\&=\{\beta,s_\beta(\alpha),s_\beta s_\alpha(\beta),\dotsc,s_{\beta\dotsc}^{(m-1)}(\alpha)\}.
		\end{align*}
		Note that $s_\beta(\Sigma^+)\cap \Sigma^-=\{s_\beta(\beta)\}$. Since $s_\beta(\alpha), s_\beta s_\alpha (\beta),\dotsc, s_{\beta\dotsc}^{(m-1)}(\alpha)$ all lie in $\Sigma^+$, we must have that $s_{\alpha\dotsc}^{(m-1)}(\beta)=\beta$.}	
\end{proof}

\begin{prop}\textup{(see \cite[Ch. I, \S 1.7, Ex. 2]{H90})}\label{prop:exchange}
Let $w=s_{i_1}\dotsb s_{i_r}$ be a (not necessarily reduced) decomposition of $w\in W$ in terms of simple reflections. If $l(sw)<l(w)$ for some simple reflection $s$, then there is an index $i_j$ for which $sw=ss_{i_1}\dotsb \hat{s}_{i_j}\dotsb s_{i_r}$. 	
\end{prop}

\begin{thm}{\textup{(see \cite[Ch. 2, \S 2.2, Thm. 2.2.2]{BB05})}}\label{thm:subword}
	Let $w=s_{i_1}\dotsb s_{i_q}$ be a reduced decomposition of $w\in W$. Then $u\leq w$ if and only if there is a reduced decomposition $u=s_{i_{r_1}}s_{i_{r_2}}\dotsb s_{i_{r_t}}$ of $u$ such that $1\leq r_1<r_2< \dotsb< r_t\leq q$.
\end{thm}
\begin{prop}{\textup{(see \cite[Ch. I, Prop. 1.2]{H90})}}\label{prop:orthogonal-property}
	For any root $\beta\in\Sigma$ and reflection $w\in W$, we have $ws_{\alpha}w^{-1}=s_{w(\alpha)}$.
\end{prop}

\begin{dfn}{\textup{(cf.} \cite[pp. 22]{H82}\textup{)}}
	If $\Sigma$ is a root system in $V$ with finite reflection group $W$, and $\Delta$ is a simple system of $\Sigma$, then we call the pair $(\Sigma,\Delta)$ a \textit{geometric realization} of $W$ in $V$.
\end{dfn}
\begin{rem}
As stated in \cite[Ch. I, \S 1.1, Ex. 1]{H90}, the reflections form a single conjugacy class in $I_2(m)$ when $m$ is odd, but form two conjugacy classes when $m$ is even. This tells us that, given a geometric realization of $I_2(m)$, all roots have the same length when $m$ is odd. However, only the lengths of the roots in each conjugacy class are guaranteed to be the equal when $m$ is even.
\end{rem}

\begin{rem}
	The rank $2$ root systems are $A_1\times A_1$, $A_2$, $B_2$, $G_2$, and the root systems of $I_2(m)$, $m\geq 3$.
\end{rem}

Let $W$ be a real finite reflection group, and let $(\Sigma,\Delta)$ be a geometric realization of $W$ in $V$, with simple system $\Delta=\{\alpha_1,\dotsc,\alpha_n\}$. Let $\alpha\in \Sigma$ be any root. By definition, there exist unique elements $c_i^\alpha\in \mathbb{R}$ such that $\alpha=c_1^\alpha\alpha_1+\dotsb+c_n^\alpha\alpha_n$. We would like to determine whether the subring $\mathcal{R}$ of $\mathbb{R}$ generated by the elements $c_i^{\alpha}$ over all $i=1,\dotsc,n$ and $\alpha\in\Sigma$ is a free finitely-generated $\mathbb{Z}$-module with a power basis, i.e., a basis of the form $\{1,\beta,\beta^2,\dotsc,\beta^{l-1}\}$ for some $l\geq 1$, where $\beta\in\mathcal{R}$. As $c_i^{\alpha_i}=1$, the ring $\mathcal{R}$ must contain $1$. If $\alpha$ is not a simple root, then \cref{lem:root-conjugation} and \cref{thm:generation-by-simple-reflections} imply that $\alpha=s_{i_1}\dotsb s_{i_k}(\alpha_j)$, where $\alpha_j\in\Delta$ and the $s_{i_r}$ are simple reflections. Thus, $\mathcal{R}$ is the unital subring of $\mathbb{R}$ generated by the elements $\alpha_i^\vee(\alpha_j)$ over all $\alpha_i,\alpha_j\in\Delta$.

For the remainder of this paper, we will work under the following Assumption:
\begin{assumption}\label{assumption:geometric-realization}
$W$ is a real finite reflection group, and $(\Sigma,\Delta)$ is a geometric realization of $W$ in $V$, such that $\mathcal{R}$ is a free finitely-generetated $\mathbb{Z}$-module with a power basis.
\end{assumption} 
We will show the existence of a geometric realization for $W$ satisfying the conditions of \cref{assumption:geometric-realization} for every real finite reflection group $W$. If $W$ is a Weyl group, then $\mathcal{R}=\mathbb{Z}$, since the elements $\alpha_i^\vee(\alpha_j)\in\mathbb{Z}$ for all $\alpha_i,\alpha_j\in\Delta$.

The geometric realizations of $W=H_3$ and $W=H_4$ that follow are taken from \cite[\S 2.13]{H90}, and the geometric realizations of $W=I_2(m)$, $m\geq 3$, that follow are taken from \cite[Ch. \S 1.1]{H90}. First, we will consider the case $W=H_4$. In this case, $V=\mathbb{R}^4$. Set $a:=\tfrac{1+\sqrt{5}}{4}$ and $b:=\tfrac{-1+\sqrt{5}}{4}$. Let $(\Sigma,\Delta)$ be the geometric realization of $W=H_4$ in $V$ with simple roots:
\[\alpha_1=\left(a,-\tfrac{1}{2},b,0\right);\quad \alpha_2=\left(-a,\tfrac{1}{2},b,0\right);\quad \alpha_3=\left(\tfrac{1}{2},b,-a,0\right);\quad \alpha_4=\left(-\tfrac{1}{2},-a,0,b\right).\]
A computation gives the following result:
\[\alpha_i^\vee(\alpha_j)=\begin{cases}
2,\quad \textcolor{white}{o}\text{     if $i=j$};\\
-\tau,\quad \text{if $(i,j)\in \{(1,2),(2,1)\}$};\\
-1,\quad \text{if $(i,j)\in\{(2,3),(3,2),(3,4),(4,3)\}$};\\
0,\quad \textcolor{white}{o}\text{     if $(i,j)\in\{(1,3),(3,1),(1,4),(4,1),(2,4),(4,2)\}$},
\end{cases}\]
where $\tau=\tfrac{1+\sqrt{5}}{2}$ is the well-known \textit{golden section}. The constant $\tau$ is a root of the characteristic equation $x^2-x-1=0$. Thus, we may take $\mathcal{R}=\mathbb{Z}[\tau]$, which has a power basis $\{1,\tau\}$ over $\mathbb{Z}$.

Now we will consider the case $W=H_3$. In this case, $V=\mathbb{R}^3$. Let $(\Sigma,\Delta)$ be the geometric realization of $W=H_3$ in $V$ with simple roots:
\[\alpha_1=\left(a,-\tfrac{1}{2},b\right);\quad \alpha_2=\left(-a,\tfrac{1}{2},b\right);\quad \alpha_3=\left(\tfrac{1}{2},b,-a\right).\] 
By the earlier computation, we may take $\mathcal{R}=\mathbb{Z}[\tau]$ in this case as well.

Next we will consider the dihedral groups $W=I_2(m)$, where $m\geq 3$. In this case, $V=\mathbb{R}^2$. Let $(\Sigma,\Delta)$ be the geometric realization of $W=I_2(m)$ in $V$ with simple roots:
\[\alpha_1=(-1,0);\quad \alpha_2=\left(\textrm{cos}\left(\tfrac{\pi}{m}\right),\textrm{sin}\left(\tfrac{\pi}{m}\right)\right).\]
A computation gives the following result: 
\[\alpha_i^\vee(\alpha_j)=\begin{cases}
2,\quad\quad\quad\quad\textcolor{white}{....} \text{if $i=j$};\\
-2\textrm{cos}\left(\tfrac{\pi}{m}\right),\quad\text{if $(i,j)\in\{(1,2),(2,1)\}$}.
\end{cases}\]
Thus, $\mathcal{R}=\mathbb{Z}[2\textrm{cos}(\tfrac{\pi}{m})]$, and $\mathcal{R}$ is indeed a free finitely-generated $\mathbb{Z}$-module with a power basis. To see this, we first note that $\left(\textrm{cos}(\tfrac{\pi}{m}),\textrm{sin}(\tfrac{\pi}{m})\right)=\left(\textrm{cos}(\tfrac{2\pi}{2m}),\textrm{sin}(\tfrac{2\pi}{2m})\right)$ is a primitive $2m$-th root of unity. Thus, by \cite{YY16}, $\mathcal{R}$ is the ring of integers of the number field $\mathbb{Q}(2\textrm{cos}(\tfrac{\pi}{m}))$. It is well known that the ring of integers $\mathcal{O}_K$ of a number field $K$ of degree $l$ has a power basis $\{1,\beta,\beta^2,\dotsc,\beta^{l-1}\}$, whenever $\mathcal{O}_K=\mathbb{Z}[\beta]$ for some $\beta\in\mathcal{O}_K$. Therefore, \[\left\{1,2\mathrm{cos}\left(\tfrac{\pi}{m}\right),\left(2\mathrm{cos}\left(\tfrac{\pi}{m}\right)\right)^2,\dotsc,\left(2\mathrm{cos}\left(\tfrac{\pi}{m}\right)\right)^{l-1}\right\}\] is a $\mathbb{Z}$-basis of $\mathcal{R}$, where $l$ is the degree of the field extension $\mathbb{Q}\subseteq \mathbb{Q}(2\mathrm{cos}(\tfrac{\pi}{m}))$. In \cite[Thm. 1]{L33}, it is shown that $l=\tfrac{\phi(2m)}{2}$, where $\phi$ is the Euler totient function, and that the minimal polynomial $\psi_{m}(x)$ of $2\mathrm{cos}(\tfrac{\pi}{m})$ over $\mathbb{Q}$ satisfies \[x^{-\tfrac{\phi(2m)}{2}}Q_{2m}(x)=\psi_{m}(x+x^{-1}),\] where $Q_{2m}$ is the $2m$-th cyclotomic polynomial.

We now go over four concrete examples of the geometric realization of dihedral groups  constructed above. Some of the formulas we compute in these examples are used in \cref{section:applications}.

\begin{ex}\label{ex:I2(3)}
	Suppose $m=3$. 
	Since $2\textrm{cos}\left(\tfrac{\pi}{3}\right)=1$, we have $\mathcal{R}=\mathbb{Z}$. Clearly, $\{1\}$ is a $\mathbb{Z}$-basis of $\mathcal{R}$.
	The roots of $\Sigma$ are:
	\[\alpha_1;\quad s_2(\alpha_1)=\alpha_1+\alpha_2; \quad s_1s_2(\alpha_1)=\alpha_2;\]
	\[s_2s_1s_2(\alpha_1)=-\alpha_2;\quad s_1s_2s_1s_2(\alpha_1)=-\alpha_1-\alpha_2;\quad s_2s_1s_2s_1s_2(\alpha_1)=-\alpha_1.\]
\end{ex}
\begin{ex}{\label{ex:I2(4)}}
	Suppose $m=4$. 
	Since $2\textrm{cos}\left(\tfrac{\pi}{4}\right)=\sqrt{2}$, we have $\mathcal{R}=\mathbb{Z}[\sqrt{2}]$. Moreover, $\{1,\sqrt{2}\}$ is a $\mathbb{Z}$-basis of $\mathcal{R}$, since $\sqrt{2}$ is a root of the irreducible quadratic polynomial $x^2-2$ over $\mathbb{Q}$.
	The roots of $\Sigma$ are:
	\[\alpha_1;\quad s_2(\alpha_1)=\alpha_1+\sqrt{2}\alpha_2;\quad s_1s_2s_1s_2(\alpha_1)=-\alpha_1;\quad s_2s_1s_2s_1s_2(\alpha_1)=-\alpha_1-\sqrt{2}\alpha_2;\]
	\[\alpha_2;\quad s_1(\alpha_2)=\sqrt{2}\alpha_1+\alpha_2;\quad s_2s_1s_2s_1(\alpha_2)=-\alpha_2;\quad s_1s_2s_1s_2s_1(\alpha_2)=-\sqrt{2}\alpha_1-\alpha_2.\]
\end{ex}
\begin{ex}{\label{ex:I2(5)}}
	Suppose $m=5$. 
	Since $2\textrm{cos}\left(\tfrac{\pi}{5}\right)=\tfrac{1+\sqrt{5}}{2}=\tau$ (this is the golden section, which was mentioned earlier), we have $\mathcal{R}=\mathbb{Z}[\tau]$. Moreover, $\{1,\tau\}$ is a $\mathbb{Z}$-basis of $\mathcal{R}$, since $\tau$ is a root of the irreducible quadratic polynomial $x^2-x-1$ over $\mathbb{Q}$.
	The roots of $\Sigma$ are:
	\[\alpha_1;\quad s_2(\alpha_1)=\alpha_1+\tau\alpha_2;\quad s_1s_2(\alpha_1)=\tau\alpha_1+\tau\alpha_2;\quad s_2s_1s_2(\alpha_1)=\tau\alpha_1+\alpha_2;\]
	\[s_{1,2,\dotsc}^{(4)}(\alpha_1)=\alpha_2; \quad s_{2,1,\dotsc}^{(5)}(\alpha_1)=-\alpha_2;\quad s_{1,2,\dotsc}^{(6)}(\alpha_1)=-\tau\alpha_1-\alpha_2;\]
	\[ s_{2,1,\dotsc}^{(7)}(\alpha_1)=-\tau\alpha_1-\tau\alpha_2;\quad s_{1,2,\dotsc}^{(8)}(\alpha_1)=-\alpha_1-\tau\alpha_2; \quad s_{2,1,\dotsc}^{(9)}(\alpha_1)=-\alpha_1.\]
\end{ex}
\begin{ex}{\label{ex:I2(7)}}
	Suppose $m=7$. 
	Set $\zeta=2\textrm{cos}\left(\tfrac{\pi}{7}\right)$. We have $\mathcal{R}=\mathbb{Z}[\zeta]$. Moreover, $\{1,\zeta,\zeta^2\}$ is a $\mathbb{Z}$-basis of $\mathcal{R}$, since $\zeta$ is a root of the irreducible cubic polynomial $x^3-x^2-2x+1$ over $\mathbb{Q}$.
	The roots of $\Sigma$ are:
	\[\alpha_1;\quad s_2(\alpha_1)=\alpha_1+\zeta\alpha_2;\quad s_1s_2(\alpha_1)=(\zeta^2-1)\alpha_1+\zeta\alpha_2;\quad s_2s_1s_2(\alpha_1)=(\zeta^2-1)\alpha_1+(\zeta^2-1)\alpha_2;\]\[ s_{1,2,\dotsc}^{(4)}(\alpha_1)=\zeta\alpha_1+(\zeta^2-1)\alpha_2;
	\quad s_{2,1,\dotsc}^{(5)}(\alpha_1)=\zeta\alpha_1+\alpha_2;\quad s_{1,2,\dotsc}^{(6)}(\alpha_1)=\alpha_2;\quad s_{2,1,\dotsc}^{(7)}(\alpha_1)=-\alpha_2; 
	\]
	\[s_{1,2,\dotsc}^{(8)}(\alpha_1)=-\zeta\alpha_1-\alpha_2; \quad s_{2,1,\dotsc}^{(9)}(\alpha_1)=-\zeta\alpha_2+(1-\zeta^2)\alpha_2;
	\quad
	 s_{1,2,\dotsc}^{(10)}(\alpha_1)=(1-\zeta^2)\alpha_1+(1-\zeta^2)\alpha_2;\]\[\quad s_{2,1,\dotsc}^{(11)}(\alpha_1)=(1-\zeta^2)\alpha_1-\zeta\alpha_2; \quad
	 s_{1,2,\dotsc}^{(12)}(\alpha_1)=-\alpha_1-\zeta\alpha_2;\quad s_{2,1,\dotsc}^{(13)}(\alpha_1)=-\alpha_1.\]
\end{ex}

For the remainder of this paper, we will work under the following Assumption:
\begin{assumption}\label{assumption:e1}
 $B=\{e_1,\dotsc,e_l\}$ is a power basis of $\mathcal{R}$, and $e_1=1$.
\end{assumption} 

\begin{dfn}
	Let $\Lambda$ be the $\mathcal{R}$-module generated by the roots $\alpha\in\Sigma$. We call $\Lambda$ a \textit{real root lattice}. If $\Sigma$ is crystallographic, then we call $\Lambda$ a \textit{crystallographic root lattice}.
\end{dfn}

The group $\Lambda$ is a free finitely-generated $\mathbb{Z}$-module of rank $nl$. Since $\Delta$ is an $\mathcal{R}$-basis of $\Lambda$, and $B$ is a $\mathbb{Z}$-basis of $\mathcal{R}$, we take $\{e_i\alpha_j\}$ as a $\mathbb{Z}$-basis of $\Lambda$. Moreover, since $e_1=1$, this basis contains $\Delta$. 

As $W$ acts on $\Sigma$ by permuting the roots, there is a natural action of $W$ on $\Lambda$ given by $w\cdot(r \alpha)=r (w(\alpha))$ for all $w\in W$, $\alpha\in \Sigma$, and $r\in \mathcal{R}$. By linearity, the action of $w\in W$ on $\Lambda$ is a group homomorphism.
\section{Formal group laws and the classical formal group ring}{\label{section:regularity}}
In this section, we recall the definition of a one-dimensional commutative formal group law, and we discuss several of its properties. We then analyze the formal group ring $\widetilde{R\llbracket\Lambda\rrbracket}_F$ of \cite{CPZ} associated to the real root lattice $\Lambda$ and the formal group law $(R,F)$.

\begin{dfn}{\textup{(see} \cite[pp. 4]{LM07}\textup{)}}
	A one-dimensional commutative formal group law $(R,F)$ over a commutative unital ring $R$ is a power series $F(u,v)\in R\llbracket u,v\rrbracket$ satisfying the following axioms:
	\begin{enumerate}
		\item $F(u,0)=F(0,u)=u\in R\llbracket u\rrbracket$;
		\item $F(u,v)=F(v,u)$;
		\item $F(u,F(v,w))=F(F(u,v),w)\in R\llbracket u,v,w\rrbracket$.
	\end{enumerate}
\end{dfn}

By \cite[Ch. I, \S 3, Prop. 1]{F68}, given a formal group law $(R,F)$, there is a unique power series $i(u)\in R\llbracket u\rrbracket$ such that $F(u,i(u))=F(i(u),u)=0$. The series $i(u)$ is called the \textit{formal inverse} of $u$ with respect to the formal group law $(R,F)$.

Following \cite[pp. 3]{BB10}, a \textit{morphism} $f\colon(R,F)\to (R,F')$ of formal group laws over $R$ is a power series $f(u)\in R\llbracket u\rrbracket$ such that $f(u+_F v)=f(u)+_{F'}f(v)$ and $f(0)=0$. The morphism $f$ is called a \textit{strong isomorphism} if its formal derivative satisfies $f'(0)=1$. The condition $f'(0)=1$ implies that there is a morphism $f^{-1}:(R,F')\to (R,F)$ such that $f(f^{-1}(u))=f^{-1}(f(u))=u$.

Let $(\mathbb{C},F)$ be a one-dimensional commutative formal group law over $\mathbb{C}$. Write $F(u,v)=u+v+\sum_{i,j\geq1}t_{i,j}u^iv^j$, where $t_{i,j}\in \mathbb{C}$, and suppose $(\mathbb{C},F_a)$ is the additive formal group law over $\mathbb{C}$, i.e., $F_a(u,v)=u+v$. In \cite[Ch. IV.5]{S09}, the author constructs a strong isomorphism of formal group laws $\mathrm{log}_{F}:(\mathbb{C},F)\to (\mathbb{C},F_a)$ called the \textit{logarithm} of $(\mathbb{C},F)$. The inverse of the logarithm is called the \textit{exponential} of $(\mathbb{C},F)$ and is denoted $\mathrm{exp}_F$. In \cite[Ch. IV.4]{S09}, it is shown that $\mathrm{log}_{F}(u)=\int \omega(u)$, where 
$\omega(v)=F_u(0,v)dv$
is the \textit{normalized invariant differential} of $(\mathbb{C},F)$ and $F_u(x,y)$ is the partial derivative of $F(u,v)$ with respect to $u$, evaluated at $(x,y)$.
There exist $a_i,b_i\in \mathbb{C}$, $i\geq 2$, such that
\[\mathrm{exp}_{F}(u)=u+\sum_{i\geq 2}a_iu^{i};\qquad \mathrm{log}_{F}(u)=u+\sum_{i\geq 2}b_iu^{i}.\]
Note that $\mathrm{exp}_{F_a}(u)=u$ and $\mathrm{log}_{F_a}(u)=u$. Set $C_F:=\{a_{i}\}_{i\geq 2}\cup\{b_{i}\}_{i\geq 2}\cup\{t_{i,j}\}_{i,j\geq 1}$. 
\begin{dfn}{\label{dfn:ample-ring}}
	If $R$ is any subring of $\mathbb{C}$ containing $C_F$, then we call $R$ an \textit{ample ring} with respect to the formal group law $(\mathbb{C},F)$.
\end{dfn}
If $R$ is an ample ring with respect to $(\mathbb{C},F)$, then we may view $(\mathbb{C},F)$ as a formal group law $(R,F)$ over $R$. In particular, there is a strong isomorphism of formal group laws $\mathrm{log}_{F}:(R,F)\to (R,F_a)$, with inverse $\mathrm{exp}_{F}:(R,F_a)\to (R,F)$, induced by the logarithm and exponential of $(\mathbb{C},F)$. We call $\mathrm{log}_{F}:(R,F)\to (R,F_a)$ and $\mathrm{exp}_{F}:(R,F_a)\to (R,F)$ the \textit{logarithm} and \textit{exponential} of $(R,F)$, respectively.

\begin{ex}{(\text{see} \cite[Ex.~1.1.4]{LM07})}
	The \textit{additive} formal group law $(R,F_a)$ over $R$ is given by $F_a(x,y)=x+y$.
	The formal inverse of $x$ under $(R,F_a)$ is $i(x)=-x$.
	
	If $R$ is an ample ring with respect to $(\mathbb{C},F_a)$, then the logarithm of $(R,F_a)$ is $\mathrm{log}_{F_a}(x)=x$, and the exponential of $(R,F_a)$ is $\mathrm{exp}_{F_a}(x)=x$.
\end{ex}
\begin{ex}{(\text{see} \cite[Ex.~1.1.5]{LM07} and \cite[Ch. IV, \S 9, Ex.~5.1]{S09})}
	The \textit{multiplicative} formal group law $(R,F_m)$ over $R$ is given by $F_m(x,y)=x+y+xy$. 	
	The formal inverse of $x$ under $(R,F_m)$ is $i(x)=\tfrac{-x}{1- x}:=-x\sum_{i\geq0} x^i$. 
	
	If $R$ is an ample ring with respect to $(\mathbb{C},F_m)$, then the logarithm and exponential of $(R,F_m)$ are given by the formulas
	\[\mathrm{log}_{F_m}(x)=\mathrm{log}(1- x)=\sum_{i\geq 1}(-1)^{i-1}\tfrac{x^i}{i};\quad \mathrm{exp}_{F_m}(x)=\mathrm{exp}(x)=\sum_{i\geq 1}\tfrac{x^i}{i!}.\]
\end{ex}
\begin{ex}{\textup{(see} \cite[Ch. IV, \S 9, Ex.~3.4(3)]{S})}
	The \textit{Lorentz} formal group law $(R,F_l)$ over $R$ is given by 
	$F_l(x,y)=\tfrac{x+y}{1+ xy}:=(x+y)\sum_{i\geq 2}(xy)^i$.
	The formal inverse of $x$ under $(R,F_l)$ is $i(x)=-x$.
	
	If $R$ is an ample ring with respect to $(\mathbb{C},F_l)$, then the logarithm and exponential of $(R,F_l)$ are given by the formulas \[\mathrm{log}_{F_l}(x)=\mathrm{tanh^{-1}}(x)=\sum_{i\geq 1}\tfrac{x^{2i-1}}{2i-1};\quad \mathrm{exp}_{F_l}(x)=\mathrm{tanh}(x)=\sum_{i\geq 1}\tfrac{2^{2i}(2^{2i}-1)B_{2i}}{(2i)!}x^{2i-1},\]
	where $B_{2i}$ is the $2i$-th Bernoulli number.
\end{ex}
\begin{ex}{\textup{(see} \cite[\S 1.1]{LM07}\textup{)}}
	Let $\mathbb{L}$ be \textit{Lazard ring}, i.e., the commutative unital ring generated by some elements $r_{i,j}$, $i,j\geq 1$, subject only to the relations imposed by the definition of the formal group law. 
	The \textit{universal} formal group law $(\mathbb{L},F_u)$ over the Lazard ring $\mathbb{L}$ is given by $F_u(x,y)=x+y+\sum_{i,j\geq 1}r_{i,j}x^iy^j$. The formal inverse of $x$ under $(\mathbb{L},F_u)$ is a power series $i(x)=-x-c_2x^2-c_3x^3-\dotsb\in\mathbb{L}\llbracket x\rrbracket$, where the $c_i$ can be computed explicitly in terms of the $r_{i,j}$ from the relation $F_u(x,i(x))=0$.
\end{ex}

\begin{dfn}{\label{dfn:additive-type-FGL}}
	Let $(R,F)$ be a formal group law of the form 
	\[F(x,y)=(x+y)g(x,y),\quad g(x,y)\in R\llbracket x,y\rrbracket.\]
	We say that $(R,F)$ is of \textit{additive type}. In particular, the formal inverse of $x$ under $(R,F)$ is $i(x)=-x$. 
\end{dfn}

\begin{ex}
	The additive formal group law and the Lorentz formal group law are examples of formal group laws of additive type.
\end{ex}

\details{We now provide a brief overview of topological groups and rings.
\begin{itemize}
	\item We say that $G$ is a \textit{topological group} if it is simultaneously a group and a topological space, and if the group operation $+:G\times G\to G$ and the inversion map $^{-1}:G\to G$ are continuous. 
	\item We say that $S$ is a \textit{topological ring} if it is simultaneously a ring and a topological space, and if addition $+:S\times S\to S$ and multiplication $\cdot :S\times S\to S$ are continuous. Note that every topological ring is a topological group.
	\item A ring homomorphism between topological rings $f: S\to S'$ is a \textit{topological ring homomorphism} if it is continuous. If $f$ is a ring isomorphism with continous inverse, then $f$ is a \textit{topological ring isomorphism}.
	\item  Every topological group is a \textit{uniform space}. We say that a topological group $G$ is \textit{complete} if it is complete as a uniform space. We say that a topological ring $S$ is \textit{complete} if it is complete as a topological group with respect to addition.
	\item  Let $J$ be any ideal in $S$. Then its topological closure $\hat{J}$ is also an ideal in $S$.
\end{itemize} }

We now state and prove some standard facts about topological rings.

\begin{lem}{\label{lem:closure}}
	Let $S$ be a topological ring, and let $I$ be an ideal in $S$. The topological closure $\hat{I}$ of $I$ in $S$ is an ideal in $S$.
\end{lem}
\begin{proof}
	Let $x,y\in \hat{I}$ and $s\in S$. The difference $x-y\in \hat{I}$ by the proof of \cite[Ch. III, \S 2.1, Prop. 1]{B89}. To see that $sx\in\hat{I}$, let $U$ be a neighbourhood of $sx$. Since $x\in\hat{I}$, every neighbourhood $X$ of $x$ intersects $I$ nontrivially. The inverse image of $U$ under the multiplication map $(a,b)\mapsto ab$ is open in $S\times S$ by continuity of multiplication, and contains $(s,x)$. Thus, there is a neighbourhood $X$ of $x$ such that $(s\cdot X)\cap I$ is nonempty and $(s\cdot X)\subseteq U$. So $U\cap I$ is nonempty, as required.
\end{proof}
 
\begin{thm}{\label{thm:quotient-by-closed-ideal}}
	Suppose $S$ is a metrizable complete Hausdorff topological ring, and $I$ is a closed ideal in $S$. Then the quotient $S/I$ is a complete Hausdorff ring. 
\end{thm}
\begin{proof}
Since $S$ is a commutative metrizable complete Hausdorff topological group, and $I$ is a closed normal subgroup of $S$, the quotient group $S/I$ is Hausdorff by \cite[Ch. III, \S2, Prop. 18]{B89} and complete by \cite[Ch. IX, \S 3, Prop. 4]{B66}. Thus, the quotient $S/I$ is a complete Hausdorff topological ring. 
\end{proof}

\begin{rem}{\textup{(see} \cite[\S 10]{AM69}\textup{)}}
Let $S$ be a ring, and let $I$ be an ideal in $S$. The \textit{$I$-adic} topology on $S$ is the topology generated by elements of the form $x+I^n$, where $x\in S$ and $n\geq 1$. In particular, $S$ is first-countable with respect to this topology. If $S$ is Hausdorff, then Urysohn's metrization theorem implies that $S$ is metrizable. If, in addition to being Hausdorff, $S$ is complete, then \cref{thm:quotient-by-closed-ideal} implies that $S/J$ is Hausdorff and complete, where $J$ is any closed ideal in $S$. The ring $S$ is complete in the $I$-adic topology if and only if the canonical ring homomorphism $S\to  \lim\limits_{\stackrel{\longleftarrow}{i}}(S/I^i)$ is a ring isomorphism, and $S$ is Hausdorff in the $I$-adic topology if and only if $\cap_{i\geq 0}I^i=(0)$. 
\end{rem}

The following construction of the formal group ring was introduced in \cite[\S 2]{CPZ}.
Let $R$ be a commutative unital ring and $(R,F)$ a one-dimensional commutative formal group law over $R$. Let $R[x_\Lambda]$ be the polynomial ring over $R$ with variables indexed by elements of the real root lattice $\Lambda$. Here, we view $\Lambda$ as a free finitely-generated $\mathbb{Z}$-module with basis $\{e_i\alpha_j\}$. The augmentation map $\epsilon: R[x_\Lambda]\to R$ sends $x_\lambda\mapsto0$ for each $\lambda\in\Lambda$ and fixes $R$. Let $R\llbracket x_\Lambda\rrbracket$ be the $\mathrm{ker}(\epsilon)$-adic completion of the polynomial ring $R[x_\Lambda]$. Given $u,v\in R\llbracket x_\Lambda\rrbracket$, $m\in\mathbb{Z}_{\geq 0}$, we define the following notation:
\[u+_Fv=F(u,v); \quad m\cdot_Fu:=\underbrace{u+_F+_F\dotsb+_Fu}_{\text{$m$ times}};\quad (-m)\cdot_F u=-_F(m\cdot_F u),\]
where $-_F u$ is the formal inverse of $u$ under $(R,F)$.

 \details{\textcolor{red}{New definition of relations:}} Let $\widetilde{\mathcal{J}}_F$ be the closure of the ideal in $R\llbracket x_\Lambda\rrbracket$ generated by the elements 
\[x_0 \quad \text{and}\quad x_{\lambda_1+\lambda_2}-(x_{\lambda_1}+_Fx_{\lambda_2}), \]
over all $\lambda_1,\lambda_2\in\Lambda$.

\begin{dfn}
The quotient
\[\widetilde{R\llbracket\Lambda\rrbracket}_F:=R\llbracket x_\Lambda\rrbracket/\widetilde{\mathcal{J}}_F\]
is the formal group ring of \cite[Def. 2.4]{CPZ}. For the remainder of this paper, we call $\widetilde{R\llbracket\Lambda\rrbracket}_F$ the \textit{classical formal group ring}. 
\end{dfn}

\begin{rem}
	If $\Sigma$ is crystallographic, then $\Lambda$ is a crystallographic root lattice, and $\widetilde{R\llbracket\Lambda\rrbracket}_F$ is a special case of the \textit{formal group ring} studied in \cite{CZZ1}.
\end{rem}

We denote the image of $x_\lambda$ in $\widetilde{R\llbracket \Lambda\rrbracket}_F$ by the same symbol. 
We will now review several facts about the classical formal group ring.

The ring $R\llbracket x_\Lambda\rrbracket$ is a complete Hausdorff ring with respect to the $\mathrm{ker}(\epsilon)$-adic topology, and $\widetilde{\mathcal{J}}_F$ is contained in $\mathrm{ker}(\epsilon)$. Thus, \cref{lem:closure} and \cref{thm:quotient-by-closed-ideal} imply that the classical formal group ring $\widetilde{R\llbracket\Lambda\rrbracket}_F$ is a complete Hausdorff ring with respect to the $\widetilde{\mathcal{I}}_F$-adic topology, where $\widetilde{\mathcal{I}}_F$ is the kernel of the induced augmentation map $\widetilde{R\llbracket\Lambda\rrbracket}_F\to R$.
By \cite[Cor. 2.13]{CPZ}, there is a continuous ring isomorphism $\varrho\colon  \widetilde{R\llbracket\Lambda\rrbracket}_F\to R\llbracket x_{1},\dotsc,x_{nl}\rrbracket$ such that
\begin{align}\label{equation:iso}
\varrho\left(x_{\sum_{i=1}^l\sum_{j=1}^nm_{i,j}e_i \alpha_j}\right) =&
(m_{1,1}\cdot_F x_1+_F\dotsb+_F m_{1,n} \cdot_F x_n)+_F\dotsb\\&\nonumber\dotsb+_F(m_{l,1}\cdot_F x_{n(l-1)+1}+_F\dotsb+_F m_{l,n}\cdot_F x_{nl}).
\end{align}
In particular, the isomorphism sends $x_{e_i\alpha_j}\mapsto x_{(i-1)n+j}$. If $R$ is an integral domain, then so is $\widetilde{R\llbracket\Lambda\rrbracket}_F$.

Since $W$ acts on $\Lambda$, it follows from \cite[Lem. 2.6]{CPZ} that $W$ also acts on the classical formal group ring by 
\begin{equation}\label{eqn:w-action}
	w(x_\lambda)=x_{w(\lambda)},\quad w\in W, \text{ }\lambda\in\Lambda.
\end{equation} 
In fact, since each $w\in W$ sends $\widetilde{\mathcal{I}}_F$ to itself, the action of $w\in W$ on $\widetilde{R\llbracket\Lambda\rrbracket}_F$ is a \textit{continuous} ring homomorphism.

Let $(R,F')$ be another formal group law over $R$. 
By \cite[Lem. 2.6]{CPZ}, any morphism $f\colon(R,F)\to (R,F')$ of formal group laws over $R$ induces a continuous ring homomorphism $f^*\colon\widetilde{R\llbracket\Lambda\rrbracket}_{F'}\to \widetilde{R\llbracket\Lambda\rrbracket}_F$ sending $x_\lambda$ to $f(x_\lambda)$. To see that $f^*$ is well-defined, note that
\begin{align*}
f^*(x_{\lambda+\mu})=f(x_{\lambda+\mu})&=f(x_\lambda+_{F'}x_\mu)=f(x_\lambda)+_Ff(x_\mu)\\&=f^*(x_\lambda)+_Ff^*(x_\mu)=f^*(x_\lambda+_Fx_\mu)=f^*(x_{\lambda+\mu}),
\end{align*}
for all $\lambda$, $\mu\in\Lambda$.
In addition, $f^*$ is $W$-equivariant: write $f(u)=\sum_{i\geq 1}a_{i}u^i$ for some $a_i\in R$. Given $w\in W$ and $\lambda\in\Lambda$, we have 
\[f^*(w\cdot x_\lambda)=f^*(x_{w(\lambda)})=\sum_{i\geq 1}a_i x_{w(\lambda)}^i= \sum_{i\geq 1}w\cdot (a_ix_{\lambda}^i)=w\cdot f^*(x_\lambda).\]
If $(\mathbb{C},F)$ is a formal group law over $\mathbb{C}$, and $R$ is an ample ring with respect to $(\mathbb{C},F)$, then the exponential of the formal group law $(R,F)$ induces a $W$-equivariant continuous ring isomorphism \[\mathrm{exp}_F^*\colon \widetilde{R\llbracket\Lambda\rrbracket}_F\to\widetilde{R\llbracket\Lambda\rrbracket}_{F_a},\quad  {\mathrm{exp}}_F^*(x_\lambda)=\mathrm{exp}_F(x_{\lambda}),\quad \lambda\in\Lambda.\] This is summarized by the following commutative diagram, which commutes for any $w\in W$: 

\begin{equation}\label{eqn:W-equivariance}
\begin{tikzcd}
\widetilde{R\llbracket\Lambda\rrbracket}_F \arrow[r,"\mathrm{exp}_F^*"] \arrow[d,"w"'] & \widetilde{R\llbracket\Lambda\rrbracket}_{F_a} \arrow[d,"w"] \\
\widetilde{R\llbracket\Lambda\rrbracket}_F \arrow[r,"\mathrm{exp}_F^*"] & \widetilde{R\llbracket\Lambda\rrbracket}_{F_a}.
\end{tikzcd}
\end{equation}

\begin{rem}\label{rem:explanation}
When $\Sigma$ is crystallographic, it is shown in \cite[Cor. 3.4]{CPZ} that $x_{\alpha}$ divides the element $u-s_\alpha(u)$ for all $u\in \widetilde{R\llbracket\Lambda\rrbracket}_F$ and $\alpha\in\Sigma$. Therefore, one can define a \textit{formal Demazure operator} $\Delta_\alpha$ on $\widetilde{R\llbracket\Lambda\rrbracket}_F$ by the formula
	\[\Delta_\alpha(u)=\tfrac{u-s_\alpha(u)}{x_\alpha},\quad u\in \widetilde{R\llbracket\Lambda\rrbracket}_F.\]
	However, when $\Sigma$ is noncrystallographic, it is not necessarily true that $x_{\alpha}$ divides the element $u-s_\alpha(u)$ for all $u\in \widetilde{R\llbracket\Lambda\rrbracket}_F$ and $\alpha\in\Sigma$. We provide an example to demonstrate this. Suppose $(R,F_a)$ is the additive formal group law over a commutative unital ring $R$, the reflection group $W=I_2(5)$, and $(\Sigma,\Delta)$ is the geometric realization of $I_2(5)$ given in \cref{ex:I2(5)}. Then, in $\widetilde{R\llbracket\Lambda\rrbracket}_{F_a}$, we have
	\begin{align*}
	x_{\alpha_1}-s_{\alpha_2}(x_{\alpha_1})=x_{\alpha_1}-x_{\alpha_1+\tau\alpha_2}&=x_{\alpha_1}-(x_{\alpha_1}+x_{\tau\alpha_2})=-x_{\tau\alpha_2}.
	\end{align*}
	If $x_{\alpha_2}$ divides $x_{\alpha_1}-s_{\alpha_2}(x_{\alpha_1})=-x_{\tau\alpha_2}$ in $\widetilde{R\llbracket\Lambda\rrbracket}_{F_a}$, then the isomorphism ${\widetilde{R\llbracket\Lambda\rrbracket}_{F_a}}\simeq R\llbracket x_1,\dotsc,x_{4}\rrbracket$ of \cref{equation:iso} implies that $x_2$ divides $-x_4$ in $R\llbracket x_1,\dotsc,x_{4}\rrbracket$, which is a contradiction. 
	
	The focus of the next two sections is to resolve this issue. Let $W$ be any real finite reflection group, and let $R$ be subring of $\mathbb{C}$ satisfying the hypotheses of \cref{assumption-ring}. In \cref{section:formal-group-algebra}, we show that there is a closed ideal $\mathcal{J}_F$ of $R\llbracket x_\Lambda\rrbracket$, such that the quotient $R\llbracket \Lambda\rrbracket_F:=R\llbracket x_\Lambda\rrbracket/\mathcal{J}_F$ is an integral domain, and such that there is a well-defined $W$-action on $R\llbracket \Lambda\rrbracket_F$, and such that $x_{\alpha_j}$ divides $x_{e_i \alpha_j}$ in $R\llbracket \Lambda\rrbracket_F$ for all $e_i\in B$ and $\alpha_j\in\Delta$ (see \cref{lem:same-properties}). Here, we are denoting the image of $x_\lambda$ in $R\llbracket \Lambda\rrbracket_F$ by the same symbol. In \cref{section:localization}, we show that $x_{\alpha}$ divides $u-s_{\alpha}(u)$ in $R\llbracket \Lambda\rrbracket_F$ for all $u\in R\llbracket \Lambda\rrbracket_F$ and $\alpha\in\Sigma$, using the fact that $x_{\alpha_j}$ divides $x_{e_i \alpha_j}$ for all $e_i\in B$ and $\alpha_j\in\Delta$. (see \cref{lem:divisible}). This allows us to define the formal Demazure operator on $R\llbracket \Lambda\rrbracket_F$ (see \cref{dfn:formal-Demazure-operator}). We call $R\llbracket \Lambda\rrbracket_F$ the \textit{formal group ring} with respect to $(R,F)$ and $\Lambda$. When $\Sigma$ is crystallographic, $R\llbracket \Lambda\rrbracket_F$ collapses to classical formal group ring $\widetilde{R\llbracket\Lambda\rrbracket}_F$ with respect to the crystallographic root lattice $\Lambda$.
\end{rem}

\section{Construction of the formal group ring}\label{section:formal-group-algebra}

The purpose of this section is to construct the formal group ring $R\llbracket \Lambda\rrbracket_F$ for all real finite reflection groups. This construction will allow us to define formal Demazure operators on $R\llbracket \Lambda\rrbracket_F$ in \cref{section:localization}. See \cref{rem:explanation} for a detailed motivation for this construction.

From here up to \cref{assumption-ring}, we assume that $\Sigma$ is a noncrystallographic root system. Let $(\mathbb{C},F)$ be a formal group law over $\mathbb{C}$, and let $R$ be an ample ring with respect to $(\mathbb{C},F)$. Furthermore, assume that $R$ contains $\mathcal{R}$. The exponential $\mathrm{exp}_F$ (resp. logarithm $\mathrm{log}_F$) of $(R,F)$ induces a $W$-equivariant continuous ring isomorphism $\mathrm{exp}_F^*:\widetilde{R\llbracket\Lambda\rrbracket}_{F}\to \widetilde{R\llbracket\Lambda\rrbracket}_{F_a}$ (resp. $\mathrm{log}_F^*:\widetilde{R\llbracket\Lambda\rrbracket}_{F_a}\to \widetilde{R\llbracket\Lambda\rrbracket}_{F}$). The isomorphisms $\mathrm{exp}_F^*$ and $\mathrm{log}_F^*$ are inverse to each other. Since $R$ contains $\mathcal{R}$, we let ${\mathcal{J}}^\star_F$ be the closure of the ideal in $\widetilde{R\llbracket \Lambda\rrbracket}_F$ generated by the elements 
\[e_i\mathrm{log}_F(x_{\alpha_j})-\mathrm{log}_F(x_{e_i\alpha_j}),\]
over all $e_i\in B$, and $\alpha_j\in\Delta$. 

The ideal ${\mathcal{J}}^\star_F$ is contained in the kernel of the augmentation map $\widetilde{R\llbracket \Lambda\rrbracket}_F\to R$, and, hence, the map $\widetilde{R\llbracket \Lambda\rrbracket}_F\to R$ factors through the quotient $\widetilde{R\llbracket \Lambda\rrbracket}_F/{\mathcal{J}}^\star_F$. Therefore, \cref{lem:closure} and \cref{thm:quotient-by-closed-ideal} imply that $\widetilde{R\llbracket \Lambda\rrbracket}_F/{\mathcal{J}}^\star_F$ is a complete Hausdorff ring with respect to the ${\mathcal{I}}^\star_F$-adic topology, where ${\mathcal{I}}^\star_F$ is the kernel of the induced augmentation map $\widetilde{R\llbracket \Lambda\rrbracket}_F/{\mathcal{J}}^\star_F\to R$. We will denote the image of $x_\lambda$ in the quotient  $\widetilde{R\llbracket \Lambda\rrbracket}_F/{\mathcal{J}}^\star_F$ by the same symbol.

\begin{lem}\label{lem:W-invariant-ideal}
	The ideal ${\mathcal{J}}^\star_F$ in ${R\llbracket\Lambda\rrbracket}_{F}$ is $W$-invariant.
\end{lem}
\begin{proof}
Under the continuous $W$-equivariant ring isomorphism $\mathrm{exp}_F^*\colon \widetilde{R\llbracket\Lambda\rrbracket}_{F}\to \widetilde{R\llbracket\Lambda\rrbracket}_{F_a}$, we have $\mathrm{exp}_F^*(\widetilde{\mathcal{J}}_F)=\widetilde{\mathcal{J}}_{F_a}$. To see this, we first note that, by continuity, $\mathrm{exp}_F^*(\widetilde{\mathcal{J}}_F)$ is contained in the closure $\widetilde{\mathcal{J}}_{F_a}$. Since $\mathrm{exp}_F^*$ is an isomorphism of topological rings, $\mathrm{exp}_F^*(\widetilde{\mathcal{J}}_F)$ is closed in $\widetilde{R\llbracket\Lambda\rrbracket}_{F_a}$, so it must equal $\widetilde{\mathcal{J}}_{F_a}$. 

Since $\mathrm{exp}_F^*$ is $W$-equivariant, we see that $\widetilde{\mathcal{J}}_{F}$ is $W$-invariant in $\widetilde{R\llbracket\Lambda\rrbracket}_{F}$ if and only if $\widetilde{\mathcal{J}}_{F_a}$ is $W$-invariant in $\widetilde{R\llbracket\Lambda\rrbracket}_{F_a}$. Thus, we will show that $\widetilde{\mathcal{J}}_{F_a}$ is $W$-invariant in $\widetilde{R\llbracket\Lambda\rrbracket}_{F_a}$. Let $\mathcal{J}$ be the ideal in $\widetilde{R\llbracket\Lambda\rrbracket}_{F_a}$ generated by the elements 
$e_ix_{\alpha_j}-x_{e_i\alpha_j},$
over all $e_i\in B$ and $\alpha_j\in\Delta$. The closure of $\mathcal{J}$ in $\widetilde{R\llbracket\Lambda\rrbracket}_{F_a}$ is ${\mathcal{J}}^\star_{F_a}$. If $\mathcal{J}$ is $W$-invariant, then so is the closure ${\mathcal{J}}^\star_{F_a}$ by continuity of the action of $w\in W$ on $\widetilde{R\llbracket\Lambda\rrbracket}_{F_a}$. In other words, suppose $\mathcal{J}$ is $W$-invariant. If $x\in{\mathcal{J}}^\star_{F_a}$ and $(x_i)_{i\geq 0}$ is a sequence of elements in $\mathcal{J}$ converging to $x$, then the sequence $(w(x_i))_{i\geq 0}$ in $\mathcal{J}$ converges to $w(x)$. So $w(x)$ lies in the closure ${\mathcal{J}}^\star_{F_a}$ as well. Thus, it is enough to show that $\mathcal{J}$ is $W$-invariant in $\widetilde{R\llbracket\Lambda\rrbracket}_{F_a}$. To do this, we will show that the image of $w(e_ix_{\alpha_j}-x_{e_i\alpha_j})$ in the quotient $\widetilde{R\llbracket\Lambda\rrbracket}_{F_a}/\mathcal{J}$ is zero for all $e_i\in B$, $\alpha_j\in \Delta$, and $w\in W$. Here, we are denoting the image of $x_\lambda$ in the quotient $\widetilde{R\llbracket\Lambda\rrbracket}_{F_a}/\mathcal{J}$ by the same symbol.

Fix $w\in W$. Since $w(\alpha_j)\in\Lambda$, we write $w(\alpha_j)=\sum_{k=1}^l\sum_{r=1}^nc_{k,r}e_k\alpha_r$ for some $c_{k,r}\in\mathbb{Z}$. Furthermore, we can write $e_ie_k=\sum_{t=1}^ld_t^{(i,k)}e_t$ for some $d_t^{(i,k)}\in \mathbb{Z}$, since $\{e_1,\dotsc,e_l\}$ is a $\mathbb{Z}$-basis of $\mathcal{R}$. Thus, in the quotient $\widetilde{R\llbracket\Lambda\rrbracket}_{F_a}/{\mathcal{J}}$, we have
\[w(x_{e_i\alpha_j})=x_{e_iw(\alpha_j)}=x_{e_i\sum_{k=1}^l\sum_{r=1}^nc_{k,r}e_k\alpha_r}=x_{\sum_{k,t=1}^l\sum_{r=1}^nc_{k,r}d_t^{(i,k)}e_t\alpha_r}.
\]
Since $(R,F_a)$ is the additive formal group law, we have \[x_{\sum_{k,t=1}^l\sum_{r=1}^nc_{k,r}d_t^{(i,k)}e_t\alpha_r}=\sum_{k,t=1}^l\sum_{r=1}^nc_{k,r}d_{t}^{(i,k)}x_{e_t\alpha_r}.\]
Finally, since $e_ix_{\alpha_j}=x_{e_i\alpha_j}$ in $\widetilde{R\llbracket\Lambda\rrbracket}_{F_a}/{\mathcal{J}}$, we have
\begin{align*}
w(x_{e_i\alpha_j})=\sum_{k,t=1}^l\sum_{r=1}^nc_{k,r}d_{t}^{(i,k)}e_tx_{\alpha_r}=e_i\sum_{k=1}^l\sum_{r=1}^nc_{k,r}e_kx_{\alpha_r}&=e_ix_{\sum_{k=1}^l\sum_{r=1}^nc_{k,r}e_k\alpha_r}&=e_ix_{w(\alpha_j)}=w(e_ix_{\alpha_j}).
\end{align*}
Therefore, ${\mathcal{J}}$ is a $W$-invariant ideal in $\widetilde{R\llbracket\Lambda\rrbracket}_{F_a}$.
\end{proof}

\begin{rem}{\label{rem:well-defined}}
	\cref{lem:W-invariant-ideal} implies that there is a well-defined $W$-action on the quotient $\widetilde{R\llbracket\Lambda\rrbracket}_F/{\mathcal{J}}^\star_{F}$, given by 
	\[w(x_\lambda)=x_{w(\lambda)},\quad w\in W,\textcolor{white}{..} \lambda\in\Lambda. \]
\end{rem}

In the proof of \cref{lem:W-invariant-ideal}, we showed that the continuous $W$-equivariant ring isomorphism $\mathrm{exp}_F^*\colon\widetilde{R\llbracket\Lambda\rrbracket}_F\to \widetilde{R\llbracket\Lambda\rrbracket}_{F_a}$ exchanges the $W$-invariant ideals ${\mathcal{J}}^\star_{F}$ and ${\mathcal{J}}^\star_{F_a}$. Thus, $\mathrm{exp}_F^*$ induces a continuous $W$-equivariant ring isomorphism on the quotients,
\[\widetilde{\mathrm{exp}}_F^*:\widetilde{R\llbracket\Lambda\rrbracket}_F/{\mathcal{J}}^\star_{F}\to \widetilde{R\llbracket\Lambda\rrbracket}_{F_a}/{\mathcal{J}}^\star_{F_a},\quad  \widetilde{\mathrm{exp}}_F^*(x_{\lambda})=\mathrm{exp}_F(x_{\lambda}),\quad \lambda\in\Lambda.\]
This is summarized by the following diagram, which commutes for all $w\in W$:
\begin{equation}\label{eqn:W-equivariance-of-quotient}
\begin{tikzcd}
\widetilde{R\llbracket\Lambda\rrbracket}_F/{\mathcal{J}}^\star_{F} \arrow[r,"\widetilde{exp}_{F}^*"] \arrow[d,"w"'] & \widetilde{R\llbracket\Lambda\rrbracket}_{F_a}/{\mathcal{J}}^\star_{F_a} \arrow[d,"w"] \\
\widetilde{R\llbracket\Lambda\rrbracket}_F/{\mathcal{J}}^\star_{F} \arrow[r,"\widetilde{exp}_{F}^*"] & \widetilde{R\llbracket\Lambda\rrbracket}_{F_a}/{\mathcal{J}}^\star_{F_a}.
\end{tikzcd}
\end{equation}
\begin{lem}{\label{lem:integral-isomorphism}}
	There is a continuous ring isomorphism \[\widetilde{R\llbracket\Lambda\rrbracket}_{F}/{\mathcal{J}}^\star_{F}\to  R\llbracket x_1,\dotsc,x_n\rrbracket,\]
	sending $x_{e_i\alpha_j}\mapsto \mathrm{exp}_F(e_ix_{j})$. In particular, $\widetilde{R\llbracket\Lambda\rrbracket}_{F}/{\mathcal{J}}^\star_{F}$ is an integral domain.
\end{lem}
\begin{proof}
There is a continuous ring isomorphism \[\widetilde{\mathrm{exp}}_F^*:\widetilde{R\llbracket\Lambda\rrbracket}_F/{\mathcal{J}}^\star_{F}\to \widetilde{R\llbracket\Lambda\rrbracket}_{F_a}/{\mathcal{J}}^\star_{F_a},\quad \widetilde{\mathrm{exp}}_F^*(x_{e_i\alpha_j})=\mathrm{exp}_F(x_{e_i\alpha_j}).\] Suppose there is a continuous ring isomorphism \[\gamma:\widetilde{R\llbracket\Lambda\rrbracket}_{F_a}/{\mathcal{J}}^\star_{F_a}\to R\llbracket x_1,\dotsc,x_n\rrbracket,\quad \gamma(x_{e_i\alpha_j})=e_ix_j.\]
 Then the composition 
 \[\gamma\circ \widetilde{\mathrm{exp}}_F^*:\widetilde{R\llbracket\Lambda\rrbracket}_{F}/{\mathcal{J}}^\star_{F}\to  R\llbracket x_1,\dotsc,x_n\rrbracket\]
  is a continuous ring isomorphism such that $(\gamma\circ \widetilde{\mathrm{exp}}_F^*)(x_{e_i\alpha_j})=\mathrm{exp}_F(e_ix_j)$, which proves the Lemma. Therefore, we will show the existence of the isomorphism $\gamma$.

Recall that ${\mathcal{J}}^\star_{F_a}$ is the closure of the ideal in $\widetilde{R\llbracket\Lambda\rrbracket}_{F_a}$ generated by the elements $x_{e_i\alpha_j}-e_ix_{\alpha_j}$,
over all $e_i\in B$ and $\alpha_j\in\Delta$.
In addition, recall that there is a continuous ring isomorphism 
\[\varrho\colon \widetilde{R\llbracket\Lambda\rrbracket}_{F_a}\to R\llbracket x_1,\dotsc,x_{nl}\rrbracket,\quad \varrho(x_{me_i\alpha_j})=mx_{(i-1)n+j},\] for all $e_i\in B$, $\alpha_j\in\Delta$, and $m\in\mathbb{Z}$ (see \cref{equation:iso}). The image $\varrho({\mathcal{J}}^\star_{F_a})$ is the closure of the ideal in $R\llbracket x_1,\dotsc x_{nl}\rrbracket$ generated by the elements 
$x_{(i-1)n+j}-e_ix_j$,
over all $i=1,\dotsc,l$, and $j=1,\dotsc,n$. By \cref{lem:closure} and \cref{thm:quotient-by-closed-ideal}, the quotient $R\llbracket x_1,\dotsc,x_{nl}\rrbracket/\varrho({\mathcal{J}}^\star_{F_a})$ is a complete Hausdorff ring with respect to the topology induced by the kernel of the augmentation map $R\llbracket x_1,\dotsc,x_{nl}\rrbracket/\varrho({\mathcal{J}}^\star_{F_a})\to R$. Thus, there is a continuous ring isomorphism on the quotients,
\[\tilde{\varrho}\colon \widetilde{R\llbracket\Lambda\rrbracket}_{F_a}/{\mathcal{J}}^\star_{F_a}\to R\llbracket x_1,\dotsc,x_{nl}\rrbracket/\varrho({\mathcal{J}}^\star_{F_a}),\quad \tilde{\varrho}(x_{me_i\alpha_j})=mx_{(i-1)n+j}.\]
Here we denote the image of $x_r$ in the quotient $ R\llbracket x_1,\dotsc,x_{nl}\rrbracket/\varrho({\mathcal{J}}^\star_{F_a})$ by the same symbol. 

Suppose there is a continuous ring isomorphism 
\[{\kappa}: R\llbracket x_1,\dotsc,x_{nl}\rrbracket/\varrho({\mathcal{J}}^\star_{F_a})\to R\llbracket x_1,\dotsc,x_n\rrbracket,\quad {\kappa}(x_{(i-1)n+j})=e_ix_j.\] Then the composition \[{\kappa}\circ \tilde{\varrho}: \widetilde{R\llbracket\Lambda\rrbracket}_{F_a}/{\mathcal{J}}^\star_{F_a}\to R\llbracket x_1,\dotsc,x_n\rrbracket\] is a continuous ring isomorphism such that $({\kappa}\circ \tilde{\varrho})(x_{e_i\alpha_j})=e_ix_j$. In particular, we can take $\gamma:=\kappa\circ \tilde{\varrho}$. Therefore, all we need to do is show the existence of the isomorphism ${\kappa}$.

	The ring homomorphism 
	\[\bar{\kappa}\colon R[ x_1,\dotsc,x_{nl}]\to  R\llbracket x_1,\dotsc,x_{n}\rrbracket,\quad  \bar{\kappa}(x_{(i-1)n+j})= e_ix_j\] extends to a continuous ring homomorphism
	\[\tilde{\kappa}\colon R\llbracket x_1,\dotsc,x_{nl}\rrbracket\to  R\llbracket x_1,\dotsc,x_{n}\rrbracket,\quad   \tilde{\kappa}(x_{(i-1)n+j})= e_ix_j.\] Since $e_1=1$, we have \[\tilde{\kappa}(x_{(i-1)n+j}-e_ix_j)=\tilde{\kappa}(x_{(i-1)n+j}-e_ix_{(1-1)n+j})=e_ix_j-e_ix_j=0.\]
	Therefore, $\tilde{\kappa}$ induces a continuous ring homomorphism  \[{\kappa}\colon R\llbracket x_1,\dotsc,x_{nl}\rrbracket/\varrho({\mathcal{J}}^\star_{F_a})\to R\llbracket x_1,\dotsc,x_{n}\rrbracket,\quad {\kappa}(x_{(i-1)n+j})= e_ix_j.\] 
	
	In the other direction, the ring homomorphism 
	\[\tilde{\iota}\colon R[x_1,\dotsc,x_n]\to R\llbracket x_1,\dotsc,x_{nl}\rrbracket/\varrho({\mathcal{J}}^\star_{F_a}),\quad \tilde{\iota}(x_j)= x_j\] extends to a continuous ring homomorphism
	\[\iota\colon R\llbracket x_1,\dotsc,x_{n}\rrbracket\to R\llbracket x_1,\dotsc,x_{nl}\rrbracket/\varrho({\mathcal{J}}^\star_{F_a}),\quad {\iota}(x_j)= x_j.\] We see that $({\kappa}\circ \iota)(x_j)={\kappa}(x_j)=x_j$ and that $(\iota\circ {\kappa})(x_{(i-1)n+j})=\iota(e_ix_j)=e_ix_j=x_{(i-1)n+j}$. Since the $x_j$, $j=1,\dotsc,n$, generate $R\llbracket x_1,\dotsc,x_n\rrbracket$ as a ring, and the $x_k$, $k=1,\dotsc,nl$, generate $R\llbracket x_1,\dotsc,x_{nl}\rrbracket/\varrho({\mathcal{J}}^\star_{F_a})$ as a ring, it follows from the continuity of $\iota$ and ${\kappa}$ that $\iota$ and ${\kappa}$ are continuous ring isomorphisms that are inverse to each other. This proves the Lemma.
\end{proof}

\begin{lem}\label{lem:mixed-divisible}
	The element $x_{\alpha_j}$ divides $x_{e_i\alpha_j}$ in $\widetilde{R\llbracket\Lambda\rrbracket}_{F}/{\mathcal{J}}^\star_{F}$ for any $\alpha_j\in \Delta$ and $e_i\in B$. 
\end{lem}
\begin{proof}
	Recall that $\mathrm{log}_{F}$ satisfies
	\[\mathrm{log}_{F}(u)=u+\sum_{k\geq 2}b_ku^k =u\left(1+\sum_{k\geq2}b_ku^{k-1}\right)\]
	for some $b_k\in R$. The series $g(u):=(1+\sum_{k\geq2}b_ku^{k-1})$ is invertible in $R\llbracket u\rrbracket$ since its constant term is 1. By definition, the relation 
	\[e_i\mathrm{log}_{F}(x_{\alpha_j})=\mathrm{log}_{F}(x_{e_i\alpha_j})\]
	holds in $\widetilde{R\llbracket\Lambda\rrbracket}_{F}/{\mathcal{J}}^\star_{F}$. Therefore, 
	\begin{align*}
	&e_i x_{\alpha_j}+e_i\sum_{k\geq 2}b_kx_{\alpha_j}^k=x_{e_i\alpha_j}+\sum_{k\geq 2}b_kx_{e_i\alpha_j}^k
	\\&\implies (e_i x_{\alpha_j})g( x_{\alpha_j})=x_{e_i\alpha_j}g( x_{e_i\alpha_j})
	\implies x_{\alpha_j} (e_i g( x_{\alpha_j}) g(x_{e_i\alpha_j})^{-1})=x_{e_i\alpha_j}.
	\end{align*}
	Thus, $x_{\alpha_j}$ divides $x_{e_i\alpha_j}$ in $\widetilde{R\llbracket\Lambda\rrbracket}_{F}/{\mathcal{J}}^\star_{F}$.
\end{proof}

Let $\mathcal{J}_F^\circ$ be the closure of the ideal in $R\llbracket x_\Lambda\rrbracket$ generated by the elements 
\[x_0 \quad \text{and}\quad x_{\lambda_1+\lambda_2}-(x_{\lambda_1}+_Fx_{\lambda_2}) \quad \text{and}\quad e_i\mathrm{log}_F(x_{\alpha_j})-\mathrm{log}_F(x_{e_i\alpha_j}),\]
over all $e_i\in B$, $\alpha_j\in\Delta$, and $\lambda_1,\lambda_2\in\Lambda$. By similar reasoning as earlier, the quotient $R\llbracket x_\Lambda\rrbracket/\mathcal{J}_F^\circ$ is a complete Hausdorff ring with respect to the $\mathcal{I}_{F}^\circ$-adic topology, where $\mathcal{I}_{F}^\circ$ is the kernel of the induced augmentation map $R\llbracket x_\Lambda\rrbracket/\mathcal{J}_F^\circ\to R$.

\begin{lem}{\label{lem:second-perspective}}
	There is a canonical $W$-equivariant continuous ring isomorphism 
	\[R\llbracket x_\Lambda\rrbracket/\mathcal{J}_F^\circ\to  \left(R\llbracket x_\Lambda\rrbracket/\widetilde{{\mathcal{J}}}_F\right)/{\mathcal{J}}_F^\star=\widetilde{R\llbracket \Lambda\rrbracket}_{F}/{\mathcal{J}}_F^\star,\]
	sending $x_\lambda+\mathcal{J}_F^\circ\mapsto \left(x_\lambda+\widetilde{{\mathcal{J}}}_F\right)+{\mathcal{J}}_F^\star$ for all $\lambda\in\Lambda$. 
\end{lem}
\begin{proof}
	The ring homomorphism 
	\[\tilde{\phi}: R[x_\Lambda]\to R\llbracket x_\Lambda\rrbracket/{{\mathcal{J}}}_F^\circ,\quad   \tilde{\phi}(x_\lambda)=x_\lambda+{{\mathcal{J}}}_F^\circ,\]
	extends to a continuous ring homomorphism \[\hat{\phi}:R\llbracket x_\Lambda\rrbracket \to R\llbracket x_\Lambda\rrbracket/{{\mathcal{J}}}_F^\circ,\quad \hat{\phi}(x_\lambda)=x_\lambda+{{\mathcal{J}}}_F^\circ.\] By continuity of $\hat{\phi}$ and Hausdorff completeness of $R\llbracket x_\Lambda\rrbracket$ and $ R\llbracket x_\Lambda\rrbracket/{{\mathcal{J}}}_F^\circ$, we see that 
	\begin{align*}
	\hat{\phi}(x_{\lambda}+_Fx_\mu)=\hat{\phi}(x_\lambda)+_F\hat{\phi}(x_\mu)&=\left(x_\lambda+{{\mathcal{J}}}_F^\circ\right)+_F\left(x_\mu+{{\mathcal{J}}}_F^\circ\right)\\&=(x_{\lambda}+_Fx_{\mu})+{{\mathcal{J}}}_F^\circ=x_{\lambda+\mu}+{{\mathcal{J}}}_F^\circ=\hat{\phi}(x_{\lambda+\mu}).
	\end{align*}
	Thus, $\hat{\phi}$ induces a continuous ring homomorphism \[\bar{\phi}:R\llbracket x_\Lambda\rrbracket/\widetilde{{\mathcal{J}}}_F \to R\llbracket x_\Lambda\rrbracket/{{\mathcal{J}}}_F^\circ,\quad \bar{\phi}\left(x_\lambda+\widetilde{{\mathcal{J}}}_F\right)=x_\lambda+{{\mathcal{J}}}_F^\circ.\] 
	By continuity of $\bar{\phi}$ and Hausdorff completeness of $R\llbracket x_\Lambda\rrbracket/\widetilde{{\mathcal{J}}}_F$ and $R\llbracket x_\Lambda\rrbracket/{{\mathcal{J}}}_F^\circ$, we see that 
	\begin{align*}
	\bar{\phi}&\left(e_i\mathrm{log}_F(x_{\alpha_j})+\widetilde{{\mathcal{J}}}_F\right)=\bar{\phi}\left(e_i\mathrm{log}_F\left(x_{\alpha_j}+\widetilde{{\mathcal{J}}}_F\right)\right)=e_i\mathrm{log}_F\left(\bar{\phi}\left( x_{\alpha_j}+\widetilde{{\mathcal{J}}}_F\right)\right)\\&=e_i\mathrm{log}_F\left( x_{\alpha_j}+{\mathcal{J}}^\circ_F\right)=e_i\mathrm{log}_F(x_{\alpha_j})+{{\mathcal{J}}}_F^\circ=\mathrm{log}_F(x_{e_i\alpha_j})+{{\mathcal{J}}}_F^\circ=\mathrm{log}_F\left(x_{e_i\alpha_j}+{{\mathcal{J}}}_F^\circ\right)\\&=\mathrm{log}_F\left(\bar{\phi}\left(x_{e_i\alpha_j}+\widetilde{{\mathcal{J}}}_F\right)\right)=\bar{\phi}\left(\mathrm{log}_F\left(x_{e_i\alpha_j}+\widetilde{{\mathcal{J}}}_F\right)\right)=\bar{\phi}\left(\mathrm{log}_F(x_{e_i\alpha_j})+\widetilde{{\mathcal{J}}}_F\right).
	\end{align*}
	Thus, $\bar{\phi}$ induces a continuous ring homomorphism \[{\phi}\colon \left(R\llbracket x_\Lambda\rrbracket/\widetilde{{\mathcal{J}}}_F\right)/\mathcal{J}_F^\star \to R\llbracket x_\Lambda\rrbracket/{{\mathcal{J}}}_F^\circ,\quad \phi\left(\left(x_\lambda+\widetilde{{\mathcal{J}}}_F\right) +{\mathcal{J}}_F^\star\right)=x_\lambda+{{\mathcal{J}}}_F^\circ.\]
	
In the other direction, the ring homomorphism 
\[\tilde{\psi}: R[x_\Lambda]\to \left(R\llbracket x_\Lambda\rrbracket/\widetilde{{\mathcal{J}}}_F\right)/{\mathcal{J}}_F^\star,\quad   \tilde{\psi}(x_\lambda)=\left(x_\lambda+\widetilde{{\mathcal{J}}}_F\right)+{\mathcal{J}}_F^\star,\] extends to a continuous ring homomorphism \[\hat{\psi}: R\llbracket x_\Lambda\rrbracket\to \left(R\llbracket x_\Lambda\rrbracket/\widetilde{{\mathcal{J}}}_F\right)/{\mathcal{J}}_F^\star,\quad \hat{\psi}(x_\lambda)=\left(x_\lambda+\widetilde{{\mathcal{J}}}_F\right)+{\mathcal{J}}_F^\star\] By continuity of $\hat{\psi}$ and Hausdorff completeness of $R\llbracket x_\Lambda\rrbracket/\widetilde{{\mathcal{J}}}_F$ and $\left(R\llbracket x_\Lambda\rrbracket/\widetilde{{\mathcal{J}}}_F\right)/{\mathcal{J}}_F^\star$ and $R\llbracket x_\Lambda\rrbracket$, we see that\begin{align*}
\hat{\psi}(x_{\lambda}+_Fx_{\mu})&=\hat{\psi}(x_{\lambda})+_F\hat{\psi}(x_{\mu})=\left(\left(x_\lambda+\widetilde{{\mathcal{J}}}_F\right)+{\mathcal{J}}_F^\star\right)+_F\left(\left(x_\mu+\widetilde{{\mathcal{J}}}_F\right)+{\mathcal{J}}_F^\star\right)\\&=\left((x_\lambda+_F x_\mu)+\widetilde{{\mathcal{J}}}_F\right)+{\mathcal{J}}_F^\star=\left(x_{\lambda+\mu}+\widetilde{{\mathcal{J}}}_F\right)+{\mathcal{J}}_F^\star=\hat{\psi}(x_{\lambda+\mu}).\end{align*}
In addition, we have
\begin{align*}
\hat{\psi}&\left(e_i\mathrm{log}_F(x_{\alpha_j})\right)=e_i\mathrm{log}_F\left(\hat{\psi}\left(x_{\alpha_j}\right)\right)=e_i\mathrm{log}_F\left(\left(x_{\alpha_j}+\widetilde{{\mathcal{J}}}_F\right)+{\mathcal{J}}_F^\star\right)=\left(e_i\mathrm{log}_F\left( x_{\alpha_j}\right)+\widetilde{{\mathcal{J}}}_F\right)+{\mathcal{J}}_F^\star\\&=\left(\mathrm{log}_F\left( x_{e_i\alpha_j}\right)+\widetilde{{\mathcal{J}}}_F\right)+{\mathcal{J}}_F^\star=\mathrm{log}_F\left(\left(x_{e_i\alpha_j}+\widetilde{{\mathcal{J}}}_F\right)+{\mathcal{J}}_F^\star\right)=\mathrm{log}_F\left(\hat{\psi}\left(x_{e_i\alpha_j}\right)\right)=\hat{\psi}\left(\mathrm{log}_F\left(x_{e_i\alpha_j}\right)\right).
\end{align*}
Thus, $\hat{\psi}$ induces a continuous ring homomorphism \[{\psi}\colon R\llbracket x_\lambda\rrbracket/\mathcal{J}_F^\circ\to \left(R\llbracket x_\Lambda\rrbracket/\widetilde{{\mathcal{J}}}_F\right)/{\mathcal{J}}_F^\star,\quad {\psi}\left(x_\lambda+{\mathcal{J}}_F^\circ\right)=\left(x_\lambda+\widetilde{{\mathcal{J}}}_F\right)+{\mathcal{J}}_F^\star.\]

The continuous ring homomorphisms ${\phi}$ and ${\psi}$ are $W$-equivariant and inverse to each other. Therefore, we have proven the Lemma.
\end{proof}

Suppose $R$ is a subring of $\mathbb{C}$, and $(R,F)$ is a one-dimensional commutative formal group law over $R$. Then we may view $(R,F)$ as a formal group law $(\mathbb{C},F)$ over $\mathbb{C}$. We now allow $\Sigma$ to be crystallographic or noncrystallographic. We work under \cref{assumption-ring} for the remainder of this paper.

\begin{assumption}\label{assumption-ring}
$R$ is a subring of $\mathbb{C}$, and $(R,F)$ is a one-dimensional commutative formal group law over $R$. In addition, if $\Sigma$ is \textit{noncrystallographic}, then $R$ is an ample ring with respect to $(\mathbb{C},F)$, such that $R$ contains $\mathcal{R}$.
\end{assumption}

We define elements $f_{i,j}\in R\llbracket x_\Lambda\rrbracket$, $i=1,\dotsc,l$, and $j=1,\dotsc,n$, as follows:
\[f_{i,j}=\begin{cases}
e_i\mathrm{log}_F(x_{\alpha_j})-\mathrm{log}_F(x_{e_i\alpha_j}),\quad \text{$\Sigma$ noncrystallographic};\\
0,\quad\quad\quad\quad\quad\quad\quad\quad\quad\quad\quad\textcolor{white}{-} \text{$\Sigma$ crystallographic}.
\end{cases}\]
Let $\mathcal{J}_F$ be the closure of the ideal in $R\llbracket x_\Lambda\rrbracket$ generated by the elements 
\[x_0 \quad \text{and}\quad x_{\lambda_1+\lambda_2}-(x_{\lambda_1}+_Fx_{\lambda_2}) \quad \text{and}\quad f_{i,j},\]
over all $i=1,\dotsc,l$, and $j=1,\dotsc,n$, and $\lambda_1,\lambda_2\in\Lambda$. 
\begin{dfn}
	The quotient 
	\[{R\llbracket\Lambda\rrbracket}_{F}:=R\llbracket x_\Lambda\rrbracket_F/\mathcal{J}_F\]
	is the \textit{formal group ring} with respect to $(R,F)$ and $\Lambda$. 
\end{dfn}
 
We denote the image of $x_\lambda$ in ${R\llbracket\Lambda\rrbracket}_{F}$ by the same symbol. By similar reasoning as earlier, ${R\llbracket\Lambda\rrbracket}_{F}$ is a complete Hausdorff ring with respect to the $\mathcal{I}_{F}$-adic topology, where $\mathcal{I}_{F}$ is the kernel of the augmentation map ${R\llbracket\Lambda\rrbracket}_{F}\to R$ sending $x_\lambda\mapsto 0$ for all $\lambda\in\Lambda$.

\begin{rem}
	If $\Sigma$ is \textit{crystallographic}, then $\mathcal{J}_F$ is the closure of the ideal in $R\llbracket x_\Lambda\rrbracket$ generated by the elements 
	\[ x_0\quad \text{and}\quad x_{\lambda+\mu}-(x_\lambda+_Fx_\mu),\]
	over all $\lambda,\mu\in\Lambda$. In this case, ${R\llbracket\Lambda\rrbracket}_{F}$
	is the classical formal group ring with respect to the crystallographic root lattice $\Lambda$.
\end{rem}

\begin{rem}
	In \cite{CPZ}, the formal group ring is defined over any commutative unital ring $S$. However, if $\Sigma$ is noncrystallographic, then the definition of $\mathcal{J}_F$ requires that $S$ contains $\mathcal{R}$. In order to provide a simple unified exposition of the formal group ring for all finite root systems, we restrict $S$ to a subring of $\mathbb{C}$ in this paper (subject to \cref{assumption-ring}).
\end{rem}

\begin{lem}\label{lem:same-properties}
	The ring ${R\llbracket\Lambda\rrbracket}_{F}$ satisfies the following properties:
	\begin{enumerate}
		\item\label{itm:same-one} There is a well-defined $W$-action on ${R\llbracket\Lambda\rrbracket}_{F}$ given by 
		\[w(x_{\lambda})=x_{w(\lambda)},\quad w\in W,\text{ }\lambda\in\Lambda.\]
		\item\label{itm:same-two} There is a continuous ring isomorphism 
		\[{R\llbracket\Lambda\rrbracket}_{F}\to R\llbracket x_1,\dotsc,x_n\rrbracket,\quad x_{e_i\alpha_j}\mapsto \mathrm{exp}_{F}(e_ix_j).\]
		In particular, ${R\llbracket\Lambda\rrbracket}_{F}$ is an integral domain.
		\item \label{itm:same-three} $x_{\alpha_j}$ divides $x_{e_i\alpha_j}$ in ${R\llbracket\Lambda\rrbracket}_{F}$ for all $e_i\in B$ and $\alpha_j\in\Delta$.
	\end{enumerate}
\end{lem}
\begin{proof}
	First assume $\Sigma$ is noncrystallographic. Property (\ref{itm:same-one}) follows from \cref{rem:well-defined} and \cref{lem:second-perspective}. Property (\ref{itm:same-two}) follows from \cref{lem:integral-isomorphism} and \cref{lem:second-perspective}. Property (\ref{itm:same-three}) follows from \cref{lem:mixed-divisible} and \cref{lem:second-perspective}. 
	
	Now assume $\Sigma$ is crystallographic. Then $B=\{e_1=1\}$. Property (\ref{itm:same-one}) follows from \cref{eqn:w-action}. Property (\ref{itm:same-two}) follows from \cref{equation:iso}. Property (\ref{itm:same-three}) is trivial.
\end{proof}

\begin{rem}
	Since $e_1=1$, the isomorphism of \cref{lem:same-properties} (\ref{itm:same-two}) sends $x_{\alpha_j}\mapsto \mathrm{exp}_F(x_{j})$ for each $\alpha_j\in \Delta$.
\end{rem}

For a finite sequence of simple roots $I=(\alpha_{i_1},\dotsc,\alpha_{i_r})$, we set $x_{I}:=x_{\alpha_{i_1}}\dotsb x_{\alpha_{i_r}}$. We define the length of $I$ to be $l(I)=r$. We say a sequence $I=(\alpha_{i_1},\dotsc,\alpha_{i_r})$ of simple roots is \textit{ordered} if $i_1\leq i_2\leq\dotsb \leq i_r$. Let $\Upsilon$ be the set of all sequences $I$ of simple roots such that $l(I)\geq 1$, and let $\Upsilon_r$ be the subset of $\Upsilon$ consisting of sequences of simple roots of length $r$.

\begin{rem}{\label{rem:unique-decomposition}}
		Let $\lambda\in\Lambda$. Since $\{e_i\alpha_j\}$ is a $\mathbb{Z}$-basis for $\Lambda$, we can write $\lambda=\sum_{j=1}^n\sum_{i=1}^lc_{i,j}e_i\alpha_j$ for some $c_{i,j}\in\mathbb{Z}$. The relations in $R\llbracket\Lambda\rrbracket_{F_a}$ allow us to write $x_\lambda\in R\llbracket\Lambda\rrbracket_{F_a}$ in the form 
		$x_\lambda=\sum_{j=1}^n\left(\sum_{i=1}^lc_{i,j}e_i\right)x_{\alpha_j}$, where $\sum_{i=1}^lc_{i,j}e_i\in\mathcal{R}$ for each fixed $j$. In particular, $x_\lambda$ can be written as an $\mathcal{R}$-linear combination of the elements $x_{\alpha_j}$, where $\alpha_j\in\Delta$. Since $e_1=1$, the isomorphism of \cref{lem:same-properties} (\ref{itm:same-two}) implies that $x_\lambda$ can be written \textit{uniquely} as an $\mathcal{R}$-linear combination of the elements $x_{\alpha_j}$, where $\alpha_j\in\Delta$. Thus, any product of $r$ $x_{\lambda}$'s, where $r\geq 1$, can be written uniquely as an $\mathcal{R}$-linear combination of elements of the form $x_I$, where $I\in\Upsilon_r$. 
\end{rem}

\begin{rem}{\label{rem:R-unique-decomposition}}
	Since $e_1=1$, the isomorphism of \cref{lem:same-properties} (\ref{itm:same-two}) implies that any element in $R\llbracket\Lambda\rrbracket_{F_a}$ can be written \textit{uniquely} as an ${R}$-linear combination of $1$ and the elements $x_I$, where $I\in\Upsilon$. 
\end{rem}

\begin{ex}{\label{ex:additiveFGL}}{\textup{(cf.} \cite[Ex. 2.19]{CPZ}\textup{)}}
	Let $S_R^i(\Lambda)$ be the $i$-th symmetric power of the $R$-module $R\otimes_{\mathcal{R}}\Lambda$, where we view $\Lambda$ as a free finitely-generated $\mathcal{R}$-module with basis $\Delta$. For $I=(\alpha_{i_1},\dotsc,\alpha_{i_r})\in\Upsilon$, set $\alpha_I:=\alpha_{i_1}\dotsb\alpha_{i_r}\in S_R^r(\Lambda)$. By definition, every element in $S_R^r(\Lambda)$ can be written \textit{uniquely} as an $R$-linear combination of $1$ and the elements $\alpha_I$, where $I\in \Upsilon_r$. The $R$-algebra $(S_R^*(\Lambda))^\wedge:=\prod\limits_{i=0}^{\infty}S_R^i(\Lambda)$ is the completion of the symmetric algebra $S_R^*(\Lambda):=\bigoplus\limits_{i=0}^{\infty}S_R^i(\Lambda)$ at the ideal generated by the kernel of the augmentation map $\alpha_I\mapsto 0$, $I\in \Upsilon$. By \cref{rem:R-unique-decomposition}, every element in $R\llbracket\Lambda\rrbracket_{F_a}$ can be written \textit{uniquely} as an ${R}$-linear combination of $1$ and the elements $x_I$, where $I\in\Upsilon$. Thus, there is a continuous $R$-linear ring isomorphism
	\[\phi:	\begin{tikzcd}
	R\llbracket\Lambda\rrbracket_{F_a} \arrow[r,] & (S_R^*(\Lambda))^\wedge,
	\end{tikzcd}\]
	sending $x_I\mapsto \alpha_I$ for all $I\in\Upsilon$, and extended by $R$-linearity. 
\end{ex}

\begin{ex}{\label{ex:multiplicaiveFGL}}{\textup{(cf.} \cite[Ex. 2.20]{CPZ}\textup{)}}
	Consider the group ring 
	\[R[\Lambda]:=\left\{\sum_{j}r_je^{\lambda_j}\mid r_j\in R,\quad \lambda_j\in\Lambda\right\}.\]
	Let $\mathrm{tr}:R[\Lambda]\to R$ be the $R$-linear \textit{trace} map that sends $e^\lambda\mapsto 1$ for all $\lambda\in\Lambda$. Let $R[\Lambda]^\wedge$ be the $\mathrm{ker}(\mathrm{tr})$-adic completion of $R[\Lambda]$. There is a continuous ring isomorphism
	\[\hat{h}:\widetilde{R\llbracket \Lambda\rrbracket}_{F_m}\to  R[\Lambda]^\wedge,\]
	where $\hat{h}(x_\lambda)= e^\lambda-1$ and $\hat{h}^{-1}(e^\lambda)= 1+x_\lambda$.
	
	Suppose that $\Sigma$ is noncrystallographic. Let $\mathcal{E}$ be the closure of the ideal in $R[\Lambda]^\wedge$ generated by the elements \[e_i\mathrm{log}_{F_m}(1-e^{\alpha_j})-\mathrm{log}_{F_m}(1-e^{e_i\alpha_j}),\] over all $e_i\in B$ and $\alpha_j\in\Delta$. We have $\hat{h}(\widetilde{\mathcal{J}}_{F_m})=\mathcal{E}$. To see this, note that, by continuity, $\hat{h}(\widetilde{\mathcal{J}}_{F_m})\subseteq \mathcal{E}$. Since $\hat{h}$ is an isomorphism of topological rings, $\hat{h}(\widetilde{\mathcal{J}}_{F_m})$ is closed, so we get equality. Through the continuous ring isomorphism $R\llbracket \Lambda\rrbracket_{F_m}\to  \widetilde{R\llbracket \Lambda\rrbracket}_{F_m}/\mathcal{J}_{F_m}$ of  \cref{lem:second-perspective}, we see that $\hat{h}$ induces a continuous ring isomorphism
	\[{h}\colon R\llbracket \Lambda\rrbracket_{F_m}\to  \left(R[\Lambda]^\wedge\right)/\mathcal{E},\]
	such that ${h}(x_\lambda)= e^\lambda-1$ and ${h}^{-1}(e^\lambda)= 1+x_\lambda$.
\end{ex}

\section{Formal Demazure operators}\label{section:localization}

In this section, we show that the element $u-s_\alpha(u)$ is divisible by $x_\alpha$ in $R\llbracket\Lambda\rrbracket_F$ for all $u\in R\llbracket\Lambda\rrbracket_F$ and $\alpha\in\Sigma$. Using this fact, we define the \textit{formal Demazure operator} $\Delta_\alpha^{(R,F)}$ in this section. When $(R,F)=(R,F_a)$ is the additive formal group law over $R$, the isomorphism $\phi$ of \cref{ex:additiveFGL} exchanges the formal Demazure operator $\Delta_\alpha^{(R,F_a)}$ on $R\llbracket \Lambda\rrbracket_{F_a}$ with an operator $(\Delta_\alpha^R)^\wedge$ on $(S_R^*(\Lambda))^\wedge$. We conclude this section by studying the restriction of $(\Delta_\alpha^R)^\wedge$ to the symmetric algebra $S_R^*(\Lambda)$.

The following result will allow us to define the formal Demazure operator on $R\llbracket\Lambda\rrbracket_F$.

\begin{lem}{\textup{(cf.} \cite[Cor. 3.4]{CPZ}\textup{)}}{\label{lem:divisible}}
	For any $u\in R\llbracket\Lambda\rrbracket_F$ and root $\alpha\in\Sigma$, the element $u-s_\alpha(u)$ is divisible by $x_\alpha$ in $R\llbracket\Lambda\rrbracket_F$.
\end{lem}
\begin{proof}\details{\textcolor{red}{This lemma and proof are different.}}
	First we assume that $\alpha=\alpha_j$ is a \textit{simple} root. Since $\alpha_j^\vee(\lambda)\in \mathcal{R}$ for any $\lambda\in\Lambda$, it follows that $\alpha_j^\vee(\lambda)=\sum_{i=1}^l c_ie_i$ for some $c_i\in\mathbb{Z}$. We have
	\begin{align*}
	s_j(x_\lambda)=x_{\lambda-\alpha_j^\vee(\lambda)\alpha_j}&=x_{\lambda-\sum_{i=1}^l c_i(e_i\alpha_j)}=x_\lambda+_F (((-c_1)\cdot_F x_{e_1\alpha_j})+_F\dotsb+_F ((-c_l)\cdot_F x_{e_l\alpha_j})).
	\end{align*}
	Set $r=((-c_1)\cdot_F x_{e_1\alpha_j})+_F\dotsb+_F ((-c_l)\cdot_F x_{e_l\alpha_j})$. By (\ref{itm:same-three}) of \cref{lem:same-properties}, $x_{\alpha_j}$ divides each $x_{e_i\alpha_j}$. Hence, $x_{\alpha_j}$ divides each $(-c_i)\cdot_F x_{e_i\alpha_j}$. Thus, $x_{\alpha_j}$ divides $r$. Write $u+_Fv=u+vg(u,v)$, where $g(u,v)\in R\llbracket u,v\rrbracket$. Then
	\[x_{\lambda}-s_\alpha(x_\lambda)=x_{\lambda}-(x_\lambda+_Fr)=x_{\lambda}-(x_\lambda+rg(x_\lambda,r))=-rg(x_\lambda,r).\] 	
	So $x_{\alpha_j}$ divides $x_{\lambda}-s_j(x_\lambda)$. Then, by the formula
	\[xy-s_j(xy)=(x-s_j(x))y+x(y-s_j(y))-(x-s_j(x))(y-s_j(y)),\]
	the element $x_{\alpha_j}$ divides $u-s_j(u)$ for any monominal $u\in R[x_\Lambda].$ Finally, $x_{\alpha_j}$ divides $u-s_j(u)$ for any $u\in R\llbracket\Lambda\rrbracket_F$ by density of $R\llbracket\Lambda\rrbracket_F$.
	
	Now let $\alpha\in\Sigma$ be \textit{any} root. By \cref{lem:root-conjugation}, the root $\alpha$ can be written $\alpha=w(\alpha_j)$ for some $\alpha_j\in\Delta$ and $w\in W$. Furthermore, \cref{prop:orthogonal-property} says that $ws_{\beta}w^{-1}=s_{w(\beta)}$ for any $w\in W$ and $\beta\in\Sigma$. Thus, given $u\in R\llbracket\Lambda\rrbracket_F$, we see that
	\begin{align*}
	\tfrac{u-s_\alpha(u)}{x_\alpha}&=\tfrac{u-s_{w(\alpha_j)}(u)}{x_{w(\alpha_j)}}=\tfrac{(ww^{-1})(u)-(ws_j w^{-1})(u)}{(ww^{-1})(x_{w(\alpha_j)})}=w\left(\tfrac{w^{-1}(u)-s_j(w^{-1}(u))}{x_{\alpha_j}}\right)\in R\llbracket\Lambda\rrbracket_F.
	\end{align*}\qedhere
\end{proof}

\begin{dfn}{\label{dfn:formal-Demazure-operator}}
	Following \cite[Def.~4.1]{HMSZ}, for each root $\alpha\in\Sigma$, we define an $R$-linear operator $\Delta_\alpha^{(R,F)}$ on $R\llbracket\Lambda\rrbracket_{F}$ by the formula
	\[\Delta_\alpha^{(R,F)}(u)=\tfrac{u-s_\alpha(u)}{x_\alpha},\quad u\in R\llbracket\Lambda\rrbracket_{F}.\]
	We call the operator $\Delta_\alpha^{(R,F)}$ a \textit{formal Demazure operator}. We will usually write $\Delta_\alpha$ instead of $\Delta_\alpha^{(R,F)}$, omitting the superscript.
\end{dfn}

\details{
\begin{ex}
	Let $W=H_2$, and let $(R,F)$ be the additive formal group law over $R$. Choose a basis $B(1,\tau)$ of $\mathcal{R}$ over $\mathbb{Z}$, where $\tau=\tfrac{1+\sqrt{5}}{2}$ is the golden ratio. In this case, we can choose our simple roots $\alpha_1=(1,0)$ and $\alpha_2=(-\tfrac{\tau}{2},\tfrac{\sqrt{3-\tau}}{2})$, with coordinates in $\mathbb{R}^2$. Note that $\tau$ satisfies the relation $\tau^2=\tau+1$. 
	
	Consider the root $\alpha=-(\phi\alpha_1+\alpha_2)$. One can show that $(\alpha,\alpha)=1$ and $(\alpha_2,\alpha)=\tfrac{\phi-1}{2}$. We compute
	\begin{align*}
	x_{\alpha_2}-s_{\alpha}(x_{\alpha_2})=x_{\alpha_2}-x_{\alpha_2-2\tfrac{(\alpha_2,\alpha)}{(\alpha,\alpha)}\alpha}
	&=x_{\alpha_2}-x_{\alpha_2+(\phi-1)(\phi\alpha_1+\alpha_2)}\\
	&=x_{\alpha_2}-x_{\alpha_1+\phi x_{\alpha_2}}\\
	&=x_{\alpha_2}-x_{\alpha_1}-\phi x_{\alpha_2}.
	\end{align*}
	Thus, we see that $x_{\alpha_2}-s_{\alpha}(x_{\alpha_2})$ is not divisible by $x_{\alpha}$
\end{ex}
}

\details{\begin{rem}
	Suppose $\Sigma$ is a crystallographic root system. If $(F,R)$ is the additive or multiplicative formal group law over $R,$ then $\Delta_{I_w}$ does not depend on the choice of reduced sequence of $w.$ In this case, it corresponds to the classical Demazure operator $\Delta_w$ of \cite[\S 3 and \S 9]{D73}. For other formal group laws, $\Delta_{I_w}$ depends on the choice of $I_w$ (see \cite[Thm. 3.9]{CPZ}).	
\end{rem}}

\begin{prop}{\label{prop:various-relations}}{\textup{(see} \cite[Prop. 3.8]{CPZ} \textup{and} \cite[\S 3]{D73}\textup{)}} The following formulas hold for any $u,v\in R\llbracket\Lambda\rrbracket_{F}$, $\alpha\in\Sigma$, and $w\in W.$
	\begin{enumerate}
		\item\label{uno} $\Delta_\alpha(1)=0$, \quad  $\Delta_\alpha(u)x_\alpha=u-s_\alpha(u)$;
		\item $\Delta_\alpha^2(u)x_\alpha=\Delta_\alpha(u) +\Delta_{-\alpha}(u)$,  \quad $\Delta_\alpha(u)x_\alpha=\Delta_{-\alpha}(u)x_{-\alpha}$;
		\item $s_\alpha\Delta_\alpha(u)=-\Delta_{-\alpha}(u),$ \quad $\Delta_\alpha s_\alpha(u)=-\Delta_\alpha(u)$;
		\item\label{fouro} $\Delta_\alpha(uv)=\Delta_\alpha(u)v+u\Delta_\alpha(v)-\Delta_\alpha(u)\Delta_\alpha(v)x_\alpha=\Delta_\alpha(u)v+s_\alpha(u)\Delta_\alpha(v)$;
		\item\label{fiveo} $w\Delta_\alpha w^{-1}(u)=\Delta_{w(\alpha)}(u)$.
	\end{enumerate}
\end{prop}
\begin{proof}The formulas (\ref{uno})-(\ref{fouro}) are straightforward computations using the definition of the formal Demazure operator. Relation (\ref{fiveo}) follows from \cref{prop:orthogonal-property}, which says that $ws_{\alpha}w^{-1}=s_{w(\alpha)}$. 
\end{proof}

Let $(\Delta_{\alpha}^{R})^\wedge\colon (S_R^*(\Lambda))^\wedge\to (S_R^*(\Lambda))^\wedge$ be the $R$-linear operator corresponding to the operator $\Delta_{\alpha}:R\llbracket\Lambda\rrbracket_{F_a}\to R\llbracket\Lambda\rrbracket_{F_a}$ with respect to the isomorphism $\phi:R\llbracket\Lambda\rrbracket_{F_a}\to (S_R^*(\Lambda))^\wedge$ of \cref{ex:additiveFGL}. The operator $(\Delta_{\alpha}^{R})^\wedge$ restricts to an $R$-linear operator $\Delta_{\alpha}^{R}\colon S_R^*(\Lambda)\to S_R^*(\Lambda)$.
If $\Sigma$ is crystallographic and $R=\mathbb{Z}$, then $\Delta_i^{\mathbb{Z}}$ is the classical Demazure operator of \cite{D73}, which acts on the symmetric algebra $S_\mathbb{Z}(\Lambda):=\bigoplus\limits_{i=1}^\infty S_\mathbb{Z}^i(\Lambda)$. If $R=\mathbb{C}$, then $\Delta_i^{\mathbb{C}}$ is the Demazure operator of \cite[Ch. IV, pp. 135]{H82}, which acts on the symmetric algebra $S(V)=S_\mathbb{C}(\Lambda):=\bigoplus\limits_{i=1}^\infty S_\mathbb{C}^i(\Lambda)$, where $V$ is the complex vector space generated by the roots $\beta\in\Sigma$.

For $\alpha_i\in \Delta$, we will use the notation $\Delta_i:=\Delta_{\alpha_i}$ and $\Delta_i^R:=\Delta_{\alpha_i}^R$. For a finite sequence of simple roots $I=(\alpha_{i_1},\dotsc,\alpha_{i_r})$, we set:
\[s_I:=s_{i_1}\dotsb s_{i_r};\quad \quad \Delta_I:=\Delta_{i_1}\circ\dotsb\circ\Delta_{i_r};\quad \quad \Delta_I^R:=\Delta_{i_1}^R\circ\dotsb\circ\Delta_{i_r}^R.\] 
We also set $\alpha_I:=\alpha_{i_1}\dotsb\alpha_{i_r}\in S_R^r(\Lambda)$. We say that $I$ is \textit{reduced} if $s_I$ is reduced in $W$. Fix a reduced decomposition of a reflection $w=s_{i_1}\dotsb s_{i_r}$, where $i_j\in\{ 1,\dotsc,n\}$. We call $I_w:=(\alpha_{i_1},\dotsc,\alpha_{i_r})$ a \textit{reduced sequence of $w$}.

\begin{lem}{\label{lem:symmetric-sum-independent-of-R}}
	Let $I=(\alpha_{i_1},\dotsc,\alpha_{i_k})\in\Upsilon_k$.
	We can write \[\Delta_i^{R}(\alpha_{i_1}\dotsb \alpha_{i_k})=\sum\limits_{I'\in\Upsilon_{k-1}} c_{I'}^R \alpha_{I'},\]
	where the $c_{I'}^R\in\mathcal{R}$ are independent of $R$.
\end{lem}
\begin{proof}
	First note that, for any $\lambda\in\Lambda$, we have $\Delta_i^{R}(\lambda)=\alpha_i^\vee(\lambda)\in\mathcal{R}$. Thus, since $\mathcal{R}$ is a ring, the result follows by induction on the degree of monomials using (\ref{fouro}) of \cref{prop:various-relations}. \qedhere
\end{proof}

\begin{prop}{\textup{(see} \cite[\S 4, Thm. 1]{D73} and \cite[Ch. IV, Prop. 1.7]{H82}\textup{)}}{\label{prop:product-trivial}}
		Let $I_w$ and $I_w'$ be reduced sequences for the same element $w\in W$. Then \[\Delta_{I_w}^R=\Delta_{I_w'}^R.\]		
\end{prop}
\begin{proof}
	By \cite[Ch. IV, Prop.~1.7]{H82}, this result holds when $R=\mathbb{C}$. Now the Proposition follows from \cref{lem:symmetric-sum-independent-of-R}, and from the fact that $\{1\}\cup \{\alpha_I\}_{I\in\Upsilon}$ is an $R$-basis of $S_R^*(\Lambda)$.
\end{proof}

\begin{rem}
	\cref{prop:product-trivial} implies that $\Delta_{I_w}^R$ is independent of the reduced sequence $I_w$ of $w\in W$. By the density of $(S_R^*(\Lambda))^\wedge$, the operator $(\Delta_{I_w}^R)^\wedge$ is independent of the reduced sequence $I_w$ of $w\in W$ as well. Thus, the operator $\Delta_{I_w}^{(R,F_a)}$ is independent of the reduced sequence $I_w$ of $w\in W$. In general, $\Delta_{I_w}^{(R,F)}$ depends on the reduced sequence $I_w$ of $w\in W$.
\end{rem}

From now on, we will simply write $\Delta_w^R$ to denote the operator $\Delta_{I_w}^R$, where $I_w$ is any reduced sequence of $w\in W$. 

\begin{prop}{\textup{(see} \cite[\S 4, Prop. 3(a)]{D73} \textup{and} \cite[Ch. IV, Lem. 2.2]{H82}\textup{)}}{\label{prop:product-trivial2}}
	Let $w,w'\in W$. We have
		\[\Delta_{w}^R\Delta_{w'}^R=\begin{cases}\Delta_{{ww'}}^R,\quad\textup{if $l(ww')=l(w)+l(w')$},\\
		0,\quad\quad\quad \textup{otherwise.}
		\end{cases}\]
		\end{prop}
\begin{proof}
	By \cite[Ch. IV, ~Lem. 2.2]{H82}, this results holds when $R=\mathbb{C}$. Now the Proposition follows from \cref{lem:symmetric-sum-independent-of-R}, and from the fact that $\{1\}\cup \{\alpha_I\}_{I\in\Upsilon}$ is an $R$-basis of $S_R^*(\Lambda)$.
\end{proof}

Since $S_R^*(\Lambda)$ injects into the integral domain $(S_R^*(\Lambda))^\wedge\simeq R\llbracket\Lambda\rrbracket_F$, the symmetric algebra $S_R^*(\Lambda)$ is an integral domain. We let $\overline{S_R^*(\Lambda)}$ be its field of fractions.  

\begin{lem}{\textup{(cf.} \cite[Lem. 3]{D73}\textup{)}}{\label{lem:operator-independent}}
	Set $\Sigma_{w}:=\Sigma^+\cap w(\Sigma^-)$ and $q_{w}= \left(\prod\limits_{\alpha\in\Sigma_{w} }\alpha\right)$. We have
	\[q_{w}\Delta_{w}^R=\mathrm{det}(w)w+\sum\limits_{w'<w}a(I_{w'})w',\]
	where the $a(I_{w'})\in \overline{S_R^*(\Lambda)}$, and the ordering $w'<w$ is with respect to the Bruhat order on $W$. 
\end{lem}
\begin{proof}
	This proof is the same as the proof of \cite[Lem. 3]{D73}; 	
	however, the proof of \cite[Lem. 3]{D73} uses the result \cite[Ch. VI, \S 1, No 6, Cor. 2]{B68}, which is a property of the Weyl group. Since we are working with real finite reflection groups, we provide an updated reference. The result \cite[Ch. VI, \S 1, No 6, Cor. 2]{B68} can be replaced by \cref{prop:positive-roots}.
\end{proof}

\begin{dfn}
Let $\mathcal{D}_R(\Lambda)$ be the $R$-algebra of endomorphisms of $S_R^*(\Lambda)$ generated by the Demazure operators $\Delta_i^R$ and by multiplication by elements in $S_R^*(\Lambda)$.
\end{dfn} 

\begin{cor}{\textup{(cf.} \cite[Cor. 1]{D73}\textup{)}}{\label{cor:basis-of-additive-Demazure-operators}}
The algebra $\mathcal{D}_R(\Lambda)$ is free as a left $S_R^*(\Lambda)$-module with basis $\{\Delta_{w}^R\}_{w\in W}$.
\end{cor}
\begin{proof}
	Since \cref{prop:various-relations} and \cref{lem:operator-independent} hold, this proof is the same as the proof of \cite[Cor. 1]{D73}.
\end{proof}

\details{
\begin{prop}{\textup{(see} \cite[Ch. IV, Prop. 1.8]{H82}\textup{)}}{\label{prop:basis-of-additive-Demazure-operators}}
	Fix a reduced sequence $I_w$ for each $w\in W$. Then $\mathcal{D}_R(\Lambda)$ is free as a $S_R^*(\Lambda)$-module with basis $\{\Delta_{I_w}^R\}_{w\in W}$.
\end{prop}
\begin{proof}
	By \cite[Ch. IV, Prop.~1.7]{H82}, $\mathcal{D}_\mathbb{C}(\Lambda)$ is free as a $S_{\mathbb{C}}^*(\Lambda)$-module with basis $\{\Delta_{I_w}^\mathbb{C}\}_{w\in W}$. Write $T^R=\sum_{w\in W}q_w^R\Delta_{I_w}^R$, with $q_w^R\in S_R^*(\Lambda)$. Suppose $T^R=0$, but $q_{w'}^R\neq0$ for some $w'\in W$. Then $q_{w'}^{\mathbb{C}}\neq 0$ by \cref{rem:zero-if-and-only-if}. Since the $\Delta_{I_w}^\mathbb{C}$ are $S_\mathbb{C}^*(\Lambda)$-linearly independent,  we must have $T^{\mathbb{C}}\neq 0$. Thus, $T^{\mathbb{C}}(\alpha_I)\neq 0$ for some ordered sequence $I$ of simple roots. Now it follows from \cref{rem:zero-if-and-only-if}  and \cref{lem:symmetric-sum-independent-of-R} that $T^R(\alpha_I)\neq 0$ as well. This contradiction shows that $\{\Delta_{I_w}^R\}_{w\in W}$ is $S_{R}^*(\Lambda)$-linearly independent in $\mathcal{D}_R(\Lambda)$.
	The fact that the $\Delta_{I_w}^R$ generate $\mathcal{D}_R(\Lambda)$ follows from (\ref{fouro}) of \cref{prop:various-relations} and \cref{rem:analogous}.
\end{proof}
}

\section{Endomorphisms of the formal group ring}\label{section:endomorphisms}

In this section, we define the associated graded ring $\mathcal{G}r_{(R,F)}^*(\Lambda)$, and we discuss the subalgebra $\mathcal{D}_{(R,F)}(\Lambda)$ of the endomorphism algebra of $R\llbracket\Lambda\rrbracket_{F}$ generated by the formal Demazure operators and by multiplication by elements in $R\llbracket\Lambda\rrbracket_{F}$. This section closely follows \cite[\S 4]{CPZ}.

If $I=(\alpha_{i_1},\dotsc,\alpha_{i_r})$ is a sequence of simple roots, we set $x_{I}:=x_{\alpha_{i_1}}\dotsb x_{\alpha_{i_r}}$. Denote by $\mathcal{I}$ the kernel of the augmentation map $R\llbracket x_\Lambda\rrbracket\to R$. Recall the kernel $\mathcal{I}_{F}$ of the augmentation map $R\llbracket\Lambda\rrbracket_{F}\to R$. By convention, we set $\mathcal{I}_{F}^i=R\llbracket\Lambda\rrbracket_{F}$ for $i\leq 0$. Recall that $R\llbracket\Lambda\rrbracket_{F}=R\llbracket x_\Lambda\rrbracket/\mathcal{J}_F$, where $\mathcal{J}_F$ is the closure of the ideal in $R\llbracket x_\Lambda\rrbracket$ generated by the elements 
\[x_0 \quad \text{and}\quad x_{\lambda_1+\lambda_2}-(x_{\lambda_1}+_Fx_{\lambda_2}) \quad \text{and}\quad f_{i,j},\]
over all $e_i\in B$, $\alpha_j\in\Delta$, and $\lambda_1,\lambda_2\in\Lambda$, where
\[f_{i,j}=\begin{cases}
e_i\mathrm{log}_F(x_{\alpha_j})-\mathrm{log}_F(x_{e_i\alpha_j}),\quad \text{$\Sigma$ noncrystallographic};\\
0,\quad\quad\quad\quad\quad\quad\quad\quad\quad\quad\quad\textcolor{white}{-} \text{$\Sigma$ crystallographic}.
\end{cases}\]
When $\Sigma$ is noncrystallographic, observe that
\[f_{i,j}=e_i\mathrm{log}_F(x_{\alpha_j})-\mathrm{log}_F(x_{e_i\alpha_j})=e_ix_{\alpha_j}- x_{e_i\alpha_j}+g_{i,j},\]
where $g_{i,j}=\sum_{k\geq 2}b_k(e_ix_{\alpha_j}^k-x_{e_i\alpha_j}^k)\in{\mathcal{I}}_F^2$. Here the $b_{k}\in R$ are the coefficients that appear in the series $\mathrm{log}_F(u)=u+\sum_{k\geq 2}b_ku^k$.

We define the associated graded ring,
\[\mathcal{G}r_{(R,F)}^*(\Lambda)=\bigoplus_{i=0}^{\infty}\mathcal{I}_{F}^i/\mathcal{I}_{F}^{i+1}.\]
We have 
\[\mathcal{I}_{F}={\mathcal{I}}/\mathcal{J}_{F}\quad\text{and}\quad \mathcal{I}_{F}^{k}/\mathcal{I}_{F}^{k+1}={\mathcal{I}}^k/(\mathcal{J}_{F}\cap {\mathcal{I}}^k+{\mathcal{I}}^{k+1}).\]
Set $\mathcal{T}_k:=\mathcal{J}_{F}\cap {\mathcal{I}}^k+{\mathcal{I}}^{k+1}$. When $\Sigma$ is crystallographic, $\mathcal{T}_k/{\mathcal{I}}^{k+1}$ is generated by elements of the form
\[\rho\cdot x_0\quad \text{and} \quad \rho \cdot (x_\lambda+x_\mu-x_{\lambda+\mu}),\] 
where $\alpha_j\in\Delta$, $\lambda,\mu\in\Lambda$, and $\rho$ is a monomial of degree $k-1$. 
When $\Sigma$ is noncrystallographic, $\mathcal{T}_k/{\mathcal{I}}^{k+1}$ is generated by elements of the form 
\[\rho\cdot x_0\quad \text{and} \quad \rho \cdot (x_\lambda+x_\mu-x_{\lambda+\mu})\quad \text{and}\quad\rho\cdot (e_ix_{\alpha_j}-x_{e_i\alpha_j}),\] 
where $e_i\in B$, $\alpha_j\in\Delta$, $\lambda,\mu\in\Lambda$, and $\rho$ is a monomial of degree $k-1$. In either case, the generators of $\mathcal{T}_k/{\mathcal{I}}^{k+1}$ allow us to write any element in \[\mathcal{I}_{F}^k/\mathcal{I}_{F}^{k+1}={\mathcal{I}}^{k}/\mathcal{T}_k\simeq ({\mathcal{I}}^k/{\mathcal{I}}^{k+1})/(\mathcal{T}_k/{\mathcal{I}}^{k+1})\] as an $R$-linear combination of elements of the form $x_{I}+\mathcal{I}_{F}^{k+1}$, where $I\in\Upsilon_k$.

\begin{lem}{\label{lem:additive-iso}}{\textup{(cf.} \cite[Lem. 4.2]{CPZ}\textup{)}}
	The morphism of graded $R$-algebras
	\[\psi:S_R^*(\Lambda)\to \mathcal{G}r_{(R,F)}^*(\Lambda)\]
	defined by sending $\alpha_I$ to $x_{I}+\mathcal{I}_{F}^{k+1}$, $I\in\Upsilon_k$, and extended by $R$-linearity, is an isomorphism.
\end{lem}
\begin{proof} 
	Every element in $S_R^k(\Lambda)$, $k>0$, can be written uniquely as an $R$-linear combination of monomials of the form $\alpha_I$, where $I\in\Upsilon_k$. The map $\psi$ sends the monomial $\alpha_I$ to the element $x_{I}+\mathcal{I}_{F}^{k+1}\in  \mathcal{G}r_{(R,F)}^*(\Lambda)$, and, therefore, $\psi$ is well-defined. Since $\mathcal{G}r_{(R,F)}^*(\Lambda)$ is generated as a graded unital $R$-algebra by the elements $x_{\alpha_i}+\mathcal{I}_{F}^2$, where $\alpha_i\in\Delta$, the map $\psi$ is surjective. We define a map in the other direction.
	
	Define the map of $R$-modules
	\[\tilde{\theta}_k:{\mathcal{I}}^k/{\mathcal{I}}^{k+1}\to S_R^k(\Lambda),\]
	by sending a monomial of degree $k$ in some $x_\lambda$'s to the product of $\lambda$'s involved. 
	The map $\tilde{\theta}_k$ factors through the quotient $({\mathcal{I}}^k/{\mathcal{I}}^{k+1})/(\mathcal{T}_k/{\mathcal{I}}^{k+1})$, hence giving a well-defined map ${\theta}_k:{\mathcal{I}}^{k}/\mathcal{T}_k\to S_R^k(\Lambda)$. Equivalently, $\tilde{\theta}_k$ gives a well-defined map ${\theta}_k:\mathcal{I}_{F}^k/\mathcal{I}_{F}^{k+1}\to S_R^k(\Lambda)$, since ${\mathcal{I}}^{k}/\mathcal{T}_k=\mathcal{I}_{F}^k/\mathcal{I}_{F}^{k+1}$. The sum ${\theta}:=\oplus {\theta}_k$ is also well-defined. Since every element in $\mathcal{I}_{F}^k/\mathcal{I}_{F}^{k+1}$ is an $R$-linear combination of elements of the form $x_{I}+\mathcal{I}_{F}^{k+1}$, where $I\in\Upsilon_k$, we see that $\theta$ is the desired inverse map.  \qedhere
\end{proof}
\begin{rem}\label{rem:uniquely-in-graded-ring}
	Since every element in $S_R^k(\Lambda)$, $k>0$, can be written uniquely as an $R$-linear combination of monomials of the form $\alpha_{I}$, where the $I\in\Upsilon_k$, it follows from the isomorphism $\psi$ of \cref{lem:additive-iso} that every element in $\mathcal{I}_{F}^{k}/\mathcal{I}_{F}^{k+1}$ can be written \textit{uniquely} as an $R$-linear combination of elements of the form $x_{I}+\mathcal{I}_{F}^{k+1}$, where $I\in\Upsilon_k$.
\end{rem}

\begin{lem}\label{lem:ideal-grading-reduction}{\textup{(cf.} \cite[Prop. 4.6 (1)]{CPZ}\textup{)}}
	For any root $\alpha\in\Sigma$, the operator $\Delta_\alpha$ sends $\mathcal{I}_{F}^{k}$ to $\mathcal{I}_{F}^{k-1}$.
\end{lem}
\begin{proof}
	This follows directly from (\ref{fouro}) of \cref{prop:various-relations}.
\end{proof}
By \cref{lem:ideal-grading-reduction}, the operator $\Delta_\alpha$ induces an $R$-linear operator of degree $-1$ on the graded ring $\mathcal{G}r_{F}^*(\Lambda),$ denoted $\mathcal{G}r\Delta_\alpha$, for all $\alpha\in\Sigma$.

\begin{prop}{\label{prop:additive-reduction}}{\textup{(see} \cite[Prop. 4.4]{CPZ}\textup{)}}
	For each root $\alpha\in\Sigma$, the isomorphism of \cref{lem:additive-iso} exchanges the operator $\mathcal{G}r\Delta_{\alpha}$ on $\mathcal{G}r_{(R,F)}^*(\Lambda)$ with $\Delta_{\alpha}^{R}$ on the symmetric algebra $S_R^*(\Lambda)$.
\end{prop}
\begin{proof} By \cref{rem:uniquely-in-graded-ring}, every element in $\mathcal{I}_{F}^{k}/\mathcal{I}_{F}^{k+1}$, $k>0$, can be written uniquely as an $R$-linear combination of elements of the form $x_{{I}}+\mathcal{I}_{F}^{k+1}$, where $I\in\Upsilon_k$. Now the result follows by induction on the degree of monomials using (\ref{fouro}) of \cref{prop:various-relations}.\qedhere
\end{proof}

We recall the following fact concerning topological modules.

\begin{rem}{\textup{(cf.} \cite[\S 10]{AM69}\textup{)}}
	Let $S$ be a topological ring in the $I$-adic topology for some ideal $I$, and suppose $M$ is an $S$-module with filtration
	\[M\supseteq \dotsb \supseteq M^{(i)}\supseteq M^{(i+1)}\supseteq \dotsb \supseteq 0,\]
	such that $IM^{(i)}\subseteq M^{(i+1)}$ for all $i\in\mathbb{Z}$, and $M=\cup_iM^{(i)}$. There is a natural topology on $M$ induced by this filtration, turning $M$ into a topological $S$-module. This topology is generated by elements of the form $x+M^{(i)}$, where $x\in M$ and $i\in \mathbb{Z}$. The $S$-module $M$ is complete with respect to this topology if and only if the canonical $S$-module homomorphism $M\to  \lim\limits_{\stackrel{\longleftarrow}{i}}(M/M^{(i)})$ is an $S$-module isomorphism, and $M$ is Hausdorff with respect to this topology if and only if $\cap_{i}M^{(i)}=(0)$. 
\end{rem}

\begin{dfn}{\textup{(see} \cite[Def. 4.5]{CPZ}\textup{)}}
	Let $\mathcal{D}_{(R,F)}(\Lambda)$ be the algebra of $R$-linear endomorphisms of $R\llbracket\Lambda\rrbracket_F$ generated by the formal Demazure operators $\Delta_\alpha$ for all roots $\alpha$, and by multiplication by elements in $R\llbracket\Lambda\rrbracket_F$.
\end{dfn}

Let $\mathcal{D}_{(R,F)}(\Lambda)^{(i)}$ be the left $R\llbracket\Lambda\rrbracket_F$-submodule of $\mathcal{D}_{(R,F)}(\Lambda)$ generated by elements of the form $u\Delta_{\alpha_{i_1}}\dotsb \Delta_{\alpha_{i_k}}$, where $u\in\mathcal{I}_{F}^j$, $j-k\geq i$, and the $\alpha_{i_r}\in\Sigma$. This defines a filtration on $\mathcal{D}_{(R,F)}(\Lambda)$, 
\[\mathcal{D}_{(R,F)}(\Lambda)\supseteq\dotsb\supseteq  \mathcal{D}_{(R,F)}(\Lambda)^{(i)}\supseteq \mathcal{D}_{(R,F)}(\Lambda)^{(i+1)}\supseteq \dotsb\supseteq 0,\]
with the property $\mathcal{D}_{(R,F)}(\Lambda)=\cup_i \mathcal{D}_{(R,F)}(\Lambda)^{(i)}$. We define the associated graded $R\llbracket\Lambda\rrbracket_F$-module,
\[\mathcal{G}r^*\mathcal{D}_{(R,F)}(\Lambda)=\bigoplus_{i=-\infty}^\infty\mathcal{D}_{(R,F)}(\Lambda)^{(i)}/\mathcal{D}_{(R,F)}(\Lambda)^{(i+1)}.\]

\begin{prop}{\textup{(see} \cite[Prop. 4.6]{CPZ}\textup{)}}{\label{prop:filtration-properties}}
	The filtration on $\mathcal{D}_{(R,F)}(\Lambda)$ above has the following properties:
	\begin{enumerate}
		\item\label{itm:filt-one} For any root $\alpha$, we have 
		\[\Delta_\alpha\mathcal{D}_{(R,F)}(\Lambda)^{(i)}\subseteq \mathcal{D}_{(R,F)}(\Lambda)^{(i-1)};\quad\mathcal{D}_{(R,F)}(\Lambda)^{(i)}\Delta_\alpha\subseteq \mathcal{D}_{(R,F)}(\Lambda)^{(i-1)}.\]
		\item\label{itm:filt-two} For any integer $j$, we have $\mathcal{I}_{F}^j\mathcal{D}_{(R,F)}(\Lambda)^{(i)}\subseteq \mathcal{D}_{(R,F)}(\Lambda)^{(i+j)}$.
		\item\label{itm:filt-three} For any operator $D\in \mathcal{D}_{(R,F)}(\Lambda)^{(i)}$, we have $D(\mathcal{I}_{F}^j)\subseteq \mathcal{I}_{F}^{(i+j)}$.
		\item\label{itm:filt-four} $\cap_i\mathcal{D}_{(R,F)}(\Lambda)^{(i)}={0}$.
	\end{enumerate}
\end{prop}
\begin{proof}
	Since (\ref{fouro}) of \cref{prop:various-relations} and \cref{lem:ideal-grading-reduction} hold, and $R\llbracket\Lambda\rrbracket_F$ is Hausdorff, the proof of this result is the same as the proof of \cite[Prop. 4.6]{CPZ}. 
\end{proof}

Point (\ref{itm:filt-three}) of \cref{prop:filtration-properties} implies that the graded module $\mathcal{G}r^*\mathcal{D}_{(R,F)}(\Lambda)$ acts by graded endomorphisms on the graded ring $\mathcal{G}r^*_{(R,F)}(\Lambda)$, which is isomorphic to $S_R^*(\Lambda)$ by \cref{lem:additive-iso}. Recall the algebra of Demazure operators $\mathcal{D}_R(\Lambda)$ on $S_R^*(\Lambda)$ considered in \cref{section:localization}. By \cref{prop:additive-reduction}, the class $\left[\Delta_\alpha\right]$ of $\Delta_\alpha$ in the quotient $\mathcal{D}_{(R,F)}(\Lambda)^{(-1)}/\mathcal{D}_{(R,F)}(\Lambda)^{(0)}$ acts by $\Delta_\alpha^R$ on $S_R^*(\Lambda)$, and the class $[x_\alpha]$ of multiplication by $x_\alpha$ in $\mathcal{D}_{(R,F)}(\Lambda)^{(1)}/\mathcal{D}_{(R,F)}(\Lambda)^{(2)}$ acts by multiplication by $\alpha$ on $S_R^*(\Lambda)$. Moreover, $\mathcal{G}r^*\mathcal{D}_{(R,F)}(\Lambda)$ is generated as a graded $\mathcal{G}r_{(R,F)}^*(\Lambda)$-module by the $[x_\alpha]$ and the $\left[\Delta_\alpha\right]$ over all $\alpha\in\Sigma$. We have justified the following result:

\begin{prop}{\textup{(see} \cite[Prop. 4.8]{CPZ}\textup{)}}{\label{prop:graded-operator-basis}}
	The graded module $\mathcal{G}r^*\mathcal{D}_{(R,F)}(\Lambda)$ over the graded ring $\mathcal{G}r_{(R,F)}^*(\Lambda)\simeq S_R^*(\Lambda)$ is isomorphic to $\mathcal{D}_R(\Lambda)$.
\end{prop}

For each element $w\in W$, fix a reduced sequence of simple roots $I_w$.
\begin{thm}{\textup{(see} \cite[Lem. 4.11 (3)]{CPZ}\textup{)}}{\label{thm:operator-basis}}
	The operators $\Delta_{I_w}$ form a basis of  $\mathcal{D}_{R,F}(\Lambda)$ as a left $R\llbracket\Lambda\rrbracket_F$-module.
\end{thm}
\begin{proof}
By \cref{prop:product-trivial} and \cref{cor:basis-of-additive-Demazure-operators}, the operators $\Delta_{I_w}^R$ form a basis of $\mathcal{D}_R(\Lambda)$ over $S_R^*(\Lambda)$. Therefore, by \cref{prop:graded-operator-basis}, the classes $\left[\Delta_{I_w}\right]$ form a basis  of $\mathcal{G}r^*\mathcal{D}_{(R,F)}(\Lambda)$ over $\mathcal{G}r_{(R,F)}^*(\Lambda)$. By (\ref{itm:filt-four}) of \cref{prop:filtration-properties}, $\mathcal{D}_{(R,F)}(\Lambda)$ is Hausdorff. Moreover, $R\llbracket\Lambda\rrbracket_F$ is complete. Therefore, the rest of this proof is the same as the proof of \cite[Lem. 4.11]{CPZ}.
\end{proof}

\begin{rem}
	We can view $\mathcal{D}_{R}(\Lambda)$ as the subalgebra of the endomorphism algebra of $(S_R^*(\Lambda))^\wedge$ generated by the Demazure operators $\Delta_\alpha^R$ and by multiplication by elements in $S_R^*(\Lambda)$.
	If $(R,F)=(R,F_a)$ is the additive formal group law over $R$, then, through the isomorphism $\phi:R\llbracket\Lambda\rrbracket_{F_a}\to (S_R^*(\Lambda))^\wedge$ of \cref{ex:additiveFGL}, we can view $\mathcal{D}_{(R,F_a)}(\Lambda)$ as the completion of $\mathcal{D}_R(\Lambda)$ at the kernel of the map $\mathcal{D}_R(\Lambda)\to R$, sending $\alpha^*\mapsto 0$ for all $\alpha\in\Sigma$. Here, we use $\alpha^*$ to denote multiplication by $\alpha$.
\end{rem}

\details{Let $A$ be a ring, and let $\mathcal{I}$ be an ideal of $A$ such that $A$ is complete and Hausdorff for the $\mathcal{I}$-adic topology. Let $M$ be an $A$-module together with an $\mathcal{I}$-filtration $(M^{(i)})_{i\in\mathbb{Z}}$. The filtration is called \textit{exhaustive} if $M=\cup_iM^{(i)}$. Consider the associated graded ring $\mathcal{G}r^*A$, and the associated graded $\mathcal{G}r^*A$-module $\mathcal{G}r^*M$. For any exhaustive filtration, let $\nu(\lambda)$ denote the largest integer such that $\lambda\in M^{\nu(\lambda)}$, and let $\overline{\lambda}$ denote the class of $\lambda$ in $M^{(\nu(\lambda))}/M^{(\nu(\lambda)+1)}$ if $\nu(\lambda)$ exists and $0$ if $\nu(\lambda)$ does not exist.
	
	\begin{lem}{\textup{(see} \cite[Lem. 4.10]{CPZ}\textup{)}}
		Assume $A$ be a Hausdorff complete ring, let $M$ is a Hausdordff $A$-module with an exhaustive $\mathcal{I}$-filtration, and let $\lambda_1,\dotsc,\lambda_t$ be elements of $M$. If $\mathcal{G}r^*M$ is a free finitely-generated $\mathcal{G}r^*A$-module with basis $\{\overline{\lambda}_i\}_{i=1}^t$, then $M$ is a free finitely-generated $A$-module with basis $\{\lambda_i\}_{i=1}^t$.
	\end{lem}\details}

\details{Let $D(\Lambda)$ be the subalgebra of $R$-linear of $R$-linear endomorphisms of $\mathcal{S}$ generated by the $\Delta_\alpha$ for all roots $\alpha\in\Sigma$ and by multiplication by elements of $\mathcal{S}.$ Note that by (\ref{uno}) of \cref{prop:various-relations}, $D(\Lambda)$ contains $s_\alpha.$ Let $D(\Lambda)^{(i)}$ be the $\mathcal{S}$-module of $D(\Lambda)$ generated by the $u\Delta_{\alpha_1}\dotsb\Delta_{\alpha_n},$ where $u\in \mathcal{I}_F^{m}$, such that $m-n\geq i.$ This defines a filtration on $D(\Lambda)$,
	\[D(\Lambda)\supseteq\dotsb\supseteq D(\Lambda)^{(i)}\supseteq D(\Lambda)^{(i+1)} \supseteq \dotsb \supseteq 0. \]

	\begin{prop}{\label{prop:filtration}}{(see \cite[Prop. 4.6, (1)]{CPZ})}
		The filtration on $\mathcal{D}$ has the following property. For any any root $\alpha,$ we have 
		\[\Delta_\alpha D(\Lambda)^{(i)}\subseteq D(\Lambda)^{(i-1)},\qquad D(\Lambda)^{(i)}\Delta_\alpha\subseteq D(\Lambda)^{(i-1)}. \]
	\end{prop}
	\begin{proof}
		\details{\textcolor{Orange}{The proof is unchanged.}}This follows from (\ref{fouro}) of \cref{prop:various-relations}.
	\end{proof}
}

\details{
	\begin{lem}\label{lem:independence}
		Let $u\in R\llbracket\Lambda\rrbracket_{F_a}$, such that $\phi(u)\in S_\mathcal{R}^m(\Lambda)$ is homogeneous of degree $m>0$. Then, for all sequences of simple roots $I$ of length $m$, the element $\Delta_I^{(R,F_a)}(u)\in \mathcal{R}$, and it is independent of $R$.
	\end{lem}
	\begin{proof}
		First note that, for $\alpha\in\Sigma$ and $\lambda\in\Lambda$, the element $\Delta_{\alpha}^{(R,F_a)}(x_\lambda)=\alpha^\vee(\lambda)\in\mathcal{R}$ is independent of $R$. Note also that, for any root $\alpha$, either $\Delta_\alpha^{(R,F_a)}(u)=0$ or the element $\phi(\Delta_\alpha^{(R,F_a)}(u))\in S_\mathcal{R}^{m-1}(\Lambda)$. Now the result follows by induction on the degree of monomials using (\ref{fouro}) of \cref{prop:various-relations}.
	\end{proof}
}

\section{The formal affine Demazure algebra}\label{section:Demazure-elements}
 In this section, we define and discuss the formal affine Demazure algebra $\mathbf{D}_{F}$. All results discussed in this section are extensions of results in \cite{CZZ1} and \cite{HMSZ} to all real finite reflection groups.
 
 By (\ref{itm:same-two}) of \cref{lem:same-properties}, the formal group ring $R\llbracket\Lambda\rrbracket_F$ is an integral domain. Thus, we let $\mathcal{Q}^{(R,F)}_{\Lambda}$ be the localization of $R\llbracket\Lambda\rrbracket_F$ at the multiplicative subset generated by the elements $x_\alpha$ for all $\alpha\in\Sigma$. The ring $\mathcal{Q}^{(R,F)}_{\Lambda}$ is an integral domain since $R\llbracket\Lambda\rrbracket_F$ is an integral domain, and we call $\mathcal{Q}^{(R,F)}_{\Lambda}$ the \textit{localized formal group ring}. As $W$ acts on $R\llbracket\Lambda\rrbracket_F$ by permuting the roots $\alpha\in\Sigma,$ it also acts on $\mathcal{Q}^{(R,F)}_{\Lambda}$. 
 
 \begin{lem}{\label{lem:injective}}{\textup{(see} \cite[Lem. 3.2]{CZZ1}\textup{)}}
 	The localization map $R\llbracket\Lambda\rrbracket_F\to \mathcal{Q}^{(R,F)}_{\Lambda}$ is injective.
 \end{lem}
 \begin{proof}\details{\textcolor{Orange}{The proof is unchanged.}}
 	This follows from the fact that $R\llbracket\Lambda\rrbracket_F$ is an integral domain.
 \end{proof}

 Following \cite[Def. 6.1]{HMSZ} and \cite[\S 4.1]{KK86}, we define the \textit{twisted formal group ring} to be the $R$-module $\mathcal{Q}^{(R,F)}_{\Lambda,W}:=\mathcal{Q}^{(R,F)}_{\Lambda}\otimes_R R[W]$ with multiplication given by
\[(q\delta_w)(q'\delta_{w'})=qw(q')\delta_{ww'},\quad w,w'\in W,\text{ } q,q'\in \mathcal{Q}^{(R,F)}_{\Lambda},\]
and extended by $R$-linearity. Here, $\delta_w$ is the element in $R[W]$ corresponding to $w\in W$ (so that $\delta_w\delta_{w'}=\delta_{ww'}$ for all $w,w'\in W$). We will denote the identity $\mathbf{1}=\delta_1$. The set $\{\delta_w\}_{w\in W}$ is the canonical basis of the group ring $R[W]$, and, hence, it is a basis of $\mathcal{Q}^{(R,F)}_{\Lambda,W}$ as a left $\mathcal{Q}^{(R,F)}_{\Lambda}$-module.

Consider the $R$-subalgebra, $D_{R\llbracket\Lambda\rrbracket_F}$, of $\mathcal{Q}^{(R,F)}_{\Lambda,W}$ generated by the elements
\[X_\alpha^{(R,F)}=\tfrac{1}{x_\alpha}(\mathbf{1}-\delta_{\alpha})\text{ for all roots }\alpha\in\Sigma,\]
where $\delta_\alpha=\delta_{s_\alpha}$ for all $\alpha\in\Sigma$.
Following \cite[Def. 6.2 and Def. 6.3]{HMSZ}, we call the elements $X_\alpha^{(R,F)}$, $\alpha\in\Sigma$, \textit{formal Demazure elements} and
we call $D_{R\llbracket\Lambda\rrbracket_F}$ the \textit{formal Demazure algebra}. We call the $R$-subalgebra of $\mathcal{Q}^{(R,F)}_{\Lambda,W}$ generated by $D_{R\llbracket\Lambda\rrbracket_F}$ and $R\llbracket\Lambda\rrbracket_F$ the \textit{formal affine Demazure algebra} and denote it $\mathbf{D}_{R\llbracket\Lambda\rrbracket_F}$. Often, we denote the formal Demazure elements by $X_\alpha$, ommiting the superscript, and we denote the formal affine Demazure algebra by $\mathbf{D}_{F}$.
For any simple root $\alpha_i\in\Delta$, we define $X_i:=X_{\alpha_i}$ and $\delta_i:=\delta_{s_{\alpha_i}}$. For any sequence of simple roots $I=(\alpha_{i_1},\dotsc,\alpha_{i_r})$, we will use the the notation $\delta_I:=\delta_{i_1}\dotsb\delta_{i_r}$ and $X_I:=X_{i_1}\dotsb X_{i_r}$.

\begin{lem}{\textup{(see} \cite[Lem. 5.8]{CZZ1}\textup{)}}
	The formal affine Demazure algebra $\mathbf{D}_{F}$ coincides with the $R$-subalgebra of $\mathcal{Q}_W^{(R,F)}$ generated by the elements of $R\llbracket\Lambda\rrbracket_F$ and the formal Demazure elements $X_i$, $i\in\{1,\dotsc,n\}$.
\end{lem}
\begin{proof}
	By \cref{lem:root-conjugation}, any root $\alpha\in\Sigma$ can be written $\alpha=w(\alpha_i)$ for some $\alpha_i\in\Delta$ and $w\in W$. Now this result follows by the proof \cite[Lem. 5.8]{CZZ1}.
\end{proof}

\begin{lem}{\label{lem:coeff_in_Q}}{\textup{(see} \cite[Lem.~5.4]{CZZ1}\textup{)}}
	Given a reduced sequence $I_v$ of $v\in W$ of length $l(I_v),$ let
	\[X_{I_v}=\sum\limits_{w\in W}a_{v,w}\delta_w\]
	for some $a_{v,w}\in\mathcal{Q}_\Lambda^{(R,F)}.$ Then
	\begin{enumerate}[label=(\alph*)]
		\item\label{itm:a} $a_{v,w}=0$ unless $w\leq v$ with respect to the Bruhat order on $W;$
		\item\label{itm:b} $a_{v,v}=(-1)^{l(I_v)}\prod\limits_{\alpha\in v(\Sigma^-)\cap\Sigma^+}x_{\alpha}^{-1}$.
	\end{enumerate}
\end{lem}
\begin{proof} This proof is the same as the proof of \cite[Lem.~5.4]{CZZ1}; however, the proof of \cite[Lem.~5.4]{CZZ1} uses the results \cite[Th. 1.1, III, (ii)]{D77} and \cite[Ch. VI, \S 1, No 6, Cor. 2]{B68}, which are properties of the Weyl group. Since we are working with real finite reflection groups, we provide updated references. The result \cite[Th. 1.1, III, (ii)]{D77} can be replaced by \cref{thm:subword}, and instead of \cite[Ch. VI, \S 1, No 6, Cor. 2]{B68}, we use \cref{prop:positive-roots}. \details{\textcolor{Red}{Same proof. Some references changed.}
	We proceed by induction on the length $l$ of $v.$
	The lemma obviously holds when $I_v=\emptyset,$ the empty sequence, since $X_\emptyset=\mathbf{1}.$
	
	Let $I_v=(\alpha_{i_1},\dotsc,\alpha_{i_r})$ be a reduced sequence of $v$, and let $\beta=\alpha_{i_1}.$ Then $(\alpha_{i_2},\dotsc,\alpha_{i_r})$ is a reduced sequence of $v'=s_{\beta}(v)$, and we have
	\begin{enumerate}
		\item\label{itm:1} $w\leq v'$ implies $w\leq v$ and $s_\beta(w)\leq v,$
		\item\label{itm:2} $\{\beta\}\cup s_\beta(v'(\Sigma_-)\cap\Sigma_+)=v(\Sigma_-)\cap\Sigma_+,$
		\item\label{itm:3} $w^{-1}\leq v$ if and only if $w\geq v^{-1}.$
	\end{enumerate}
	The properties (\ref{itm:1}) and (\ref{itm:3}) are consequences of the fact that elements smaller than $v$ are the elements $w$ obtained by taking a subsequence of a reduced sequence of $v,$ by \textcolor{red}{(new reference:)} \cite[Prop. 3.1.2]{BB05}. Property (\ref{itm:2}) follows from \textcolor{red}{[Cannot find this reference.]}
	
	So we compute
	\begin{align*}
	X_{I_v}&=\tfrac{1}{x_{\beta}}(\mathbf{1}-\delta_\beta)\sum\limits_{w\leq v'}a_{v',w}\delta_w=\sum\limits_{w\leq v'}x_\beta^{-1}a_{v',w}\delta_w-\sum\limits_{w\leq v'}x_\beta^{-1}s_\beta(a_{v',w})\delta_{s_\beta(w)}\\
	&\mathop{=}\limits^{(\ref{itm:2})}x_{\beta}^{-1}s_\beta(a_{v',v})\delta_v+\sum\limits_{w<v}a_{v,w}\delta_w.
	\end{align*}
	Therefore, part (\ref{itm:a}) holds and $a_{v,v}=-x_\beta^{-1}s_\beta(a_{v',v'}).$ Hence the expression $a_{v,v}$ in part (\ref{itm:b}) follows from (\ref{itm:2}). Property (\ref{itm:3}) implies (\ref{itm:c}) by applying the anti-involution sending $\delta_w$ to $\delta_{w^{-1}}$, thus identifying $a_{v,w}'=a_{v,w^{-1}}.$
}
\end{proof}

\begin{rem}{\label{rem:sums-of-products}}
	By definition, we have $X_\alpha=\tfrac{1}{x_\alpha}(1-\delta_\alpha)$. Therefore, since $W$ permutes the roots, the $a_{v,w}$ of \cref{lem:coeff_in_Q} are sums of products of elements of the form $\tfrac{1}{x_\alpha}$, where $\alpha\in\Sigma$.
\end{rem}

\begin{cor}{\label{cor:base-change}}{\textup{(see} \cite[Cor. 5.6]{CZZ1}\textup{)}}
	Choose a reduced sequence $I_w$ for each $w\in W$. The elements $\{X_{I_w}\}_{w\in W}$ form a basis of $\mathcal{Q}^{(R,F)}_{\Lambda,W}$ as a left $\mathcal{Q}^{(R,F)}_{\Lambda}$-module, and the element $\delta_v$ decomposes as $\sum_{w\leq v}b_{w,v}X_{I_w}$ with $b_{w,v}\in\mathcal{Q}^{(R,F)}_{\Lambda}$. Furthermore,
	\[b_{v,v}=(-1)^{l(I_v)}\prod_{\alpha\in v(\Sigma^-)\cap\Sigma^+}x_\alpha.\]
\end{cor}
\begin{proof}
	Since \cref{lem:coeff_in_Q} holds, this proof is the same as the proof of \cite[Cor. 5.6]{CZZ1}.
\end{proof}
\details{\begin{rem}
	Observe that the coefficient $a_{v,v}$ does not depend on the choice of reduced sequence $I_v$ of $v.$
\end{rem}}

\details{We state the following corollary because it is used in the proof of \cref{lem:relation-in-S}.
\begin{cor}{\label{cor:coeff_in_Q2}}{(see \cite[Cor. 5.6]{CZZ1})}
	For each $w\in W$, let $I_w$ be a reduced sequence of $w.$ The elements $(X_{I_w})_{w\in W}$ form a basis of $\mathcal{Q}_W$ as a left (resp. right) $\mathcal{Q}$-module. The element $\delta_v$ decomposes as $\sum\limits_{w\leq v}b_{w,v}X_{I_w}$ with $b_{w,v}$ in $\mathcal{Q}.$ Furthermore,
	\[b_{v,v}=(-1)^l\prod\limits_{\alpha\in v(\Sigma_-)\cap \Sigma_+}x_\alpha.\]
\end{cor}
	\begin{proof}The proof is unchanged. It uses \cref{lem:coeff_in_Q}.\details{
	\textcolor{Orange}{The proof is unchanged.} The matrix $(a_{v,w})_{(v,w)\in W^2}$ (resp. $(a_{v,w}')_{(v,w)\in W^2}$) is triangular with invertible coefficients on the diagonal (resp. anti-diagonal) by \cref{lem:coeff_in_Q}. It expresses the elements $X_{I_w}$ in terms of the basis $(\delta_w)_{w\in W}$ of the left (resp. right) $\mathcal{Q}$-module $\mathcal{Q}$.
	
	The decomposition of $\delta_w$ follows from the fact that the inverse of a triangular matrix is also triangular with inverse coefficients along the diagonal.}
\end{proof}}

The proofs of \cref{lem:one} and \cref{lem:square} are direct computations:
\begin{lem}\label{lem:one} \textup{(see} \cite[Lemma 6.5]{HMSZ}\textup{)} For all $\alpha\in\Sigma$, we have the following commuting relation in $\mathcal{Q}^{(R,F)}_{\Lambda,W}$,
	\[
	X_\alpha q =s_\alpha(q) X_\alpha +\Delta_\alpha(q),\quad q\in\mathcal{Q}^{(R,F)}_{\Lambda}.
	\]
	Here, $\Delta_\alpha(q)$ denotes the element $\tfrac{q-s_\alpha(q)}{x_\alpha}\in \mathcal{Q}^{(R,F)}_{\Lambda}$. By \cref{lem:divisible}, if $q\in R\llbracket\Lambda\rrbracket_F$, then $\Delta_\alpha(q)\in R\llbracket\Lambda\rrbracket_F$. Hence, the formula above gives a relation in $\mathbf{D}_{F}$.
\end{lem}

\begin{lem}{\label{lem:square}}{\textup{(see} \cite[Eq. (6.1)]{HMSZ}\textup{)}} For any $\alpha\in \Sigma$ we have 
	\[
	X_\alpha^2=\kappa_{\alpha} X_\alpha,\text{ where }
	\kappa_{\alpha}=\tfrac{1}{x_{\alpha}}+s_\alpha\left(\tfrac{1}{x_{\alpha}}\right).
	\]
	Here, $\kappa_{\alpha}\in R\llbracket\Lambda\rrbracket_F$ by the argument in \cite[Def. 3.12]{CPZ}.
\end{lem}

\begin{rem}
	If $(R,F)$ is a formal group law of additive type, then the $\kappa_{\alpha}$ of \cref{lem:square} is $0$ for each $\alpha\in\Sigma$.
\end{rem}

The following Proposition says that all formal affine Demazure algebras over ample rings containing $\mathcal{R}$ are isomorphic.  Its proof is the same as the proof of \cite[Thm. 7.4]{HMSZ}. We include the proof.

\begin{prop}{\label{prop:iso-of-formal-affine-Demazure-algebras}}{\textup{(see} \cite[Thm. 7.4]{HMSZ}\textup{)}}
	 Assume $R$ is an ample ring with respect to a formal group law $(\mathbb{C},F)$, and that $R$ contains $\mathcal{R}$. There is an $R$-algebra isomorphism $\mathbf{D}_{R\llbracket\Lambda\rrbracket_F}\simeq \mathbf{D}_{R\llbracket\Lambda\rrbracket_{F_a}}$.
\end{prop}
\begin{proof}
	The $W$-equivariant continuous ring isomorphism $\widetilde{\mathrm{exp}}_F^*:R\llbracket\Lambda\rrbracket_F\to  R\llbracket\Lambda\rrbracket_{F_a}$ of formal group rings sends $x_\alpha\mapsto \mathrm{exp}_F(x_\alpha)=x_\alpha g(x_\alpha)$, where $g(u)\in R\llbracket u\rrbracket$ is the invertible power series defined in the proof of \cref{lem:mixed-divisible}. Thus, $\widetilde{\mathrm{exp}}_F^*$ induces an isomorphism of twisted formal group rings $\widetilde{\mathrm{exp}}_F^*:\mathcal{Q}^{(R,F) }_{\Lambda}\to \mathcal{Q}^{(R,F_a)}_{\Lambda}$, and $\mathbf{D}_{R\llbracket\Lambda\rrbracket_{F}}$ is isomorphic to its image $D':=\mathrm{exp}_F^*(\mathbf{D}_{R\llbracket\Lambda\rrbracket_{F_a}})$ under this map. Now, $D'$ is generated over $R\llbracket\Lambda\rrbracket_{F_a}$ by the elements 
	\[\mathrm{exp}_F^*\left(\Delta_i^{(R,F)}\right)=\tfrac{1}{\mathrm{exp}_F(x_{\alpha_i})}(1-\delta_i)=\tfrac{x_{\alpha_i}}{\mathrm{exp}_F(x_{\alpha_i})}\Delta_{i}^{(R,F_a)},\]
	where $\alpha_i\in \Delta$.
	Since $\tfrac{\mathrm{exp}_F(x_{\alpha_i})}{x_{\alpha_i}}\in R\llbracket\Lambda\rrbracket_{F_a}$ has constant term $1$, it is invertible in $R\llbracket\Lambda\rrbracket_{F_a}$. Therefore, $D'$ is generated over $R\llbracket\Lambda\rrbracket_{F_a}$ by the elements $\Delta_{i}^{(R,F_a)}$, $\alpha_i\in\Delta$, and is, thus, isomorphic to $\mathbf{D}_{R\llbracket\Lambda\rrbracket_{F_a}}$.
\end{proof}

\section{Presentation in terms of generators and relations}\label{section:presentation}

In this section, we prove a presentation for the formal affine Demazure algebra in terms of formal Demazure elements. This section closely follows \cite[\S 5.2]{CPZ}, \cite[\S 6 and \S 7]{CZZ1}, and \cite[\S 6]{HMSZ}.

The arguments used from here up to and including the proof of \cref{lem:invariance} closely follow the arguments used in \cite[\S 5.2]{CPZ}. Recall that we can view $S^*_\mathbb{C}(\Lambda)$ as the symmetric algebra of the complex vector space $V$ generated by the roots $\alpha\in\Sigma.$ Let $w_0$ be the longest element of $W$, and let $N=l(w_0).$ By \cite[Ch. IV, Prop. 1.6 and Prop. 1.7]{H82} (see also the proof of \cite[Ch. IV, Thm. 1.10]{H82}), the element $d_0:=\prod\limits_{\beta\in \Sigma^+}\beta\in S^N_R(\Lambda)$ satisfies $\Delta_{I_{w_0}}^{\mathbb{C}}(d_0)=|W|$ for any reduced sequence $I_{w_0}$ of $w_0$. Since each root $\beta$ is an $\mathcal{R}$-linear combination of the simple roots, it follows from \cref{lem:symmetric-sum-independent-of-R} that $\Delta_{I_{w_0}}^{R}(d_0)=|W|$ as well.

Through the isomorphism $\psi\colon S_R^*(\Lambda)\to \mathcal{G}r_{(R,F)}^*(\Lambda)$ of \cref{lem:additive-iso}, we choose an element $u_0\in \mathcal{I}_{F}^{N}/\mathcal{I}_{F}^{N+1}$ with $u_0=\psi(d)$. Therefore, by \cref{prop:additive-reduction}, for any reduced sequence $I_{w_0}$ of $w_0$, we have that $\Delta_{I_{w_0}}^{(R,F)}(u_0)=\Delta_{I_{w_0}}^R(d)+\mathcal{I}_{F}$. Hence, $\epsilon\Delta_{I_{w_0}}^{(R,F)}(u_0)=\Delta_{I_{w_0}}^R(d)=|W|$ for any reduced sequence $I_{w_0}$ of $w_0$.

The proof of \cref{lem:invariance} is the same as the proof of \cite[Lemma 5.3 (3)]{CPZ}. We include the proof of \cref{lem:invariance} because it refers to several facts in \cite{CPZ} that we have shown hold for all real finite reflection groups as well. The element $u_0$ satisfies the following property:
\begin{lem}{\textup{(cf.} \cite[Lemma 5.3 (3)]{CPZ}\textup{)}}{\label{lem:invariance}}
	If $I$ is any sequence of simple roots, and $1\leq l(I)\leq N$, then 
	\[\epsilon\Delta_{I}(u_0)=\begin{cases}
	|W|,\quad \text{\textup{if $I$ is reduced and $l(I)=N$}};\\
	0,\quad\quad\textup{otherwise.}
	\end{cases}\]
\end{lem}
\begin{proof}
For $k>0$, \cref{lem:ideal-grading-reduction} implies that the operator $\Delta_{I}$ sends $\mathcal{I}_{F}^{k}$ to  $\mathcal{I}_{F}^{k-l(I)}$. 
Thus, if $l(I)<N$, then $\epsilon\Delta_{I}(u_0)\in \epsilon\left(\mathcal{I}_{F}^{N-l(I)}\right)=\{0\}$.

If $l(I)= N$ and $I$ is not reduced, then  $\epsilon\Delta_{I}(u_0)=0$. To see this, it suffices to show that, in this case, $\Delta_{I}^{F}$ sends $\mathcal{I}_{F}^{k}$ to  $\mathcal{I}_{F}^{k-l(I)+1}$. Thus, it is enough to show that $\mathcal{G}r\Delta_I=0$, which, by \cref{prop:additive-reduction}, is equivalent to the statement $\Delta_I^R=0$. The fact that $\Delta_I^R=0$ when $I$ is not reduced follows from \cref{prop:product-trivial2}.

If $l(I)=N$ and $I$ is reduced, then $\epsilon\Delta_{I}(u_0)=|W|$ by definition.
\end{proof}

In \cite{CPZ, CZZ1}, the torsion index $\mathfrak{t}$ of \cite{D73} serves the same role as the element $|W|$ does in this paper. For each $w\in W$, let $I_w$ be a reduced sequence of $w$. With the torsion index $\mathfrak{t}$ replaced by the element $|W|$, the following result goes through:
\begin{lem}\label{lem:matrix-invertible}{\textup{(see} \cite[Lem. 6.1]{CZZ1}\textup{)}}
	 Suppose $|W|$ is invertible in $R$. Then the matrix $\left(\Delta_{I_v}\Delta_{I_w}(u_0)\right)_{(v,w)\in W\times W}$ with coefficients in $R\llbracket\Lambda\rrbracket_{F}$ is invertible.
\end{lem}
\begin{proof}
	By hypothesis, the element $|W|$ is invertible in $R$. In addition, \cref{lem:invariance} holds.
	Thus, this result follows by the proof of \cite[Prop. 6.6]{CPZ}.
\end{proof}

The arguments used from here up to and including the proof of \cref{thm:basis-theorem} closely follow the arguments used in \cite[\S6 and \S7]{CZZ1} and \cite[\S 6]{HMSZ}. Let $\phi:\mathcal{Q}_{\Lambda,W}^{(R,F)}\to \textrm{End}_R(\mathcal{Q}_{\Lambda}^{(R,F)})$ be the $R$-algebra homomorphism induced by the left action of $\mathcal{Q}_{\Lambda,W}^{(R,F)}$ on $\mathcal{Q}_{\Lambda}^{(R,F)}$. By definition, $R\llbracket\Lambda\rrbracket_F$ acts on the left on both algebras, $\phi$ is $R\llbracket\Lambda\rrbracket_{F}$-linear, and $\phi(X_\alpha)=\Delta_\alpha$. Let $\phi_{\mathbf{D}_{F}}$ be the restriction of $\phi$ to $\mathbf{D}_{F}$. Note that the image of $\phi_{\mathbf{D}_F}$ is $\mathcal{D}_{(R,F)}(\Lambda)$, since the $R$-algebra generators of $\mathbf{D}_{F}$ map to the $R$-algebra generators of $\mathcal{D}_{(R,F)}(\Lambda)$.

\begin{thm}{\label{thm:iso-with-operators}}{\textup{(see} \cite[Thm. 7.10]{CZZ1}\textup{)}}
	The map $\phi_{\textbf{D}_{F}}:\mathbf{D}_{F}\to\mathcal{D}_{(R,F)}(\Lambda)$ is a left $R\llbracket \Lambda\rrbracket_F$-algebra isomorphism.
\end{thm}
\begin{proof}
Since \cref{cor:base-change}, and \cref{lem:matrix-invertible} hold, this proof is the same as the proof of \cite[Thm. 7.10]{CZZ1}.
\end{proof}

\begin{cor}{\label{cor:basis-for-Demazure-elements}}
	The $R$-algebra $\mathbf{D}_{F}$ is free as a left $R\llbracket\Lambda\rrbracket_{F}$-submodule of $\mathcal{Q}^{(R,F)}_{\Lambda,W}$ with basis $\{X_{I_w}\}_{w\in W}$.
\end{cor}
\begin{proof}
	This follows from \cref{thm:operator-basis} and \cref{thm:iso-with-operators}.
\end{proof}

\begin{lem}{\label{lem:relation-in-S}}{\textup{(see} \cite[Lem. 7.1]{CZZ1}\textup{)}}
	Let $I$ and $I'$ be reduced sequences of the same element $w\in W$. Then in the formal affine Demazure algebra $\mathbf{D}_{F},$ we have
	\[X_I-X_{I'}=\sum\limits_{v< w}c_v X_{I_v}\quad \text{for some $c_v\in R\llbracket\Lambda\rrbracket_{F}$,}\]
	where the ordering $v<w$ is with respect to the Bruhat order on $W$.	
\end{lem}
\begin{proof}
	By \cref{cor:basis-for-Demazure-elements}, the difference $X_I-X_{I'}=\sum_{v\in W}c_vX_{I_v}$ for some $c_v\in R\llbracket\Lambda\rrbracket_F$. Since \cref{lem:injective}, \cref{lem:coeff_in_Q}, and \cref{cor:base-change} hold, the rest of this proof is the same as the proof of \cite[Lem. 7.1]{CZZ1}.
\end{proof}

\begin{rem}\label{rem:equality-in-additive-demazure-algebra}
	Let $I$ and $I'$ be reduced sequences of the same element $u\in W$. Then \cref{prop:product-trivial} and \cref{thm:iso-with-operators} imply that, in $\mathbf{D}_{R\llbracket\Lambda\rrbracket_{F_a}}$, we have $X_I^{(R,F_a)}=X_{I'}^{(R,F_a)}$.
\end{rem} 

\begin{rem}\label{rem:equality-in-multiplicative-demazure-algebra}
	Let $I$ and $I'$ be reduced sequences of the same element $u\in W$. Suppose $\Sigma$ is crystallographic. By \cite[Thm. 3.10]{CPZ}, in $\mathbf{D}_{R\llbracket\Lambda\rrbracket_{F_m}}$, we have $X_I^{(R,F_m)}=X_{I'}^{(R,F_m)}$. Note that \cite[Thm. 3.10]{CPZ} refers to the proof of \cite[Thm. 2, pp. 86]{D74}, which relies on the geometry of flag varieties. If $\Sigma$ is noncrystallographic, we leave it open to determine whether $X_I^{(R,F_m)}=X_{I'}^{(R,F_m)}$ in $\mathbf{D}_{R\llbracket\Lambda\rrbracket_{F_m}}$.
\end{rem}

\begin{rem}{\label{rem:unique-reduced-subword}}{\textup{(cf. }\cite[Ex. 7.3]{CZZ1}\textup{)}}
	Suppose $\alpha_i$ and $\alpha_j$ are simple roots, and $m_{i,j}$ is the order of $s_is_j$ in $W$. Set $w_0^{i,j}=\underbrace{s_is_js_i\dotsb}_{\text{$m_{i,j}$ times}}=\underbrace{s_js_is_j\dotsb}_{\text{$m_{i,j}$ times}}$. If $\Sigma$ is a root system of rank $2$, then $w_0^{1,2}$ is the unique element in $W$ with more than one reduced decomposition; $w_0^{1,2}$ has exactly two reduced sequences $(\alpha_1,\alpha_2,\alpha_1,\dotsc)$ and $(\alpha_2,\alpha_1,\alpha_2,\dotsc)$, both of length $m_{1,2}$; and $w_0^{1,2}$ is the longest word in $W$. 
\end{rem}

\begin{thm}{\textup{(see} \cite[Thm. 7.9]{CZZ1} \textup{and} \cite[Thm. 6.14]{HMSZ}\textup{)}}\label{thm:basis-theorem}
	Let $\Sigma$ be a root system of a real finite reflection group $W$. Let $\mathcal{Q}^{(R,F)}_\Lambda$ denote the localization of $R\llbracket\Lambda\rrbracket_{F}$ at the elements $x_\alpha,$ $\alpha\in\Sigma.$ Given a set of simple roots $\{\alpha_1,\dotsc,\alpha_n\}$ with associated simple reflections $\{s_1,\dotsc,s_n\}$, let $m_{i,j}$ denote the order of the product $s_{i}s_j$ in $W$. 	
	The elements $q\in\mathcal{Q}^{(R,F)}_\Lambda$ (resp. $q\in R\llbracket\Lambda\rrbracket_{F}$) and the formal Demazure elements $X_i$ satisfy the following relations for all $i,j=1,\dotsc,n$:
	\begin{enumerate}
		\item\label{itm:main-1} $X_iq=\Delta_i(q)+s_i(q)X_i;$
		\item\label{itm:main-2} $X_i^2=\kappa_{\alpha_i}X_i$, where
		$\kappa_{\alpha_i}=\tfrac{1}{x_{\alpha_i}}+s_i\left(\tfrac{1}{x_{\alpha_i}}\right);$
		\item\label{itm:main-3} $\underbrace{X_iX_jX_i\dotsb}_{\text{$m_{i,j}$ times}}-\underbrace{X_jX_iX_j\dotsb}_\text{{$m_{i,j}$ times}}=\sum\limits_{w<w_0^{i,j}}\eta_w^{i,j}X_{I_w}$, $\quad $$\eta_w^{i,j}\in R\llbracket\Lambda\rrbracket_{F}$.
	\end{enumerate}
	Here, $w_0^{i,j}$ is defined in \cref{rem:unique-reduced-subword}, and the ordering $w<w_0^{i,j}$ is with respect to the Bruhat order on $W$.
	These relations, together with the ring law in $R\llbracket\Lambda\rrbracket_{F}$ and the fact that the $X_i$ are $R$-linear, form a complete set of relations in the localized twisted formal group ring $Q^{(R,F)}_{\Lambda,W}$ (resp. the formal affine Demazure algebra $\mathbf{D}_{F}$). 
\end{thm}
\begin{proof}
	The three relations hold by \cref{lem:one}, \cref{lem:square}, and \cref{lem:relation-in-S}. Since \cref{cor:base-change} and \cref{cor:basis-for-Demazure-elements} hold and $W$ is a Coxeter group (see \cref{thm:generation-by-simple-reflections}), the proofs of the presentations of $\mathcal{Q}_{\Lambda,W}^{(R,F)}$ and $\mathbf{D}_F$ are similar to the proof of \cite[Thm. 6.14]{HMSZ}.\qedhere
\end{proof}

\details{
Let $(\mathbb{C},F)$ be a formal group law over $\mathbb{C}$.
We now give alternate proofs of \cref{cor:basis-for-Demazure-elements} and \cref{lem:relation-in-S} when $R$ is an ample ring with respect to $(\mathbb{C},F)$, containing $\mathcal{R}$, and $\Sigma$ is a root system of rank $2$. Fix a reduced sequence $I_w$ for each $w\in W$. Since $\Sigma$ has rank $2$, if $w$ is not the longest word in $W$, then $I_w$ is the unique reduced sequence for $w$.

\begin{lem}{\label{lem:relation-without-inverting}}
	Assume $R$ is an ample ring with respect to $(\mathbb{C},F)$, containing $\mathcal{R}$, and $\Sigma$ is a root system of rank $2$. In $\mathbf{D}_{F}$, we can write
	\[\underbrace{X_1X_2X_1\dotsb}_{\text{$m_{1,2}$ times}}-\underbrace{X_2X_1X_2\dotsb}_\text{{$m_{1,2}$ times}}=\sum\limits_{w<w_0^{1,2}}\eta_w^{1,2}X_{I_w}, \quad \eta_w^{1,2}\in R\llbracket\Lambda\rrbracket_{F},\]
	where $m_{1,2}$ is the order of $s_1s_2$ in $W$.
\end{lem}
\begin{proof}
	\cref{prop:iso-of-formal-affine-Demazure-algebras} and \cref{rem:equality-in-additive-demazure-algebra} imply that 
	\[\underbrace{(\mu_{1}X_{1})(\mu_2 X_2)(\mu_{1}X_{1})\dotsb}_{\text{$m_{1,2}$ times}} =\underbrace{(\mu_2 X_2)(\mu_{1}X_{1})(\mu_2 X_2)\dotsb}_{\text{$m_{1,2}$ times}}\]
	in $\mathbf{D}_{F}$, where the $\mu_{1}$, $\mu_{2}\in R\llbracket\Lambda\rrbracket_F$ are defined in the proof of \cref{prop:iso-of-formal-affine-Demazure-algebras}. By \cref{lem:one}, \cref{lem:square}, and \cref{rem:unique-reduced-subword}, we can write 
	\begin{align*}\underbrace{(\mu_{1}s_1(\mu_2)s_{1}s_2(\mu_1)\dotsb)}_{\text{$m_{1,2}$ factors}}& \underbrace{X_{1}X_2X_1\dotsb}_{\text{$m_{2,1}$ times}}-\underbrace{(\mu_{1}s_2(\mu_1)s_{2}s_1(\mu_2)\dotsb)}_{\text{$m_{1,2}$ factors}} \underbrace{X_{2}X_1X_2\dotsb}_{\text{$m_{1,2}$ times}}=\sum\limits_{w<w_0^{i,j}}\Omega_w^{i,j}X_{I_w}, \end{align*}
	where the $\Omega_w^{1,2}\in R\llbracket\Lambda\rrbracket_F$. By \cref{prop:positive-roots}, \[\Theta_{1,2}:=\underbrace{(\mu_{1}s_1(\mu_2)s_{1}s_2(\mu_1)\dotsb)}_{\text{$m_{1,2}$ factors}}=\underbrace{(\mu_{2}s_2(\mu_1)s_{2}s_1(\mu_2)\dotsb)}_{\text{$m_{1,2}$ factors}}.\]
	Since $\mu_1$ and $\mu_2$ are invertible in $R\llbracket\Lambda\rrbracket_F$, so is $\Theta_{1,2}$. Now the result follows, with $\eta_w^{1,2}=\tfrac{\Omega_w^{1,2}}{\Theta_{1,2}}\in R\llbracket\Lambda\rrbracket_F$.
\end{proof}
\begin{cor}{\label{cor:relation-in-S}}{\textup{(see} \cite[Lem. 7.1]{CZZ1}\textup{)}}
	 Assume $R$ is an ample ring with respect to $(\mathbb{C},F)$, containing $\mathcal{R}$, and $\Sigma$ is a root system of rank $2$. Let $I$ and $I'$ be reduced sequences of the same element $w\in W$. Then, in the formal affine Demazure algebra $\mathbf{D}_{F}$, we have
	\[X_I-X_{I'}=\sum\limits_{v< w}c_v X_{I_v}\quad \text{for some $c_v\in R\llbracket\Lambda\rrbracket_{F}$,}\]
	where the ordering $v<w$ is with respect to the Bruhat order on $W$.	
\end{cor}
\begin{proof}
By \cref{rem:unique-reduced-subword}, $w_0^{i,j}$ has exactly two reduced sequences, and it is the only element in $W$ with more than one reduced sequence. Thus, this result follows from \cref{lem:relation-without-inverting}.
\end{proof}
\begin{lem}{\label{lem:subalgebra-relation}}
	 Assume $R$ is an ample ring with respect to $(\mathbb{C},F)$, containing $\mathcal{R}$, and $\Sigma$ is a root system of rank $2$. In the formal affine Demazure algebra $\mathbf{D}_{F}$, given any sequence $I$ of simple roots, we are able to write $X_I=\sum_{v\in W}a_{I,I_v}X_{I_v}$ for some $a_{I,I_v}\in R\llbracket\Lambda\rrbracket_{F}$ such that:
	\begin{enumerate}
		\item \label{itm:1.2}If $I$ is a reduced sequence of $w\in W$, then $a_{I,I_v}=0$ unless $v\leq w,$ and $a_{I,I_w}=1$.
		\item \label{itm:2.2}If $I$ is not reduced, then $a_{I,I_v}=0$ for all $v$ such that $l(v)\geq l(I).$
	\end{enumerate}
	Moreover, this decomposition is unique.
\end{lem}
\begin{proof}Since \cref{lem:injective}, \cref{lem:one}, \cref{lem:square}, and \cref{cor:relation-in-S} hold, this proof is the same as the proof of \cite[Lem. 7.6]{CZZ1}; however the proof of \cite[Lem. 7.6]{CZZ1} uses the reference \cite[Ch. IV, \S 1, no 5, Prop. 4]{B68}, which is a property of the Weyl group. This property holds for all real finite reflection groups, so we provide an updated reference, \cref{prop:exchange}.
\end{proof}

\begin{prop}{\label{prop:submodule-basis}}{\textup{(see} \cite[Prop. 7.7 and Rem. 7.8]{CZZ1}\textup{)}}
	 Assume $R$ is an ample ring with respect to $(\mathbb{C},F)$, containing $\mathcal{R}$, and $\Sigma$ is a root system of rank $2$. The $R$-algebra $\mathbf{D}_{F}$ is free as a left $R\llbracket\Lambda\rrbracket_{F}$-submodule of $\mathcal{Q}^{(R,F)}_{\Lambda,W}$ with basis $\{X_{I_w}\}_{w\in W}$.
\end{prop}
\begin{proof}
	Since \cref{lem:injective}, \cref{lem:one}, and \cref{lem:subalgebra-relation} hold, this result follows by the proof of \cite[Prop. 7.7]{CZZ1}.
\end{proof}
}

\section{Computations of structure coefficients}\label{section:non-simply-laced}

In the present section, we study the braid relation (\ref{itm:main-3}) of \cref{thm:basis-theorem}. In particular, we give general formulas for several structure coefficients $\eta_w^{i,j}$ that appear in this braid relation, thus generalizing several formulas obtained in \cite[Prop. 6.8]{HMSZ} to all real finite reflection groups.

Before we proceed, we will define some notation that will be used in this section and in \cref{section:applications}.

\begin{nota} 
    For any root $\gamma\in\Sigma,$ set $y_\gamma=\tfrac{1}{x_\gamma}$.
	Let $\{\alpha,\beta\}$ be distinct simple roots, and let $m$ be the order of $s_\alpha s_\beta$ in $W$. Write
	\[w_0^{\alpha,\beta}=\underbrace{s_\alpha s_j s_\alpha \dotsb}_{\text{$m$ times}} = \underbrace{s_\beta s_\alpha s_\beta \dotsb}_{\text{$m$ times}} .\]
	As in the proof of \cref{lem:diheral-property}, we set 
	\begin{align*}\Sigma_{\alpha,\beta}^+:&=\{\alpha,s_\alpha(\beta),s_\alpha s_\beta(\alpha),\dotsc,s_{\alpha,\beta,\dotsc}^{(m-1)}(\alpha)\}=\{\beta,s_\beta(\alpha),s_\beta s_\alpha(\beta),\dotsc,s_{\beta,\alpha,\dotsc}^{(m-1)}(\beta)\},
	\end{align*}
	and we will use the notation
	\[y_{\Sigma^+_{\alpha,\beta}}=\prod\limits_{\gamma\in \Sigma^+_{\alpha,\beta}}y_\gamma.\]

	We define the following notation for products of $i$ formal Demazure elements, $\delta$'s, and simple reflections: 
	\begin{align*}
	X_{\alpha,\beta,\ldots}^{(i)}&=\underbrace{X_\alpha X_\beta X_\alpha\dotsb}_i \ ,\quad X_{\beta,\alpha,\ldots}^{(i)}=\underbrace{X_\beta X_\alpha X_\beta\dotsb}_i, \\
	\delta_{\alpha,\beta,\ldots }^{(i)}&=\underbrace{\delta_\alpha \delta_\beta \delta_\alpha,\dotsb}_i \ \quad\hspace{1mm}, \quad\hspace{1mm}  \delta_{\beta,\alpha,\ldots}^{(i)}=\underbrace{\delta_\beta \delta_\alpha \delta_\beta\dotsb}_i, \\
	s_{\alpha,\beta,\ldots }^{(i)}&=\underbrace{s_\alpha s_\beta s_\alpha\dotsb}_i \ \quad\hspace{1mm}, \quad\hspace{1mm}  s_{\beta,\alpha,\ldots}^{(i)}=\underbrace{s_\beta s_\alpha s_\beta\dotsb}_i,
	\end{align*}
	So, for example,
	\[X_{\alpha,\beta,\ldots}^{(3)}=X_\alpha X_\beta X_\alpha\quad \text{and} \quad s_{\beta,\alpha,\ldots}^{(7)}=s_\beta s_\alpha s_\beta s_\alpha s_\beta s_\alpha s_\beta.\]
	
	We define the following : 
	\[
	\omega_i=\begin{cases}
	\alpha,\quad \textit{i \textup{even}},\\
	\beta,\quad \textit{i \textup{odd}}.
	\end{cases}
	\]
	Let $i=1,\dotsc,m-2$. We define:
	\[\upsilon_{\alpha,\beta}^{(i)}=
	\prod\limits_{j=0}^{m-i-1}s_{\alpha,\beta,\ldots}^{(i)}(y_{\omega_j})
	\ , \quad \upsilon_{\beta,\alpha}^{(i)}=\prod\limits_{j=0}^{m-i-1}s_{\beta,\alpha,\ldots}^{(j)}(y_{\omega_{j+1}}).
	\]
	By \cref{prop:positive-roots}, $\upsilon_{\alpha,\beta}^{(0)}=\upsilon_{\beta,\alpha}^{(0)}=y_{\Sigma^+_{\alpha,\beta}}$.
	
	For $j>i\geq0$, we define the operators $S_{\beta,\alpha}^{(i,j)}:=\sum\limits_{k=i}^j s_{\beta,\alpha,\ldots}^{(k)}$ and $S_{\alpha,\beta}^{(i,j)}:=\sum\limits_{k=i}^j s_{\alpha,\beta,\ldots}^{(k)}$, which act on a root $\gamma\in\Sigma$ as follows:
	\[S_{\beta,\alpha}^{(i,j)}(\gamma)= s_{\beta,\alpha,\ldots}^{(i)}(\gamma)+s_{\beta,\alpha,\ldots}^{(i+1)}(\gamma)+\dotsb+s_{\beta,\alpha,\ldots}^{(j)}(\gamma),\quad\text{and}\]
	\[S_{\alpha,\beta}^{(i,j)}(\gamma)= s_{\alpha,\beta\ldots}^{(i)}(\gamma)+s_{\alpha,\beta,\ldots}^{(i+1)}(\gamma)+\dotsb+s_{\alpha,\beta\ldots}^{(j)}(\gamma).\]
\end{nota}

\

We have the following extension of \cite[Prop. 6.8]{HMSZ} to all real finite reflection groups. The proof is very similar to the proof of \cite[Prop. 6.8]{HMSZ}. We include it because we will refer to the proof several times in this section.

\begin{lem} \label{lem:linear-combo} \textup{(cf.} \cite[Prop. 6.8]{HMSZ}\textup{)}
	The difference $X_{\alpha,\beta,\ldots}^{(m)}-X_{\beta,\alpha,\ldots}^{(m)}$ can be written as a linear combination 
	\begin{equation}\label{eq:difference}
	X_{\alpha,\beta,\ldots}^{(m)}-X_{\beta,\alpha,\ldots}^{(m)}=\sum_{i=1}^{m-2} \big( \kappa_{\alpha,\beta}^{(i)} X_{\alpha,\beta,\ldots}^{(i)} - \kappa_{\beta,\alpha}^{(i)} X_{\beta,\alpha,\ldots}^{(i)} \big),
	\end{equation}
	where $\kappa_{\alpha,\beta}^{(i)}\in R\llbracket\Lambda\rrbracket_F$, $i=1,\dotsc,m-2$.
\end{lem}

\begin{proof}		
	We let $\kappa_{\alpha,\beta}^{(m-1)}$ and $\kappa_{\beta,\alpha}^{(m-1)}$ denote the coefficients of $X_{\alpha,\beta,\ldots}^{(m-1)}$ and $X_{\beta,\alpha,\ldots}^{(m-1)}$ on the right side of \cref{eq:difference}. Similarly, we let $\kappa_{\alpha,\beta}^{(m)}$ and $\kappa_{\beta,\alpha}^{(m)}$ denote the coefficients of $X_{\alpha\ldots}^{(m)}$ and $X_{\beta,\alpha,\ldots}^{(m)}$ on the right side of \cref{eq:difference}. We will show by direct computation that all four of these coefficients equal $0$. 
	
	The fact that $\kappa_{\alpha,\beta}^{(m)}=\kappa_{\beta,\alpha}^{(m)}=0$ follows by the same reasoning as in the proof of \cite[Prop. 6.8]{HMSZ}. We will include the proof because we use it in \cref{cor:linear-combo-2} and \cref{thm:linear-combo2}. Alternatively, $\kappa_{\alpha,\beta}^{(m)}=\kappa_{\beta,\alpha}^{(m)}=0$ by \cref{lem:relation-in-S}.
	
	In the expansion of the product
	\begin{align*}\label{eq:expan}
	X_{\alpha,\beta,\ldots}^{(m)}&=\underbrace{X_\alpha X_\beta X_\alpha \dotsb X_{\omega_{m+1}}}_m=\underbrace{(y_\alpha -y_\alpha \delta_\alpha)(y_\beta -y_\beta \delta_\beta)(y_\alpha -y_\alpha \delta_\alpha)\dotsb(y_{\omega_{m+1}}-y_{\omega_{m+1}}\delta_{\omega_{m+1}})}_m,
	\end{align*}
	and in the expansion of the product $X_{\beta,\alpha,\ldots}^{(m)},$
	the coefficient of $\delta_{w_0^{\alpha,\beta}}=\delta_{\alpha,\beta,\ldots}^{(m)}=\delta_{\beta,\alpha,\ldots}^{(m)}$ is
	\[\pm y_\alpha s_\alpha(y_\beta)\dotsb s_{\alpha,\beta,\ldots}^{(m-1)}(y_{\omega_{m+1}}) = \pm y_\beta s_\beta(y_\alpha)\dotsb s_{\beta,\alpha,\ldots}^{(m-1)}(y_{\omega_{m}})=\pm y_{\Sigma^+_{\alpha,\beta}},\]
	where we use `$+$' when $m$ is even and we use `$-$' when $m$ is odd.
	Hence, in the difference $X_{\alpha,\beta,\ldots}^{(m)}-X_{\beta,\alpha,\ldots}^{(m)}$, the  $\delta_{w_0^{\alpha,\beta}}$-term is 
	\[
	\pm y_{\Sigma^+_{\alpha,\beta}}  \delta_{\alpha,\beta,\ldots}^{(m)} - (\pm y_{\Sigma^+_{\alpha,\beta}}  \delta_{\beta,\alpha,\ldots}^{(m)})=0.\]
	So $\kappa_{\alpha,\beta}^{(m)}=\kappa_{\beta,\alpha}^{(m)}=0$. Next we will show that $\kappa_{\alpha,\beta}^{(m-1)}=\kappa_{\beta,\alpha}^{(m-1)}=0$.
	
	Suppose $m$ is odd. To obtain a $\delta_{\alpha,\beta,\ldots}^{(m-1)}$-term in the expansion of the product $X_{\alpha,\beta,\ldots}^{(m)}$, we must choose the $\delta$-term in each factor, except for the last $X_\alpha$, where we must choose the constant term. If we choose the constant term in an $X_\beta$ factor, we would get cancellations of $\delta_\alpha$, resulting in a word in $\delta_\alpha$ and $\delta_\beta$ of length less than $m-1.$ If we choose the constant term in the first $X_\alpha$, then the word would begin with $\delta_\beta,$ and if we choose a constant term in an $X_\alpha$ factor that is not the first or last factor, we would again get cancellations of $\delta$'s and obtain a word of length less than $m-1$. Hence, the $\delta_{\alpha,\beta,\ldots}^{(m-1)}$-term is 
	\begin{align*}
	(y_\alpha \delta_\alpha) (y_\beta \delta_\beta) \dotsb  (y_\beta \delta_\beta) y_\alpha&=y_\alpha s_\alpha(y_\beta)\dotsb s_{\alpha,\beta,\ldots}^{(m-1)}(y_\beta)\delta_{\alpha,\beta,\ldots}^{(m-1)}=y_{\Sigma^+_{\alpha,\beta}} \delta_{\alpha,\beta,\ldots}^{(m-1)}.
	\end{align*}
	By a similar argument, the $\delta_{\beta,\alpha,\ldots}^{(m-1)}$-term in $X_{\alpha,\beta,\ldots}^{(m)}$ is
	\begin{align*}
	y_\alpha (y_\beta \delta_\beta)(y_\alpha \delta_\alpha) \dotsb (y_\alpha\delta_\alpha)&= y_\alpha y_\beta s_\beta(y_\alpha)\dotsb s_{\beta,\alpha,\ldots}^{(m-2)}(y_\alpha)\delta_{\beta,\alpha,\ldots}^{(m-1)}=c\delta_{\beta,\alpha,\ldots}^{(m-1)}.
	\end{align*}
	By \cref{lem:diheral-property}, $\alpha=s_{\beta,\alpha,\ldots}^{(m-1)}(\beta)$. Thus, \cref{prop:positive-roots} implies that $c=y_{\Sigma^+_{\alpha,\beta}}$.
	Therefore, in the difference $X_{\alpha,\beta,\ldots}^{(m)}-X_{\beta,\alpha,\ldots}^{(m)},$ the coefficients of $\delta_{\alpha,\beta,\ldots}^{(m-1)}$ and $\delta_{\beta,\alpha,\ldots}^{(m-1)}$ are zero, which implies that
	$\kappa_{\alpha,\beta}^{(m-1)}=\kappa_{\beta,\alpha}^{(m-1)}=0$. 
	
	Now suppose $m$ is even. One can show by a similar argument that the $\delta_{\alpha,\beta}^{(m-1)}$-term and the $\delta_{\beta,\alpha}^{(m-1)}$-term in $X_{\alpha,\beta,\ldots}^{(m)}$ are both $-y_{\Sigma_{\alpha,\beta}^+}$. Thus, 	$\kappa_{\alpha,\beta}^{(m-1)}=\kappa_{\beta,\alpha}^{(m-1)}=0$ when $m$ is even as well.
	
	Now we consider the constant terms. By the discussion just before the statement of \cref{thm:iso-with-operators}, there is a natural action of $\mathcal{Q}_{\Lambda,W}^{(R,F)}$ on $\mathcal{Q}_{\Lambda}^{(R,F)}$, where $X_\alpha(r)=\Delta_{\alpha}(r)=0$ and $X_\beta(r)=\Delta_{\beta}(r)=0$ for all $r\in R$. Therefore, the constant term on the right side of \cref{eq:difference} is zero.
	
	Finally, it follows from \cref{lem:relation-in-S} that all coefficients on the right hand side of \cref{eq:difference} are in $R\llbracket\Lambda\rrbracket_F$.
\end{proof}

\begin{rem}
	If $\Sigma$ is crystallographic, then the structure coefficients  $\kappa_{\alpha,\beta}^{(i)}$ and $\kappa_{\beta,\alpha}^{(i)}$ can be explicitly computed from the formulas in \cite[Prop. 6.8]{HMSZ} by applying \cref{lem:one} and \cref{lem:square}. 
\end{rem}

We assume that $m\geq 3$ is odd for the rest of this section. We believe that results similar to those proven in the rest of this section hold when $m$ is even. However, for simplicity of our exposition, we will only consider the case when $m$ is odd.

We use the formulas given in the following Corollary to prove \cref{thm:linear-combo2}.
\begin{cor}\label{cor:linear-combo-2}
Assume $m\geq 3$ is odd. The $\kappa_{\beta,\alpha}^{(1)},\dotsc,\kappa_{\beta,\alpha}^{(m-2)}$ of \cref{lem:linear-combo} satisfy
\begin{align*}
X_{\alpha,\beta,\ldots}^{(m)}=& \sum\limits_{i=1}^{m-2}(-1)^{m-i}\kappa_{\omega_{i+1},\omega_{i}}^{(i)}X_{\omega_{i+1},\omega_i,\ldots}^{(i)} - y_{\Sigma^+_{\alpha,\beta}}\left(\delta_{\alpha,\beta,\ldots}^{(m)}+\sum\limits_{i=1}^{m-1}(-1)^{i}(\delta_{\alpha,\beta,\ldots}^{(m-i)}+\delta_{\beta,\alpha,\ldots}^{(m-i)})-\mathbf{1}\right);\\
X_{\beta,\alpha,\ldots}^{(m)}=&\sum\limits_{i=1}^{m-2}(-1)^{m-i}\kappa_{\omega_i,\omega_{i+1}}^{(i)}X_{\omega_i,\omega_{i+1},\ldots}^{(i)} - y_{\Sigma^+_{\alpha,\beta}}\left(\delta_{\beta,\alpha,\ldots}^{(m)}+\sum\limits_{i=1}^{m-1}(-1)^{i}(\delta_{\alpha,\beta,\ldots}^{(m-i)}+\delta_{\beta,\alpha,\ldots}^{(m-i)})-\mathbf{1}\right).
\end{align*}
\end{cor}
\begin{proof}
 \cref{lem:coeff_in_Q} and \cref{thm:subword} imply that the elements $X_{\alpha,\beta,\ldots}^{(m)}$ and $X_{\beta,\alpha,\ldots}^{(m)}$ are $\mathcal{Q}_\Lambda^{(R,F)}$-linear combinations of the elements in \[\{\mathbf{1},\delta_\alpha,\delta_\beta,\delta_{\alpha}\delta_\beta,\delta_\beta\delta_\alpha,\dotsc,\delta_{\alpha,\beta,\ldots}^{(m-1)},\delta_{\beta,\alpha,\ldots}^{(m-1)},\delta_{\alpha,\beta,\ldots}^{(m)}=\delta_{\beta,\alpha,\ldots}^{(m)}\}.\] 
	So there must exist coefficients $p_{\mathbf{1},j}$, $p_{\alpha,j}^{(i)}$, $p_{\beta,j}^{(i)}\in\mathcal{Q}_\Lambda^{(R,F)}$; $i=1,\dotsc,m$; $j=1,2$, that satisfy 
	\begin{align}\label{eqn:one}
	X_{\alpha,\beta\ldots}^{(m)}=& \sum_{i=1}^{m-2}(-1)^{m-i}\kappa_{\omega_{i+1},\omega_i}^{(i)}X_{\omega_{i+1},\omega_i,\ldots}^{(i)}
	+p_{\alpha,1}^{(m)}\delta_{\alpha,\beta,\dots}^{(m)}+p_{\alpha,1}^{(m-1)}\delta_{\alpha,\beta,\dots}^{(m-1)}\\&+p_{\beta,1}^{(m-1)}\delta_{\beta,\alpha,\dots}^{(m-1)}+\dotsb+p_{\alpha,1}^{(1)}\delta_{\alpha}+p_{\beta,1}^{(1)}\delta_{\beta}+p_{\mathbf{1},1}\mathbf{1},\nonumber
	\end{align}
	\begin{align}\label{eqn:two}
	X_{\beta,\alpha,\ldots}^{(m)}=& \sum_{i=1}^{m-2}(-1)^{m-i}\kappa_{\omega_i,\omega_{i+1}}^{(i)}X_{\omega_i,\omega_{i+1},\ldots}^{(i)} +p_{\beta,2}^{(m)}\delta_{\beta,\alpha,\dots}^{(m)}+p_{\alpha,2}^{(m-1)}\delta_{\alpha,\beta,\dots}^{(m-1)}\\&+p_{\beta,2}^{(m-1)}\delta_{\beta,\alpha,\dots}^{(m-1)}+\dotsb+p_{\alpha,2}^{(1)}\delta_{\alpha}+p_{\beta,2}^{(1)}\delta_{\beta}+p_{\mathbf{1},2}\mathbf{1}.\nonumber
	\end{align}
	By \cref{lem:linear-combo}, together with the fact that $\delta_{\alpha,\beta,\ldots}^{(m)}=\delta_{\beta,\alpha,\ldots}^{(m)}$ and the fact that the $\delta_{w}$, $w\in W$, are $\mathcal{Q}_\Lambda^{(R,F)}$-linearly independent, we find that 
	\begin{equation}\label{eq:equalities} p_{\mathbf{1},1}=p_{\mathbf{1},2}, \quad p_{\beta,1}^{(i)}=p_{\beta,2}^{(i)}, \text{ and} \quad p_{\alpha,1}^{(i)}=p_{\alpha,2}^{(i)}, \quad i=1,\dotsc,m-1.\end{equation}
	Now, in the proof of \cref{lem:linear-combo}, we showed that, for $j=1,2$, the coefficients
	
	\[ p_{\alpha,1}^{(m)}=p_{\beta,2}^{(m)}=- y_{\Sigma^+_{\alpha,\beta}}\quad\text{and} \quad p_{\beta,j}^{(m-1)}=p_{\alpha,j}^{(m-1)}= y_{\Sigma^+_{\alpha,\beta}}.\]	
	
	\noindent Furthermore, since $m$ is odd, in the expansion of the product $X_{\alpha,\beta,\ldots}^{(m)}$, the coefficient of $\delta_{\alpha,\beta,\ldots}^{(i)}$ equals the coefficient of $-\delta_{\alpha,\beta,\ldots}^{(i-1)}$ when $i$ is odd, and the coefficient of $\delta_{\beta,\alpha,\ldots}^{(i)}$ equals the coefficient of $-\delta_{\beta,\alpha,\ldots}^{(i-1)}$ when $i$ is even. This follows from the definition of the formal Demazure element, $X_\alpha=y_\alpha(\mathbf{1}-\delta_{\alpha})$, and from the fact that $\delta_\alpha^2=\delta_\beta^2=\mathbf{1}$. Switching $\alpha$ with $\beta$ gives a similar result for $X_{\beta,\alpha,\ldots}^{(m)}$. See the computations in \cref{calc:calculation} for concrete examples that display this property. 
	This tells us that
	\[ p_{\beta,1}^{(m-2)}=-p_{\beta,1}^{(m-1)}=- y_{\Sigma^+_{\alpha,\beta}}, \text{ and} \quad p_{\alpha,2}^{(m-2)}=-p_{\alpha,2}^{(m-1)}=- y_{\Sigma^+_{\alpha,\beta}}.\]	
	Combining this with \cref{eq:equalities} gives us
	\begin{align*}&p_{\alpha,1}^{(m-3)}=-p_{\alpha,1}^{(m-2)}=-p_{\alpha,2}^{(m-2)}= y_{\Sigma^+_{\alpha,\beta}}, \text{ and}
	\\&p_{\beta,2}^{(m-3)}=-p_{\beta,2}^{(m-2)}=-p_{\beta,1}^{(m-2)}= y_{\Sigma^+_{\alpha,\beta}}.
	\end{align*}
	Now we continue recursively to obtain the formulas for the coefficients that appear in the statement of this Corollary.
\end{proof}

\begin{rem}
	Assume $m\geq 3$ is odd. Suppose $(R,F)$ is the additive formal group law over $R$; or $(R,F)$ is the multiplication formal group law over $R$ and $\Sigma$ is crystallographic. In the first case, \cref{rem:equality-in-additive-demazure-algebra} implies that all the $\kappa_{\alpha,\beta}^{(i)}=0$. In the second case, \cref{rem:equality-in-multiplicative-demazure-algebra} implies that all the $\kappa_{\alpha,\beta}^{(i)}=0$. Thus, in either case, \cref{cor:linear-combo-2} implies that 
	\begin{align*}
	X_{\alpha,\beta,\dotsc}^{(m)}=X_{\beta,\alpha,\dotsc}^{(m)}= y_{\Sigma^+_{\alpha,\beta}}\left(\delta_{\alpha,\beta,\ldots}^{(m)}+\sum\limits_{i=1}^{m-1}(-1)^{i}(\delta_{\alpha,\beta,\ldots}^{(m-i)}+\delta_{\beta,\alpha,\ldots}^{(m-i)})-\mathbf{1}\right).
	\end{align*}
\end{rem}

We will now motivate \cref{lem:coefficient2} with the following example.

\begin{ex}
	Suppose $m=5$. The coefficient of $\delta_{\alpha,\beta,\ldots}^{(3)}$ in the expansion of the product $X_{\alpha,\beta,\ldots}^{(5)}$ is 
	\[c_{\alpha,\beta}^{(3)}=
	-\upsilon_{\alpha,\beta}^{(2)}S_{\alpha,\beta}^{(0,3)}(y_\alpha y_\beta).
	\] 
	To see this, note that in the expansion of 
	\[X_{\alpha,\beta\ldots}^{(5)}=(y_\alpha-y_\alpha\delta_{\alpha})(y_\beta-y_\beta\delta_{\beta})(y_\alpha-y_\alpha\delta_{\alpha})(y_\beta-y_\beta\delta_{\beta})(y_\alpha-y_\alpha\delta_{\alpha}),\]
	all $\delta_{\alpha,\beta,\ldots}^{(3)}$ summands are obtained by choosing a $\delta$-term in each factor, except for two adjacent factors. In the adjacent factors, one should instead choose the constant term. This way, one obtains $4$ summands:
	\begin{align*}
	c_{\alpha,\beta}^{(3)}\delta_{\alpha,\beta,\ldots}^{(3)}=& -(y_\alpha) (y_\beta) (y_\alpha\delta_\alpha)(y_\beta\delta_\beta)(y_\alpha\delta_\alpha) - (y_\alpha\delta_\alpha) (y_\beta) (y_\alpha)(y_\beta\delta_\beta)(y_\alpha\delta_\alpha) \\
	&- (y_\alpha\delta_\alpha) (y_\beta\delta_\beta ) (y_\alpha)(y_\beta)(y_\alpha\delta_\alpha)-  (y_\alpha\delta_\alpha) (y_\beta\delta_\beta ) (y_\alpha \delta_\alpha)(y_\beta)(y_\alpha).
	\end{align*}	
	The first summand is
	\begin{align*}
	-(y_\alpha) (y_\beta) (y_\alpha\delta_\alpha)(y_\beta\delta_\beta)(y_\alpha\delta_\alpha)&=-y_\alpha y_\beta y_\alpha s_\alpha(y_\beta)s_\alpha s_\beta(y_\alpha) \delta_{\alpha}\delta_{\beta}\delta_{\alpha}\\
	&=-y_\alpha y_\beta \upsilon_{\alpha,\beta}^{(2)} \delta_{\alpha,\beta,\ldots}^{(3)}.
	\end{align*}
	The second summand is 
	\begin{align*}
	-(y_\alpha\delta_\alpha) (y_\beta) (y_\alpha)(y_\beta\delta_\beta)(y_\alpha\delta_\alpha)&=-y_\alpha s_\alpha(y_\beta y_\alpha) \delta_\alpha (y_\beta\delta_\beta)(y_\alpha\delta_\alpha)\\
	&=-s_\alpha(y_\alpha y_\beta)(y_\alpha \delta_\alpha)(y_\beta\delta_\beta)(y_\alpha\delta_\alpha)\\
	&=-s_\alpha(y_\alpha y_\beta) y_\alpha s_\alpha(y_\beta)s_\alpha s_\beta(y_\alpha)\delta_\alpha \delta_\beta\delta_\alpha\\
	&=-s_\alpha(y_\alpha y_\beta) \upsilon_{\alpha,\beta}^{(2)}\delta_{\alpha,\beta,\ldots}^{(3)}.
	\end{align*}
	The third summand is 
	\begin{align*}
	- (y_\alpha\delta_\alpha) (y_\beta\delta_\beta ) (y_\alpha)(y_\beta)(y_\alpha\delta_\alpha)&=-y_\alpha s_\alpha(y_\beta)s_\alpha s_\beta(y_\alpha y_\beta) \delta_\alpha \delta_\beta (y_\alpha \delta_\alpha)\\
	&=-s_\alpha s_\beta(y_\alpha y_\beta) y_\alpha (s_\alpha(y_\beta) \delta_\alpha) \delta_\beta (y_\alpha \delta_\alpha)\\
	&=-s_\alpha s_\beta(y_\alpha y_\beta) y_\alpha (\delta_\alpha y_\beta) \delta_\beta (y_\alpha \delta_\alpha)\\
	&=-s_\alpha s_\beta(y_\alpha y_\beta) (y_\alpha \delta_\alpha) (y_\beta \delta_\beta) (y_\alpha \delta_\alpha)\\
	&=-s_\alpha s_\beta(y_\alpha y_\beta) y_\alpha s_\alpha(y_\beta)s_\alpha s_\beta(y_\alpha)\delta_\alpha \delta_\beta\delta_\alpha\\
	&=-s_\alpha s_\beta (y_\alpha y_\beta) \upsilon_{\alpha,\beta}^{(2)}\delta_{\alpha,\beta,\ldots}^{(3)}.
	\end{align*}
	The fourth summand is 
	\begin{align*}
	-  (y_\alpha\delta_\alpha) (y_\beta\delta_\beta ) (y_\alpha \delta_\alpha)(y_\beta)(y_\alpha)&=-y_\alpha s_\alpha(y_\beta) s_\alpha s_\beta(y_\alpha) s_\alpha s_\beta s_\alpha(y_\alpha y_\beta) \delta_\alpha \delta_\beta \delta_\alpha\\
	&=-s_\alpha s_\beta s_\alpha (y_\alpha y_\beta) \upsilon_{\alpha,\beta}^{(2)}\delta_{\alpha,\beta,\ldots}^{(3)}.
	\end{align*}
	Combining these summands, we see that
	\begin{align*}c_{\alpha,\beta}^{(3)}&=-(y_\alpha y_\beta+s_\alpha(y_\alpha y_\beta)+s_{\alpha}s_\beta(y_\alpha y_\beta)+s_\alpha s_\beta s_\alpha(y_\alpha y_\beta))\upsilon_{\alpha,\beta}^{(2)}\\
	&=-\upsilon_{\alpha,\beta}^{(2)}S_{\alpha,\beta}^{(0,3)}(y_\alpha y_\beta).
	\end{align*}
\end{ex}

\begin{lem} \label{lem:coefficient2}
	Assume $m\geq 3$ is odd. The coefficient of $\delta_{\alpha,\beta,\ldots}^{(m-2)}$ in the expansion of the product $X_{\alpha,\beta,\ldots}^{(m)}$ is
	\[c_{\alpha,\beta}^{(m-2)}=
	-\upsilon_{\alpha,\beta}^{(2)}S_{\alpha,\beta}^{(0,m-2)}(y_\alpha y_\beta).
	\] 
	By symmetry, the coefficient of $\delta_{\beta,\alpha,\ldots}^{(m-2)}$ in the expansion of $X_{\beta,\alpha,\ldots}^{(m)}$ is $c_{\beta,\alpha}^{(m-2)}$.
\end{lem}

\begin{proof}
	A $\delta_{\alpha,\beta,\ldots}^{(m-2)}$-summand in the expansion of $X_{\alpha,\beta,\ldots}^{(m)}$ is obtained by choosing $\delta$-terms in all factors except
	of two adjacent factors, $X_\alpha X_\beta$  or $X_\beta X_\alpha$, where one chooses constant terms (if one doesn't choose adjacent factors, then there is cancellation of $\delta$'s, resulting in a word of length less than $m-2$).
	So we obtain $(m-1)$ summands,
	\begin{align*}
	c_{\alpha,\beta}^{(m-2)}\delta_{\alpha,\beta,\ldots}^{(m-2)}=& -(y_\alpha) (y_\beta) (y_\alpha\delta_\alpha)(y_\beta\delta_\beta) \dotsb (y_\beta\delta_\beta)(y_\alpha\delta_\alpha) \\
	&- (y_\alpha\delta_\alpha) (y_\beta) (y_\alpha)(y_\beta\delta_\beta) \dotsb (y_\beta\delta_\beta)(y_\alpha\delta_\alpha) \\
	&- (y_\alpha\delta_\alpha) (y_\beta\delta_\beta ) (y_\alpha)(y_\beta) \dotsb (y_\beta\delta_\beta)(y_\alpha\delta_\alpha) \\
	& \vdots \\
	&-  (y_\alpha\delta_\alpha) (y_\beta\delta_\beta ) (y_\alpha \delta_\alpha)(y_\beta\delta_\beta) \dotsb (y_\beta)(y_\alpha).
	\end{align*}	
	We use the multiplication in $\mathcal{Q}_{\Lambda,W}^{(R,F)}$ to obtain a formula for $c_{\alpha,\beta}^{(m-2)}$.
	Let $i=0,\ldots,m-2$. Then the $(i+1)^{st}$ summand above is
	\begin{align*}
	&-\underbrace{(y_\alpha \delta_\alpha)(y_\beta \delta_\beta)(y_\alpha \delta_\alpha)\dotsb(y_{\omega_{i+1}} \delta_{\omega_{i+1}})}_{i}(y_{\omega_i})(y_{\omega_{i+1}})(y_{\omega_i} \delta_{\omega_i})\ldots(y_\beta \delta_\beta)(y_\alpha \delta_\alpha)	
	\\=&- y_\alpha s_\alpha(y_\beta) s_\alpha s_\beta(y_\alpha) \dotsb\\&\dotsb s_{\alpha,\beta,\ldots}^{(i-1)}(y_{\omega_{i+1}}) s_{\alpha,\beta,\ldots}^{(i)}(y_{\omega_i} y_{\omega_{i+1}}) \delta_{\alpha,\beta,\ldots}^{(i)}(y_{\omega_i} \delta_{\omega_i})\dotsb (y_\beta \delta_\beta)(y_\alpha \delta_\alpha) \\
	=& - s_{\alpha,\beta,\ldots}^{(i)}(y_{\omega_i} y_{\omega_{i+1}})\underbrace{(y_\alpha \delta_\alpha)(y_\beta \delta_\beta)(y_\alpha \delta_\alpha)\dotsb(y_\beta \delta_\beta)(y_\alpha \delta_\alpha)}_{m-2}
	\\=& - s_{\alpha,\beta,\ldots}^{(i)}(y_{\omega_i} y_{\omega_{i+1}}) y_\alpha s_\alpha(y_\beta) s_\alpha s_\beta(y_\alpha)\dotsb s_{\alpha,\beta,\ldots}^{(m-3)}(y_\alpha)\delta_{\alpha,\beta,...}^{(m-2)} 
	\\ =&- s_{\alpha,\beta,\ldots}^{(i)}(y_{\omega_i} y_{\omega_{i+1}}) \upsilon_{\alpha,\beta}^{(2)}\delta_{\alpha,\beta,\ldots}^{(m-2)}.
	\end{align*}
	Note that we can replace $s_{\alpha,\beta,\ldots}^{(i)}(y_{\omega_i} y_{\omega_{i+1}})$ by $s_{\alpha,\beta,\ldots}^{(i)}(y_\alpha y_\beta)$.	Combining these summands, we obtain the formula for the coefficient $c_{\alpha,\beta}^{(m-2)}$ that appears in the statement of the Lemma.
\end{proof}

We will now motivate \cref{lem:coefficient3} with the following example.

\begin{ex}
	Suppose $m=5$. The coefficient of $\delta_{\beta,\alpha,\ldots}^{(2)}$ in the expansion of the product $X_{\alpha,\beta,\ldots}^{(5)}$ is 
	\[
	c_{\beta,\alpha}^{(2)}=
	y_\alpha \upsilon_{\beta,\alpha}^{(3)}\{S_{\beta,\alpha}^{(0,2)}(y_\alpha y_\beta) + s_\alpha(y_\alpha y_\beta)\}.
	\] 
	To see this, note that in the expansion of 
	\[X_{\alpha,\beta,\ldots}^{(5)}=(y_\alpha-y_\alpha\delta_{\alpha})(y_\beta-y_\beta\delta_{\beta})(y_\alpha-y_\alpha\delta_{\alpha})(y_\beta-y_\beta\delta_{\beta})(y_\alpha-y_\alpha\delta_{\alpha}),\]
	all $\delta_{\beta,\alpha,\ldots}^{(2)}$ summands are obtained in one of two ways. In the first way: one chooses the constant term in the first factor, followed by the $\delta$-term in all remaining factors, except for two adjacent factors. In the two adjacent factors, one should instead choose the constant term. Here, one obtains $3$ summands. In the second way: one chooses the constant term in the second factor, and a $\delta$-term in the remaining factors. Here, the $\delta$'s in the first and third factors will cancel each other, giving a summand of $\delta_{\beta,\alpha,\ldots}^{(2)}$. So we get $4$ summands:
	\begin{align*}
	c_{\beta,\alpha}^{(2)}\delta_{\beta,\alpha,\ldots}^{{(2)}}=& (y_\alpha\delta_\alpha) (y_\beta) (y_\alpha\delta_\alpha)(y_\beta\delta_\beta)(y_\alpha\delta_\alpha) \\& +(y_\alpha) (y_\beta) (y_\alpha)(y_\beta\delta_\beta) (y_\alpha\delta_\alpha) 
	\\& +(y_\alpha) (y_\beta\delta_\beta ) (y_\alpha)(y_\beta) (y_\alpha\delta_\alpha) 
	\\&  +(y_\alpha) (y_\beta\delta_\beta) (y_\alpha \delta_\alpha)(y_\beta)(y_\alpha).
	\end{align*}		
	The first summand is 
	\begin{align*}
	(y_\alpha\delta_\alpha) (y_\beta) (y_\alpha\delta_\alpha)(y_\beta\delta_\beta)(y_\alpha\delta_\alpha)&=y_\alpha s_\alpha(y_\beta y_\alpha) \delta_{\alpha}\delta_\alpha (y_\beta\delta_\beta)(y_\alpha\delta_\alpha)\\
	&=y_\alpha s_\alpha(y_\beta y_\alpha) y_\beta s_\beta(y_\alpha)\delta_{\beta}\delta_{\alpha}\\
	&=y_\alpha \upsilon_\beta^{(3)} s_\alpha(y_\alpha y_\beta)\delta_{\beta,\alpha,\ldots}^{(2)}.
	\end{align*}
	The second summand is
	\begin{align*}
	(y_\alpha) (y_\beta) (y_\alpha)(y_\beta\delta_\beta) (y_\alpha\delta_\alpha) &=y_\alpha y_\beta y_\alpha (y_\beta s_\beta(y_\alpha))\delta_\beta \delta_\alpha\\
	&=y_\alpha \upsilon_\beta^{(3)} y_\alpha y_\beta \delta_{\beta,\alpha,\ldots}^{(2)}.
	\end{align*}
	The third summand is
	\begin{align*}
	(y_\alpha) (y_\beta\delta_\beta ) (y_\alpha)(y_\beta) (y_\alpha\delta_\alpha) &=y_\alpha y_\beta s_\beta(y_\alpha y_\beta) \delta_\beta (y_\alpha \delta_\alpha)\\
	&=y_\alpha s_\beta(y_\alpha y_\beta) (y_\beta \delta_\beta)(y_\alpha \delta_\alpha)\\
	&=y_\alpha s_\beta(y_\alpha y_\beta) y_\beta s_\beta(y_\alpha) \delta_\beta \delta_\alpha\\
	&=y_\alpha \upsilon_{\beta,\alpha}^{(3)} s_\beta(y_\alpha y_\beta) \delta_{\beta,\alpha,\ldots}^{(2)}. 
	\end{align*}
	The fourth summand is
	\begin{align*}
	(y_\alpha) (y_\beta\delta_\beta) (y_\alpha \delta_\alpha)(y_\beta)(y_\alpha)&=y_\alpha y_\beta s_\beta(y_\alpha) s_\beta s_\alpha(y_\beta y_\alpha) \delta_\beta \delta_\alpha\\
	&=y_\alpha \upsilon_{\beta,\alpha}^{(3)} s_\beta s_\alpha(y_\alpha y_\beta) \delta_{\beta,\alpha,\ldots}^{(2)}.
	\end{align*}
	Combining these summands, we see that
	\begin{align*}c_{\beta,\alpha}^{(2)}&=y_\alpha\upsilon_{\beta,\alpha}^{(3)}\{ y_\alpha y_\beta +s_\beta(y_\alpha y_\beta) + s_\beta s_\alpha(y_\alpha y_\beta) + s_\alpha(y_\alpha y_\beta)\}\\
	&=y_\alpha\upsilon_{\beta,\alpha}^{(3)}\{S_{\beta,\alpha}^{(0,2)}(y_\alpha y_\beta) + s_\alpha(y_\alpha y_\beta)\}.
	\end{align*}
\end{ex}

\begin{lem} \label{lem:coefficient3}
Assume $m\geq 5$ is odd. Then the coefficient of $\delta_{\beta,\alpha,\ldots}^{(m-3)}$ in the expansion of the product $X_{\alpha,\beta,\ldots}^{(m)}$ is 
	\[
	c_{\beta,\alpha}^{(m-3)}=
	y_\alpha \upsilon_{\beta,\alpha}^{(3)}\{S_{\beta,\alpha}^{(0,m-3)}(y_\alpha y_\beta) + s_\alpha(y_\alpha y_\beta)\}.
	\] 
	By symmetry, the coefficient of $\delta_{\alpha,\beta,\ldots}^{(m-3)}$ in the expansion of $X_{\beta,\alpha,\ldots}^{(m)}$ is $c_{\alpha,\beta}^{(m-3)}$.
\end{lem}

\begin{proof}
	A $\delta_{\beta,\alpha,\ldots}^{(m-3)}$-summand in the expansion of $X_{\alpha,\beta,\ldots}^{(m)}$ is obtained in one of two ways. In the first way, one chooses the constant term in the first $X_\alpha$, followed by $\delta$-terms in the remaining factors, except for two adjacent factors. In the two adjacent factors, one should instead choose the constant terms (if one doesn't choose adjacent factors, then there is a cancellation of $\delta$'s, resulting in a word of length less than $m-3$; if one chooses the $\delta$-term in the first factor, the product of $\delta$'s will begin with $\delta_\alpha$). This gives $m-2$ summands. In the second way, we obtain one additional summand by choosing the $\delta$-term in the first factor, followed by the constant term in the second factor, followed by the $\delta$-term in each of the remaining factors (so that the first and third $\delta_\alpha$'s will cancel). So we have $m-1$ summands in total:
	\begin{align*}
	c_{\beta,\alpha}^{(m-3)}\delta_{\beta,\alpha,\ldots}^{(m-3)}=& (y_\alpha\delta_\alpha) (y_\beta) (y_\alpha\delta_\alpha)(y_\beta\delta_\beta) \dotsb (y_\beta\delta_\beta)(y_\alpha\delta_\alpha) \\& +(y_\alpha) (y_\beta) (y_\alpha)(y_\beta\delta_\beta) \dotsb (y_\beta\delta_\beta)(y_\alpha\delta_\alpha) 
	\\& +(y_\alpha) (y_\beta\delta_\beta ) (y_\alpha)(y_\beta) \dotsb (y_\beta\delta_\beta)(y_\alpha\delta_\alpha) 
	\\ \vdots & 
	\\& + (y_\alpha) (y_\beta\delta_\beta) (y_\alpha \delta_\alpha)\dotsb (y_\alpha\delta_\alpha)(y_\beta)(y_\alpha).
	\end{align*}	
	We use the multiplication in $\mathcal{Q}_{\Lambda,W}^{(R,F)}$ to obtain a formula for $c^{(m-3)}_{\beta,\alpha}$.
	The formula for the first summand above is
	\begin{align*}
	&(y_\alpha\delta_\alpha) (y_\beta) (y_\alpha \delta_\alpha)\dotsb (y_\alpha\delta_\alpha)(y_\beta\delta_\beta)(y_\alpha\delta_\alpha)
	\\ =&  y_\alpha s_\alpha (y_\alpha y_\beta) \underbrace{(y_\beta \delta_\beta)\dotsb (y_\alpha\delta_\alpha)(y_\beta\delta_\beta)(y_\alpha\delta_\alpha)}_{m-3}
	\\=&  y_\alpha s_\alpha (y_\alpha y_\beta) y_\beta s_\beta(y_\alpha)\dotsb s_{\beta,\alpha,\ldots}^{(m-4)}(y_\alpha)\delta_{\beta,\alpha,\ldots}^{(m-3)}
	\\=&  y_\alpha s_\alpha (y_\alpha y_\beta) \upsilon_{\beta,\alpha}^{(3)}\delta_{\beta,\alpha,\ldots}^{(m-3)}.
	\end{align*}	
	Now, let $i=0,\ldots,m-3$. The $(i+2)^{nd}$ summand above is
	\begin{align*}
	&(y_\alpha) \underbrace{(y_\beta \delta_\beta) (y_\alpha \delta_\alpha)\dotsb (y_{\omega_i} \delta_{\omega_i})}_{i} (y_{\omega_{i+1}})(y_{\omega_i}) (y_{\omega_{i+1}} \delta_{\omega_{i+1}})\dotsb (y_\beta\delta_\beta)(y_\alpha\delta_\alpha)\\
	=& y_\alpha y_\beta s_\beta(y_\alpha) s_\beta s_\alpha(y_\beta) \dotsb\\&\textcolor{white}{}\dotsb s_{\beta,\alpha,\ldots}^{(i-1)}(y_{\omega_i}) s_{\beta,\alpha,\ldots}^{(i)}(y_{\omega_{i+1}} y_{\omega_i}) \delta_{\beta,\alpha,\ldots}^{(i)} (y_{\omega_{i+1}} \delta_{\omega_{i+1}})\dotsb(y_\beta \delta_\beta)(y_\alpha \delta_\alpha)\\=& y_\alpha s_{\beta,\alpha,\ldots}^{(i)}(y_{\omega_{i+1}} y_{\omega_i})\underbrace{(y_\beta \delta_\beta)(y_\alpha \delta_\alpha)(y_\beta \delta_\beta)\dotsb(y_\beta \delta_\beta)(y_\alpha \delta_\alpha)}_{m-3}
	\\ =& y_\alpha s_{\beta,\alpha,\ldots}^{(i)}(y_{\omega_{i+1}} y_{\omega_i}) y_\beta s_\beta(y_\alpha) s_\beta s_\alpha(y_\beta)\dotsb s_{\beta,\alpha,\ldots}^{(m-4)}(y_\gamma)\delta_{\beta,\alpha,\ldots}^{(m-3)}
	\\=& y_\alpha s_{\beta,\alpha,\ldots}^{(i)}(y_{\omega_{i+1}} y_{\omega_i}) \upsilon_{\beta,\alpha}^{(3)}\delta_{\beta,\alpha,\ldots}^{(m-3)}.
	\end{align*}
	Note that we can replace $s_{\beta,\alpha,\ldots}^{(i)}(y_{\omega_{i+1}} y_{\omega_i})$ by $s_{\beta,\alpha,\ldots}^{(i)}(y_\alpha y_\beta)$. Combining these summands, we obtain the general formula for the coefficient $c_{\beta,\alpha}^{(m-3)}$ that appears in the statement of the Lemma. \qedhere	
\end{proof}

\begin{thm}\label{thm:linear-combo2}
	Suppose $m\geq 3$ is odd.
	Below is an explicit formula for the structure coefficient $\kappa_{\beta,\alpha}^{(m-2)}$:
	\[
	\kappa_{\beta,\alpha}^{(m-2)}=
	S_{\beta,\alpha}^{(0,m-2)}(y_\alpha y_\beta)-y_\alpha s_{\beta,\alpha,\ldots}^{(m-2)}(y_\alpha).
	\]

	\
	
	\noindent Suppose $m\geq 5$ is odd.
	Below is an explicit formula for the structure coefficient $\kappa_{\alpha,\beta}^{(m-3)}$:
	\begin{align*}
	\kappa_{\alpha,\beta}^{(m-3)}=
	&-y_\beta\huge\{s_\alpha (y_\alpha y_\beta) +\left[S_{\alpha,\beta}^{(2,m-3)}-S_{\beta,\alpha}^{(2,m-3)}\right] (y_\alpha y_\beta) \\&- s_{\beta,\alpha,\ldots}^{(m-2)}(y_\alpha y_\beta) + y_\alpha s_{\beta,\alpha,\ldots}^{(m-2)}(y_\alpha) - s_{\alpha,\beta,\ldots}^{(m-3)}(y_\alpha) s_{\alpha,\beta,\ldots}^{(m-2)}(y_\beta)\huge\}.
	\end{align*}
\end{thm}

\begin{proof}
	By \cref{lem:diheral-property}, we have $s_{\alpha,\beta,\ldots}^{(m-1)}(\alpha)=\beta$ when $m$ is odd. We use this property implicitly in this proof.
	 
	First, we will determine the coefficient $\kappa_{\beta,\alpha}^{(m-2)}.$ Replacing $m$ with $m-2$ in \cref{lem:linear-combo} and using the same method of proof, we deduce that the coefficient of $\delta_{\alpha,\beta,\ldots}^{(m-2)}$ in the expansion of the product $X_{\alpha,\beta, \ldots}^{(m-2)}$ is 
	\[
	b_{\alpha,\beta}=
	-\upsilon_{\alpha,\beta}^{(2)}.
	\]
	In \cref{lem:coefficient2}, we found the coefficient of $\delta_{\beta,\alpha,\ldots}^{(m-2)}$ in the expansion of $X_{\alpha,\beta,\ldots}^{(m)}$, which we denoted $c_{\beta,\alpha}^{(m-2)}$. Therefore, by \cref{cor:linear-combo-2}, we have the following formula:
	
	\[
	X_{\beta,\alpha,\ldots}^{(m)}=
	\left(\kappa_{\beta,\alpha}^{(m-2)}X_{\beta,\alpha,\ldots}^{(m-2)} -\dotsb\right) + y_{\Sigma^+_{\alpha,\beta}}\left(-\delta_{\beta,\alpha,\dots}^{(m-2)}+\dotsb\right),
	\]
	where \[\kappa_{\beta,\alpha}^{(m-2)}=
	\tfrac{1}{b_{\alpha,\beta}}\left(c_{\beta,\alpha}^{(m-2)}+y_{\Sigma^+_{\alpha,\beta}}\right).\]
	One can check that, after cancellations and simplifications, this equation is the same as the equation given in the statement of this Theorem. 
	
	Now we will determine the coefficient $\kappa_{\alpha,\beta}^{(m-3)}.$ Replacing $m$ with $m-2$ in \cref{lem:linear-combo} and using the same method of proof, we deduce that the coefficient of $\delta_{\alpha,\beta,\ldots}^{(m-3)}$ in the expansion of the product $X_{\beta,\alpha,\ldots}^{(m-2)}$ is
	\[d_{\alpha,\beta}=
	y_\beta \upsilon_{\alpha,\beta}^{(3)}.
	\]
	Replacing $m$ with $m-3$ instead, we deduce that the coefficient of $\delta_{\alpha,\beta,\ldots}^{(m-3)}$ in the expansion of the product $X_{\alpha,\beta,\ldots}^{(m-3)}$ is 
	\[
	e_{\alpha,\beta}=
	\upsilon_{\alpha,\beta}^{(3)}.
	\]
	In \cref{lem:coefficient3}, we found the coefficient of $\delta_{\alpha,\beta,\ldots}^{(m-3)}$ in the expansion of $X_{\beta,\alpha,\ldots}^{(m)}$, which we denoted $c_{\alpha,\beta}^{(m-3)}$. Therefore, by  \cref{cor:linear-combo-2}, we have the following formula:
	
	\[
	X_{\beta,\alpha,\ldots}^{(m)}=
	\left(\kappa_{\beta,\alpha}^{(m-2)}X_{\beta,\alpha,\ldots}^{(m-2)} - \kappa_{\alpha,\beta}^{(m-3)}X_{\alpha,\beta,\ldots}^{(m-3)}+\dotsb\right) + y_{\Sigma^+_{\alpha,\beta}}\left(\delta_{\alpha,\beta,\dots}^{(m-3)}+\dotsb\right),\]
	where \[\kappa_{\alpha,\beta}^{(m-3)}=
	-\tfrac{1}{e_{\alpha,\beta}}\left(c_{\alpha,\beta}^{(m-3)}-\kappa_{\beta,\alpha}^{(m-2)}d_{\alpha,\beta}-y_{\Sigma^+_{\alpha,\beta}}\right).
	\]
	One can check that, after cancellations and simplifications, this equation is the same as the equation given in the statement of this Theorem.
\end{proof}
\begin{rem}
	We leave the problem of computing general formulas for the remaining structure coefficients of \cref{lem:linear-combo} open. In \cref{section:applications}, we compute all structure coefficients for the groups $I_2(5)$, $I_2(7)$, $H_3$, and $H_4$.
\end{rem}

\section{Applications}\label{section:applications}

In this section, we specialize the structure coefficients derived in \cref{section:non-simply-laced} to certain formal group laws. We also compute all structure coefficients for the dihedral groups $I_2(5)$ and $I_2(7)$, and the reflection groups $H_3$ and $H_4$.

We will use the following Lemma in \cref{ex:special-cases}.
\begin{lem}\label{lem:reflection-equality} 
	Suppose $m\geq 3$ is odd. Let $i=1,\dotsc,m-1$. If $i$ is odd, we have $s_{\alpha,\beta,\ldots}^{(i)}(\beta)=s_{\beta,\alpha,\ldots}^{(m-i-1)}(\alpha)$, and if $i$ is even, we have $s_{\alpha,\beta,\ldots}^{(i)}(\alpha)=s_{\beta,\alpha,\ldots}^{(m-i-1)}(\beta).$ 
\end{lem}
\begin{proof}
	It follows from \cref{lem:diheral-property} that $s_{\alpha,\beta,\ldots}^{(m)}(\beta)=-\alpha$ when $m$ is odd. Applying compositions of reflections $s_\alpha$ and $s_\beta$ to both sides of this equation to reduce the length of $s_{\alpha,\beta,\ldots}^{(m)}$ gives the desired equations.
\end{proof}

\begin{ex}\label{ex:special-cases}
	Suppose that $m\geq 3$ is \textit{odd}.
	Let $(R,F)$ be any formal group law such that, for all $\alpha\in\Sigma$, the formula $x_\alpha+x_{-\alpha}-qx_\alpha x_{-\alpha}=0$ holds, where $q\in R$ is fixed. This includes the multiplicative formal group law over $R$, and any formal group law of additive type. Observe that, in $\mathcal{Q}_\Lambda^{(R,F)}$, the condition  $x_\alpha+x_{-\alpha}-qx_\alpha x_{-\alpha}=0$ for all $\alpha\in\Sigma$ is equivalent to the condition $y_{-\alpha}=q-y_\alpha$ for all $\alpha\in\Sigma$. We show that under this condition, the structure coefficients of \cref{thm:linear-combo2} satisfy
	\begin{align*}
	&\textup{(a)} \quad\kappa_{\alpha,\beta}^{(m-2)}=\kappa_{\beta,\alpha}^{(m-2)}, \quad \text{and}\\
	&\textup{(b)} \quad\kappa_{\alpha,\beta}^{(m-3)}=\kappa_{\beta,\alpha}^{(m-3)}=0.
	\end{align*}	
	(a) Using the formula of $\kappa_{\alpha,\beta}^{(m-2)}$ obtained in \cref{thm:linear-combo2} and making the substitution $y_{-\alpha}=(q-y_\alpha)$, we can write 
	\begin{align*}\kappa_{\beta,\alpha}^{(m-2)}&=y_\alpha y_\beta + \{[q-y_{\beta}]s_\beta(y_\alpha)+[q-s_\beta(y_\alpha)]s_\beta s_\alpha(y_\beta)\}+\dotsb
	\\&+ \{[q-s_{\beta,\alpha,\ldots}^{(m-4)}(y_\alpha)]s_{\beta,\alpha,\ldots}^{(m-3)}(y_\beta)+ [q-s_{\beta,\alpha,\ldots}^{(m-3)}(y_{\beta})]s_{\beta,\alpha,\ldots}^{(m-2)}(y_\alpha)\}\\&-y_\alpha s_{\beta,\alpha,\ldots }^{(m-2)}(y_\alpha), \quad \text{ and}
	\end{align*}
	\begin{align*}
	\kappa_{\alpha,\beta}^{(m-2)}&=y_\alpha y_\beta + \{[q-y_{\alpha}]s_\alpha(y_\beta)+[q-s_\alpha(y_\beta)]s_\alpha s_\beta(y_\alpha)\}+\dotsb
	\\&+ \{[q-s_{\alpha,\beta,\ldots}^{(m-4)}(y_\beta)]s_{\alpha,\beta,\ldots}^{(m-3)}(y_\alpha) +[q-s_{\alpha,\beta,\ldots}^{(m-3)}(y_{\alpha})]s_{\alpha,\beta,\ldots}^{(m-2)}(y_\beta)\}\\&-y_\beta s_{\alpha,\beta,\ldots }^{(m-2)}(y_\beta).
	\end{align*}
	By \cref{lem:reflection-equality}, we have the following relations:
	\begin{align}
	&y_\alpha s_\alpha(y_\beta)=y_\alpha s_{\beta,\alpha,\ldots}^{(m-2)}(y_\alpha) \quad \text{ and         } \quad  y_\beta s_\beta(y_\alpha)=y_\beta s_{\alpha,\beta,\ldots}^{(m-2)}(y_\beta), \nonumber
	\\&qs_{\alpha,\beta,\ldots}^{(i)}(y_\beta)=qs_{\beta,\alpha,\ldots}^{(m-i-1)}(y_\alpha), \quad \quad \quad\quad \quad\quad \quad \quad\quad\quad i=1,3,5,\dotsc,m-2, \label{eq:5}\\
	&qs_{\alpha,\beta,\ldots}^{(i)}(y_\alpha)=qs_{\beta,\alpha,\ldots}^{(m-i-1)}(y_\beta), \quad \quad\quad \quad\quad \quad \quad\quad \quad\quad i=2,4,\dotsc,m-3, \label{eq:6}
	\\&s_{\alpha,\beta,\ldots}^{(i)}(y_\alpha)s_{\alpha,\beta,\ldots}^{(i+1)}(y_\beta)=s_{\beta,\alpha,\ldots}^{(m-i-1)}(y_\beta)s_{\beta,\alpha,\ldots}^{(m-i-2)}(y_\alpha), \quad i=2,4,\dotsc,m-3,  \label{eq:7}
	\\&s_{\alpha,\beta,\ldots}^{(i)}(y_\beta)s_{\alpha,\beta,\ldots}^{(i+1)}(y_\alpha)=s_{\beta,\alpha,\ldots}^{(m-i-1)}(y_\alpha)s_{\beta,\alpha,\ldots}^{(m-i-2)}(y_\beta), \quad i=1,3,5,\dotsc,m-4. \label{eq:8}
	\end{align}	
	Using these relations and comparing $\kappa_{\beta,\alpha}^{(m-2)}$ and $\kappa_{\alpha,\beta}^{(m-2)}$, it is straightforward to verify that $\kappa_{\alpha,\beta}^{(m-2)}=\kappa_{\beta,\alpha}^{(m-2)}$.\\
	
	(b) Using the formula for $\kappa_{\alpha,\beta}^{(m-3)}$ obtained in \cref{thm:linear-combo2} and making the substitution $y_{-\alpha}=(q-y_\alpha)$, we can write 
	\begin{align*}
	-\tfrac{\kappa_{\alpha,\beta}^{(m-3)}}{y_\beta}=& [q-y_{\alpha}]s_\alpha(y_\beta) +\{[q-s_\alpha(y_{\beta})]s_\alpha s_\beta(y_\alpha) - [q-s_\beta (y_{\alpha})]s_\beta s_\alpha(y_\beta)\} +\dotsb\\& + \{[q-s_{\alpha,\beta,\ldots}^{(m-4)}(y_{\beta})]s_{\alpha,\beta,\ldots}^{(m-3)}(y_\alpha) - [q-s_{\beta,\alpha,\ldots}^{(m-4)}(y_{\alpha})] s_{\beta,\alpha,\ldots}^{(m-3)}(y_\beta)\} \\&- [q-s_{\beta,\alpha\ldots}^{(m-3)}(y_{\beta})]s_{\beta,\alpha,\ldots}^{(m-2)}(y_\alpha) + y_\alpha s_{\beta,\alpha,\ldots}^{(m-2)}(y_\alpha) - s_{\alpha,\beta,\ldots}^{(m-3)}(y_\alpha) s_{\alpha,\beta,\ldots}^{(m-2)}(y_\beta).
	\end{align*}
	By \cref{lem:reflection-equality}, we have the following relations:
	\begin{align*}
	&y_\alpha s_\alpha (y_\beta)=y_\alpha s_{\beta,\alpha,\ldots}^{(m-2)}(y_\alpha),\quad  \text{and}
	\\&qs_{\beta,\alpha,\ldots}^{(m-2)}(y_\alpha)=qs_\alpha(y_\beta).
	\end{align*}
	These relations, together with the relations \cref{eq:5,eq:6,eq:7,eq:8}, allow us to deduce that the terms of $\kappa_{\alpha,\beta}^{(m-3)}$ cancel in pairs. We conclude that $\kappa_{\alpha,\beta}^{(m-3)}=\kappa_{\beta,\alpha}^{(m-3)}=0$.
\end{ex}

\cref{ex:lower-order-relations} follows from \cref{lem:linear-combo} and \cref{thm:linear-combo2}:

\begin{ex}{\label{ex:lower-order-relations}}
	\begin{enumerate}[label=(\alph*)]
		\item\label{itm:order-2} Suppose $m=2$. In this case,
		\[X_\alpha X_\beta=X_\beta X_\alpha.\]
		\item \label{itm:order-3} Suppose $m=3$. In this case,
		\[X_\alpha X_\beta X_\alpha-X_\beta X_\alpha X_\beta=\kappa_{\alpha,\beta}^{(1)}X_\alpha-\kappa_{\beta,\alpha}^{(1)}X_\beta,\]
		where $\kappa_{\beta,\alpha}^{(1)}=y_{\alpha}y_{\beta}+s_\beta (y_{\alpha}y_{\beta})-y_{\alpha}s_\beta (y_{\alpha})$.
		\details{\item \label{itm:order-4} Suppose $m=4$. In this case,
		\begin{align*}
		X_\alpha X_\beta X_\alpha X_\beta-X_\beta X_\alpha X_\beta X_\alpha=&\kappa_{\alpha,\beta}^{(2)}X_\alpha X_\beta -\kappa_{\beta,\alpha}^{(2)}  X_\beta X_\alpha \\&+ \kappa_{\alpha,\beta}^{(1)}X_\alpha -\kappa_{\beta,\alpha}^{(1)}  X_\beta ,
		\end{align*}
		where \[\kappa_{\beta,\alpha}^{(2)}=y_\alpha y_\beta+s_\beta (y_\alpha y_\beta) +s_{\beta}s_\alpha(y_\alpha y_\beta)-y_\alpha s_{\beta}s_\alpha(y_\beta);\] 
		\[\kappa_{\alpha,\beta}^{(1)}=-y_\beta s_\alpha(y_\alpha y_\beta)+y_\beta s_\beta s_\alpha(y_\alpha y_\beta)-y_\alpha y_\beta s_\beta s_\alpha(y_\beta)+y_\beta s_\alpha(y_\beta)s_\alpha s_\beta(y_\alpha).\]}
	\end{enumerate}
\end{ex}

\begin{ex} \label{ex:order-5}
	Suppose $m=5$. We provide explicit formulas for $\kappa_{\beta,\alpha}^{(1)},\kappa_{\alpha,\beta}^{(2)}$, and $\kappa_{\beta,\alpha}^{(3)}$:
	\begin{align*}
	\kappa_{\beta,\alpha}^{(3)} = &S_{\beta,\alpha}^{(0,3)}(y_\alpha y_\beta) - y_\alpha s_\beta s_\alpha s_\beta (y_\alpha),	\\
	\kappa_{\alpha,\beta}^{(2)}=&-y_\beta\{s_\alpha (y_\alpha y_\beta) + s_\alpha s_\beta (y_\alpha y_\beta) -s_\beta s_\alpha (y_\alpha y_\beta)-s_\beta s_\alpha s_\beta (y_\alpha y_\beta)\\& + 
	y_\alpha s_\beta s_\alpha s_\beta (y_\alpha) -s_\alpha s_\beta (y_\alpha) s_\alpha s_\beta s_\alpha (y_\beta)\},\\
	\kappa_{\beta,\alpha}^{(1)}=&-s_\beta (y_\alpha y_\beta)s_\beta s_\alpha s_\beta(y_\alpha)\{s_\beta s_\alpha s_\beta ( y_\beta) - y_\alpha\}-y_\alpha y_\beta s_\alpha s_\beta(y_\alpha)\{s_\alpha s_\beta ( y_\beta) \\&-
	s_\alpha s_\beta s_\alpha (y_\beta)\}-y_\alpha s_\beta (y_\alpha) s_\beta s_\alpha (y_\beta) s_\beta s_\alpha s_\beta (y_\alpha).
	\end{align*}
	We use the results proven in \cref{section:non-simply-laced} to justify these formulas.

	The formulas for $\kappa_{\beta,\alpha}^{(3)}$ and $\kappa_{\alpha,\beta}^{(2)}$ are obtained from \cref{thm:linear-combo2}. Now we will find $\kappa_{\beta,\alpha}^{(1)}$. In \cref{calc:calculation}, we provide formulas of products of up to seven formal Demazure elements. We use these formulas, together with the formulas for the coefficients $\kappa_{\beta,\alpha}^{(3)}$ and  $\kappa_{\alpha,\beta}^{(2)}$, to determine $\kappa_{\beta,\alpha}^{(1)}$. We do this by subtracting $\kappa_{\beta,\alpha}^{(3)} X_\beta X_\alpha X_\beta-\kappa_{\alpha,\beta}^{(2)}X_\alpha X_\beta+y_{\Sigma^+_{\alpha,\beta}}(\mathbf{1}-\delta_{\beta})$ from $X_{\beta\ldots}^{(5)}$, and then using \cref{cor:linear-combo-2}.\\
	
	\noindent The coefficient of $(\mathbf{1}-\delta_\beta)$ in the expansion of the product $X_{\beta,\alpha,\ldots}^{(5)}$ is
	\begin{align*}
	&+y_\beta\{[S_{\beta,\alpha}^{(0,1)}(y_\alpha y_\beta)]^2 + s_\beta (y_\alpha y_\beta) s_\beta s_\alpha (y_\alpha y_\beta)\}
	\\&+y_\alpha y_\beta^2 s_\alpha (y_\alpha y_\beta).
	\end{align*} 
	The coefficient of $(\mathbf{1}-\delta_\beta)$ in the expansion of the product $\kappa_{\beta,\alpha}^{(3)} X_\beta X_\alpha X_\beta$ is 
	\begin{align*}
	&+y_\beta\{[S_{\beta,\alpha}^{(0,1)}(y_\alpha y_\beta)]^2 + s_\beta (y_\alpha y_\beta) s_\beta s_\alpha (y_\alpha y_\beta)\}
	\\&+y_\alpha y_\beta^2\{S_{\beta,\alpha}^{(2,3)}(y_\alpha y_\beta) - y_\alpha s_\beta s_\alpha s_\beta (y_\alpha)\}
	\\&+y_\beta s_\beta (y_\alpha y_\beta)\{s_\beta s_\alpha s_\beta (y_\alpha y_\beta) - y_\alpha s_\beta s_\alpha s_\beta (y_\alpha)\}.
	\end{align*}
	The coefficient of $(\mathbf{1}-\delta_\beta)$ in the expansion of the product $-\kappa_{\alpha,\beta}^{(2)}X_\alpha X_\beta$ is 
	\begin{align*}
	&+y_\alpha y_\beta^2 s_\alpha (y_\alpha y_\beta)
	\\&-y_\alpha y_\beta^2\{S_{\beta,\alpha}^{(2,3)}(y_\alpha y_\beta) - y_\alpha s_\beta s_\alpha s_\beta (y_\alpha)\}
	\\&+y_\alpha y_\beta^2\{s_\alpha s_\beta (y_\alpha y_\beta) - s_\alpha s_\beta (y_\alpha) s_\alpha s_\beta s_\alpha (y_\beta)\}.
	\end{align*}
	After cancellations, the coefficient of $(\mathbf{1}-\delta_\beta)$ in the term $X_{\beta\ldots}^{(5)} - \kappa_{\beta,\alpha}^{(3)} X_\beta X_\alpha X_\beta + \kappa_{\alpha,\beta}^{(2)}X_\alpha X_\beta$ is
	\begin{align*}
	C=&-y_\beta \{s_\beta (y_\alpha y_\beta)[s_\beta s_\alpha s_\beta (y_\alpha y_\beta) - y_\alpha s_\beta s_\alpha s_\beta (y_\alpha)]+y_\alpha y_\beta[s_\alpha s_\beta (y_\alpha y_\beta) \\&- s_\alpha s_\beta (y_\alpha) s_\alpha s_\beta s_\alpha (y_\beta)]\}.
	\end{align*}
	The coefficient of $(\mathbf{1}-\delta_\beta)$ in  $X_\beta$ is $y_\beta$. Therefore, by \cref{cor:linear-combo-2}, the coefficient at $X_\beta$ in the braid relation is $\kappa_\beta^1=\tfrac{C-y_{\Sigma^+_{\alpha,\beta}}}{y_\beta}$. Finally, by \cref{lem:reflection-equality}, we have that $y_\alpha=s_{\beta,\alpha,\ldots}^4(y_\beta)$, and that $y_{\Sigma^+_{\alpha,\beta}}=y_\beta s_\beta(y_\alpha) s_\beta s_\alpha(y_\beta) s_\beta s_\alpha s_\beta(y_\alpha)s_{\beta,\alpha,\ldots}^{(4)}(y_\beta)$. 
\end{ex}

\begin{ex}\label{ex:order-7}	
	Suppose $m=7$. We provide explicit formuals for $\kappa_{\beta,\alpha}^{(1)}$, $\kappa_{\alpha,\beta}^{(2)}$, $\kappa_{\beta,\alpha}^{(3)}$, $\kappa_{\alpha,\beta}^{(4)}$, and $\kappa_{\beta,\alpha}^{(5)}$:
	\begin{align*}	
	k_{\beta,\alpha}^{(5)}= &S_{\beta,\alpha}^{(0,5)}(y_\alpha y_\beta)-y_\alpha s_{\beta,\alpha,\ldots}^{(5)}(y_\alpha),\\
	k_{\alpha,\beta}^{(4)}=&-y_\beta\{s_\alpha(y_\alpha y_\beta)+[S_{\alpha,\beta}^{(2,4)}-S_{\beta,\alpha}^{(2,4)}](y_\alpha y_\beta)-s_{\beta,\alpha,\ldots}^{(5)}(y_\alpha y_\beta)+y_\alpha s_{\beta,\alpha,\ldots}^{(5)}(y_\alpha)\\&-s_{\alpha,\beta,\ldots}^{(4)}(y_\alpha) s_{\alpha,\beta,\ldots}^{(5)}(y_\beta)\},\\
	k_{\beta,\alpha}^{(3)}=&	-S_{\beta,\alpha}^{(1,3)}(y_\alpha y_\beta)[s_{\beta,\alpha,\ldots}^{(5)}(y_\alpha y_\beta) - y_\alpha s_{\beta,\alpha,\ldots}^{(5)}(y_\alpha)]-S_{\beta,\alpha}^{(1,2)}(y_\alpha y_\beta)[s_{\beta,\alpha,\ldots}^{(4)}(y_\alpha y_\beta)] \\&-s_\beta(y_\alpha y_\beta) s_\beta s_\alpha s_\beta (y_\alpha y_\beta)-y_\alpha y_\beta[S_{\alpha,\beta}^{(2,4)}(y_\alpha y_\beta) - s_{\alpha,\beta,\ldots}^{(4)}(y_\alpha) s_{\alpha,\beta,\ldots}^{(5)}(y_\beta)]\\&- y_\alpha s_\beta s_\alpha s_\beta (y_\alpha) s_{\beta,\alpha,\ldots}^{(4)}(y_\beta)s_{\beta,\alpha,\ldots}^{(5)}(y_\alpha),	\\
	k_{\alpha,\beta}^{(2)}=&-y_\beta\{-s_\alpha (y_\alpha y_\beta)[S_{\alpha,\beta}^{(3,4)}(y_\alpha y_\beta)- s_{\alpha,\beta,\ldots}^{(4)}(y_\alpha) s_{\alpha,\beta,\ldots}^{(5)}(y_\beta)] \\&- s_\alpha s_\beta (y_\alpha y_\beta)[s_{\alpha,\beta,\ldots}^{(4)}(y_\alpha y_\beta) - s_{\alpha,\beta,\ldots}^{(4)}(y_\alpha) s_{\alpha,\beta,\ldots}^{(5)}(y_\beta)]\\&+s_\beta s_\alpha (y_\alpha y_\beta)[S_{\beta,\alpha}^{(4,5)}(y_\alpha y_\beta)- y_\alpha s_{\beta,\alpha,\ldots}^{(5)}(y_\alpha)] \\&+s_\beta s_\alpha s_\beta (y_\alpha) s_{\beta,\alpha,\ldots}^{(5)}(y_\alpha)[s_\beta s_\alpha s_\beta (y_\beta) s_{\beta,\alpha,\ldots}^{(5)}(y_\beta) + y_\alpha s_{\beta,\alpha,\ldots}^{(4)}(y_\beta) - y_\alpha s_\beta s_\alpha s_\beta (y_\beta)] \\&- s_\beta (y_\alpha) s_\beta s_\alpha (y_\beta) s_\beta s_\alpha s_\beta (y_\alpha) s_{\beta,\alpha,\ldots}^{(4)}(y_\beta)\},\\
	k_{\beta,\alpha}^{(1)}=&s_\beta (y_\alpha y_\beta) s_\beta s_\alpha s_\beta (y_\alpha) s_{\beta,\alpha,\ldots}^{(5)}(y_\alpha)[s_\beta s_\alpha s_\beta (y_\beta) s_{\beta,\alpha\ldots}^{(5)}(y_\beta) + y_\alpha s_{\beta,\alpha,\ldots}^{(4)}(y_\beta) - y_\alpha s_\beta s_\alpha s_\beta (y_\beta)] \\&+
	y_\alpha y_\beta\{s_\beta(y_\alpha)s_\beta s_\alpha (y_\beta) s_\beta s_\alpha s_\beta (y_\alpha) s_{\beta,\alpha,\ldots}^{(4)}(y_\beta) +s_\alpha s_\beta (y_\alpha y_\beta ) s_{\alpha,\beta,\ldots}^{(4)}(y_\alpha y_\beta)\\&-s_\alpha s_\beta (y_\alpha y_\beta)s_{\alpha,\beta,\ldots}^{(4)}(y_\alpha) s_{\alpha,\beta,\ldots}^{(5)}(y_\beta)\}-y_\alpha s_\beta (y_\alpha) s_\beta s_\alpha (y_\beta) s_\beta s_\alpha s_\beta (y_\alpha) s_{\beta,\alpha,\ldots}^{(4)}(y_\beta) s_{\beta,\alpha,\ldots}^{(5)}(y_\alpha).
	\end{align*}
	These formulas are obtained using the method of \cref{ex:order-5}.
	
	\details{

		The formulas of $\kappa_{\beta}^5$ and $\kappa_{\alpha}^4$ follow from \cref{thm:linear-combo2}, respectively. Now we will find $\kappa_{\beta}^3$, $\kappa_{\alpha}^2$, and $\kappa_{\beta}^1$. In the \cref{calc:calculation}, we have computed the products of up to seven Demazure elements. We use these expressions, together with the coefficients $\kappa_{\beta}^5$ and $\kappa_{\alpha}^4$ that we already know, to determine the remaining coefficients by subtracting using the method of \cref{ex:order-5}.
		
		The coefficient of $X_{\beta\ldots}^7$ at $(\delta_{\beta\alpha}-\delta_{\beta\alpha\beta})$  can be written as
		
		\begin{align*}
		&+y_\alpha y_\beta^2 s_\beta (y_\alpha) s_\beta s_\alpha (y_\beta) [S_{\beta}^{(0,3)}(y_\alpha y_\beta)]\\
		&+y_\beta s_\beta (y_\alpha) s_\beta s_\alpha  (y_\alpha y_\beta^2)[S_{\beta}^{(1,3)}(y_\alpha y_\beta)]\\
		&+ y_\beta s_\beta (y_\alpha) s_\beta s_\alpha (y_\beta)s_\beta s_\alpha s_\beta (y_\alpha y_\beta)[S_{\beta}^{(1,4)}(y_\alpha y_\beta)]\\
		&+y_\beta s_\beta (y_\alpha^2 y_\beta) s_\beta s_\alpha (y_\beta) [S_\beta^{(0,2)}(y_\alpha y_\beta)]\\
		&+y_\alpha y_\beta^2 s_\beta (y_\alpha) s_\beta s_\alpha (y_\beta) s_\alpha (y_\alpha y_\beta).
		\end{align*}

		The coefficient of $\kappa_\beta^5X_{\beta\ldots}^5$ at $(\delta_{\beta\alpha}-\delta_{\beta\alpha\beta})$ can be written as
		
		\begin{align*}		
		&+y_\alpha y_\beta^2 s_\beta (y_\alpha) s_\beta s_\alpha (y_\beta) [S_{\beta}^{(0,3)}(y_\alpha y_\beta)]\\
		&+y_\beta s_\beta (y_\alpha) s_\beta s_\alpha (y_\alpha y_\beta^2)[S_{\beta}^{(1,3)}(y_\alpha y_\beta)]\\
		&+ y_\beta s_\beta (y_\alpha) s_\beta s_\alpha (y_\beta)s_\beta s_\alpha s_\beta (y_\alpha y_\beta)[S_{\beta}^{(1,4)}(y_\alpha y_\beta)]\\
		&+y_\beta s_\beta (y_\alpha^2 y_\beta) s_\beta s_\alpha (y_\beta) [S_\beta^{(0,2)}(y_\alpha y_\beta)]\\
		&+y_\alpha y_\beta^2 s_\beta (y_\alpha) s_\beta s_\alpha (y_\beta)[S_\beta^{(2,5)}(y_\alpha y_\beta) -y_\alpha s_{\beta\ldots}^5(y_\alpha)]\\
		&+y_\beta s_\beta (y_\alpha) s_\beta s_\alpha (y_\beta)\{[S_\beta^{(1,3)}(y_\alpha y_\beta)][s_{\beta\ldots}^5(y_\alpha y_\beta) - y_\alpha s_{\beta\ldots}^5(y_\alpha)]\\
		&+[S_\beta^{(1,2)}(y_\alpha y_\beta)][s_{\beta\ldots}^4(y_\beta y_\beta)] + s_\beta(y_\alpha y_\beta) s_\beta s_\alpha s_\beta (y_\alpha y_\beta)\}.	
		\end{align*}
		
		The coefficient of $-\kappa_\alpha^4X_{\alpha\ldots}^4$ at $(\delta_{\beta\alpha}-\delta_{\beta\alpha\beta})$ can be written as
		
		\begin{align*}	
		&+y_\alpha y_\beta^2 s_\beta (y_\alpha) s_\beta s_\alpha (y_\beta) s_\alpha (y_\alpha y_\beta)\\
		&-y_\alpha y_\beta^2 s_\beta (y_\alpha) s_\beta s_\alpha (y_\beta)[S_\beta^{(2,5)}(y_\alpha y_\beta) -y_\alpha s_{\beta\ldots}^5(y_\alpha)]\\
		&+ y_\alpha y_\beta^2 s_\beta (y_\alpha) s_\beta s_\alpha(y_\beta)[S_\alpha^{(2,4)}(y_\alpha y_\beta) - s_{\alpha\ldots}^4(y_\alpha) s_{\alpha\ldots}^5(y_\beta)].
		\end{align*}	
		
		After cancellations, the coefficient of $X_{\beta\ldots}^7 - \kappa_\beta^5X_{\beta\ldots}^5 + \kappa_\alpha^4X_{\alpha\ldots}^4$ at $(\delta_{\beta\alpha}-\delta_{\beta\alpha\beta})$ is
		
		\begin{align*}
		&-y_\beta s_\beta (y_\alpha) s_\beta s_\alpha (y_\beta)\{[S_\beta^{(1,3)}(y_\alpha y_\beta)][s_{\beta\ldots}^5(y_\alpha y_\beta) - y_\alpha s_{\beta\ldots}^5(y_\alpha)]\\
		&+[S_\beta^{(1,2)}(y_\alpha y_\beta)][s_{\beta\ldots}^4(y_\beta y_\beta)] + s_\beta(y_\alpha y_\beta) s_\beta s_\alpha s_\beta (y_\alpha y_\beta)+S_\alpha^{(2,4)}(y_\alpha y_\beta) \\
		&- s_{\alpha\ldots}^4(y_\alpha) s_{\alpha\ldots}^5(y_\beta)\}.
		\end{align*}
		
		The coefficient of $X_\beta X_\alpha X_\beta$ is $y_\beta s_\beta (y_\alpha) s_\beta s_\alpha (y_\beta)$. Hence, using \cref{cor:linear-combo-2} and \cref{lem:reflection-equality}, the coefficient at $X_\beta X_\alpha X_\beta$ is $\kappa_\beta^3$.
		
		\
		
		The coefficient of $X_{\alpha\ldots}^7$ at $(\delta_{\alpha\beta} - \delta_\alpha)$ can be written as
		
		\begin{align*}
		&+y_\alpha y_\beta s_\alpha (y_\beta)\{[S_\alpha^{(0,2)}(y_\alpha y_\beta) + s_\beta(y_\alpha y_\beta)][S_\beta^{(0,1)}(y_\alpha y_\beta)]+s_\beta(y_\alpha y_\beta) s_\beta s_\alpha (y_\alpha y_\beta)\}\\
		&+y_\alpha y_\beta s_\alpha (y_\beta) \{s_\alpha (y_\alpha y_\beta)[S_\alpha^{(0,2)}(y_\alpha y_\beta)] + s_\alpha s_\beta (y_\alpha y_\beta)[S_\alpha^{(1,3)}(y_\alpha y_\beta)]\}.
		\end{align*} 
		
		The coefficient of 	$\kappa_\beta^5X_{\beta\ldots}^5$ at $(\delta_{\alpha\beta} - \delta_\alpha)$ can be written as
		
		\begin{align*}	
		&+y_\alpha y_\beta s_\alpha (y_\beta)\{[S_\alpha^{(0,2)}(y_\alpha y_\beta) + s_\beta(y_\alpha y_\beta)][S_\beta^{(0,1)}(y_\alpha y_\beta)]+s_\beta(y_\alpha y_\beta) s_\beta s_\alpha (y_\alpha y_\beta)\}\\
		&+y_\alpha y_\beta s_\alpha(y_\beta)[S_\alpha^{(0,2)}(y_\alpha y_\beta)][S_\beta^{(2,5)}(y_\alpha y_\beta)- y_\alpha s_{\beta\ldots}^5(y_\alpha)]\\
		&+y_\alpha y_\beta s_\alpha (y_\beta) s_\beta (y_\alpha y_\beta)[S_\beta^{(3,5)}(y_\alpha y_\beta) - y_\alpha s_{\beta\ldots}^5(y_\alpha)].
		\end{align*}

		The coefficient of 	$-\kappa_\alpha^4X_{\alpha\ldots}^4$ at $(\delta_{\alpha\beta} - \delta_\alpha)$ can be written as
		
		\begin{align*}
		&+y_\alpha y_\beta s_\alpha (y_\beta) \{s_\alpha (y_\alpha y_\beta)[S_\alpha^{(0,2)}(y_\alpha y_\beta)] + s_\alpha s_\beta (y_\alpha y_\beta)[S_\alpha^{(1,3)}(y_\alpha y_\beta)]\}\\
		&-y_\alpha y_\beta s_\alpha(y_\beta)[S_\alpha^{(0,2)}(y_\alpha y_\beta)][S_\beta^{(2,5)}(y_\alpha y_\beta)- y_\alpha s_{\beta\ldots}^5(y_\alpha)]\\
		&+(y_\alpha y_\beta)^2 s_\alpha (y_\beta) [S_\alpha^{(2,4)}(y_\alpha y_\beta)- s_{\alpha\ldots}^4(y_\alpha) s_{\alpha\ldots}^5(y_\beta)]\\
		&+y_\alpha y_\beta s_\alpha (y_\beta)\{s_\alpha (y_\alpha y_\beta)[S_\alpha^{(3,4)}(y_\alpha y_\beta) - s_{\alpha\ldots}^4(y_\alpha) s_{\alpha\ldots}^5(y_\beta)] \\
		&+ s_\alpha s_\beta (y_\alpha y_\beta)[s_{\alpha\ldots}^4(y_\alpha y_\beta) - s_{\alpha\ldots}^4(y_\alpha) s_{\alpha\ldots}^5(y_\beta)]\}.
		\end{align*}
		
		The coefficient of 	$\kappa_\beta^3X_\beta X_\alpha X_\beta$ at $(\delta_{\alpha\beta} - \delta_\alpha)$ can be written as
		
		\begin{align*}
		&-y_\alpha y_\beta s_\alpha (y_\beta) s_\beta (y_\alpha y_\beta)[S_\beta^{(3,5)}(y_\alpha y_\beta) - y_\alpha s_{\beta\ldots}^5(y_\alpha)]\\
		&-(y_\alpha y_\beta)^2 s_\alpha (y_\beta) [S_\alpha^{(2,4)}(y_\alpha y_\beta)- s_{\alpha\ldots}^4(y_\alpha) s_{\alpha\ldots}^5(y_\beta)]\\
		&-y_\alpha y_\beta s_\alpha (y_\beta)\{s_\beta s_\alpha (y_\alpha y_\beta)[S_\beta^{(4,5)}(y_\alpha y_\beta) - y_\alpha s_{\beta\ldots}^5(y_\alpha)] \\
		&+ s_\beta s_\alpha s_\beta (y_\alpha) s_{\beta\ldots}^5(y_\alpha)[s_\beta s_\alpha s_\beta (y_\beta) s_{\beta\ldots}^5(y_\beta) + y_\alpha s_{\beta\ldots}^4(y_\beta) - y_\alpha s_\beta s_\alpha s_\beta (y_\beta)]\}.
		\end{align*}	
		
		After cancellations, the coefficient of $X_{\beta\ldots}^7 - \kappa_\beta^5X_{\beta\ldots}^5 + \kappa_\beta^4X_{\alpha\ldots}^4 - \kappa_\beta^3X_\beta X_\alpha X_\beta$ at $(\delta_{\alpha\beta} - \delta_\alpha)$ is
		
		\begin{align*}
		&-y_\alpha y_\beta s_\alpha (y_\beta)\{s_\alpha (y_\alpha y_\beta)[S_\alpha^{(3,4)}(y_\alpha y_\beta)- s_{\alpha\ldots}^4(y_\alpha) s_{\alpha\ldots}^5(y_\beta)] \\
		&+ s_\alpha s_\beta (y_\alpha y_\beta)[s_{\alpha\ldots}^4(y_\alpha y_\beta) - s_{\alpha\ldots}^4(y_\alpha) s_{\alpha\ldots}^5(y_\beta)]\\
		&-s_\beta s_\alpha (y_\alpha y_\beta)[S_\beta^{(4,5)}(y_\alpha y_\beta) - y_\alpha s_{\beta\ldots}^5(y_\alpha)]\\
		&- s_\beta s_\alpha s_\beta (y_\alpha) s_{\beta\ldots}^5(y_\alpha)[s_\beta s_\alpha s_\beta (y_\beta) s_{\beta\ldots}^5(y_\beta) + y_\alpha s_{\beta\ldots}^4(y_\beta) - y_\alpha s_\beta s_\alpha s_\beta (y_\beta)]\}.
		\end{align*}
		
		The coefficient of $X_\alpha X_\beta$ at $(\delta_{\alpha\beta} - \delta_\alpha)$ is $y_\alpha s_\alpha (y_\beta)$. Hence, using \cref{cor:linear-combo-2} and \cref{lem:reflection-equality}, the coefficient at $X_\alpha X_\beta$ is $\kappa_\alpha^2$.
		
		\
		
		The coefficient of $X_{\alpha\ldots}^7$ at $(\mathbf{1} - \delta_\beta)$ can be written as
		
		\begin{align*}
		&+ y_\beta\{[S_\beta^{(0,1)}(y_\alpha y_\beta)]^3+s_\beta (y_\alpha y_\beta) s_\beta s_\alpha (y_\alpha y_\beta)[S_\beta^{(0,1)}(y_\alpha y_\beta)+S_\beta^{(0,3)}(y_\alpha y_\beta)] \\
		&+ y_\alpha y_\beta s_\alpha (y_\alpha y_\beta)[S_\beta^{(0,1)}(y_\alpha y_\beta)]\}\\
		&+y_\alpha y_\beta^2 s_\alpha (y_\alpha y_\beta)[S_\beta^{(0,1)}(y_\alpha y_\beta) +S_\alpha^{(1,2)}(y_\alpha y_\beta)].		
		\end{align*}
		
		The coefficient of $\kappa_\beta^5X_{\beta\ldots}^5$ at $(\mathbf{1} - \delta_\beta)$ can be written as
		
		\begin{align*}
		&+ y_\beta\{[S_\beta^{(0,1)}(y_\alpha y_\beta)]^3+s_\beta (y_\alpha y_\beta) s_\beta s_\alpha (y_\alpha y_\beta)[S_\beta^{(0,1)}(y_\alpha y_\beta)+S_\beta^{(0,3)}(y_\alpha y_\beta)] \\
		&+ y_\alpha y_\beta s_\alpha (y_\alpha y_\beta)[S_\beta^{(0,1)}(y_\alpha y_\beta)]\}\\
		&+y_\alpha y_\beta^2[S_\alpha^{(0,1)}(y_\alpha y_\beta) + s_\beta (y_\alpha y_\beta)][S_\beta^{(2,5)}(y_\alpha y_\beta) - y_\alpha s_{\beta\ldots}^5(y_\alpha)]\\
		&+y_\beta s_\beta (y_\alpha y_\beta)\{[S_\beta^{(0,1)}(y_\alpha y_\beta)][S_\beta^{(3,5)}(y_\alpha y_\beta) - y_\alpha s_{\beta\ldots}^5(y_\alpha)]\\
		&+s_\beta s_\alpha (y_\alpha y_\beta)[S_\beta^{(4,5)}(y_\alpha y_\beta)- y_\alpha s_{\beta\ldots}^5(y_\alpha)]\}.
		\end{align*}

		The coefficient of $-\kappa_\alpha^4X_{\alpha\ldots}^4$ at $(\mathbf{1} - \delta_\beta)$ can be written as
		
		\begin{align*}
		&+y_\alpha y_\beta^2 s_\alpha (y_\alpha y_\beta)\{S_\beta^{(0,1)}(y_\alpha y_\beta) +S_\alpha^{(1,2)}(y_\alpha y_\beta)\}\\
		&-y_\alpha y_\beta^2[S_\alpha^{(0,1)}(y_\alpha y_\beta) + s_\beta (y_\alpha y_\beta)][S_\beta^{(2,5)}(y_\alpha y_\beta) - y_\alpha s_{\beta\ldots}^5(y_\alpha)]\\
		&+y_\alpha y_\beta^2[S_\beta^{(0,1)}(y_\alpha y_\beta)][S_\alpha^{(2,4)}(y_\alpha y_\beta) - s_{\alpha\ldots}^4(y_\alpha) s_{\alpha\ldots}^5(y_\beta)]\\
		&+y_\alpha y_\beta^2 s_\alpha (y_\alpha y_\beta)\{S_\alpha^{(3,4)}(y_\alpha y_\beta) - s_{\alpha\ldots}^4(y_\alpha) s_{\alpha\ldots}^5(y_\beta)\}.
		\end{align*}
		
		The coefficient of $\kappa_\beta^3X_\beta X_\alpha X_\beta$ at $(\mathbf{1} - \delta_\beta)$ can be written as

		\begin{align*}
		&-y_\beta s_\beta (y_\alpha y_\beta)\{[S_\beta^{(0,1)}(y_\alpha y_\beta)][S_\beta^{(3,5)}(y_\alpha y_\beta) - y_\alpha s_{\beta\ldots}^5(y_\alpha)]\\
		&+s_\beta s_\alpha (y_\alpha y_\beta)[S_\beta^{(4,5)}(y_\alpha y_\beta)- y_\alpha s_{\beta\ldots}^5(y_\alpha)]\}\\
		&-y_\alpha y_\beta^2[S_\beta^{(0,1)}(y_\alpha y_\beta)][S_\alpha^{(2,4)}(y_\alpha y_\beta) - s_{\alpha\ldots}^4(y_\alpha) s_{\alpha\ldots}^5(y_\beta)]\\
		&-y_\alpha y_\beta^2\{s_\beta s_\alpha (y_\alpha y_\beta)[S_\beta^{(4,5)}(y_\alpha y_\beta) - y_\alpha s_{\beta\ldots}^5(y_\alpha)] \\
		&+ s_\beta s_\alpha s_\beta (y_\alpha) s_{\beta\ldots}^5(y_\alpha)[s_\beta s_\alpha s_\beta (y_\beta) s_{\beta\ldots}^5(y_\beta) + y_\alpha s_{\beta\ldots}^4(y_\beta) - y_\alpha s_\beta s_\alpha s_\beta (y_\beta)]\}\\
		&-y_\beta s_\beta (y_\alpha y_\beta) s_\beta s_\alpha s_\beta (y_\alpha) s_{\beta\ldots}^5(y_\alpha)[s_\beta s_\alpha s_\beta (y_\beta) s_{\beta\ldots}^5(y_\beta) + y_\alpha s_{\beta\ldots}^4(y_\beta) - y_\alpha s_\beta s_\alpha s_\beta (y_\beta)].
		\end{align*}

		The coefficient of $-\kappa_\alpha^2X_\alpha X_\beta$ at $(\mathbf{1} - \delta_\beta)$ can be written as
		
		\begin{align*}
		&-y_\alpha y_\beta^2 s_\alpha (y_\alpha y_\beta)\{S_\alpha^{(3,4)}(y_\alpha y_\beta) - s_{\alpha\ldots}^4(y_\alpha) s_{\alpha\ldots}^5(y_\beta)\}\\
		&+y_\alpha y_\beta^2\{s_\beta s_\alpha (y_\alpha y_\beta)[S_\beta^{(4,5)}(y_\alpha y_\beta) - y_\alpha s_{\beta\ldots}^5(y_\alpha)] \\
		&+ s_\beta s_\alpha s_\beta (y_\alpha) s_{\beta\ldots}^5(y_\alpha)[s_\beta s_\alpha s_\beta (y_\beta) s_{\beta\ldots}^5(y_\beta) + y_\alpha s_{\beta\ldots}^4(y_\beta) - y_\alpha s_\beta s_\alpha s_\beta (y_\beta)]\}\\
		&-y_\alpha y_\beta^2\{s_\beta(y_\alpha)s_\beta s_\alpha (y_\beta) s_\beta s_\alpha s_\beta (y_\alpha) s_{\beta\ldots}^4(y_\beta) + s_\alpha s_\beta (y_\alpha y_\beta ) s_{\alpha\ldots}^4(y_\alpha y_\beta) \\
		&- s_\alpha s_\beta (y_\alpha y_\beta)s_{\alpha\ldots}^4(y_\alpha) s_{\alpha\ldots}^5(y_\beta)\}.	
		\end{align*}
		
		After cancellations, the coefficient of	
		$X_{\beta\ldots}^7 - \kappa_\beta^5X_{\beta\ldots}^5+\kappa_\alpha^4X_{\alpha\ldots}^4 - \kappa_\beta^3X_\beta X_\alpha X_\beta + \kappa_\alpha^2X_\alpha X_\beta$ at $(\mathbf{1} - \delta_\beta)$ is

		\begin{align*}
		&+y_\beta \{s_\beta (y_\alpha y_\beta) s_\beta s_\alpha s_\beta (y_\alpha) s_{\beta\ldots}^5(y_\alpha)[s_\beta s_\alpha s_\beta (y_\beta) s_{\beta\ldots}^5(y_\beta) + y_\alpha s_{\beta\ldots}^4(y_\beta) \\
		&- y_\alpha s_\beta s_\alpha s_\beta (y_\beta)] + y_\alpha y_\beta[s_\beta(y_\alpha)s_\beta s_\alpha (y_\beta) s_\beta s_\alpha s_\beta (y_\alpha) s_{\beta\ldots}^4(y_\beta) \\
		&+ s_\alpha s_\beta (y_\alpha y_\beta ) s_{\alpha\ldots}^4(y_\alpha y_\beta) - s_\alpha s_\beta (y_\alpha y_\beta)s_{\alpha\ldots}^4(y_\alpha) s_{\alpha\ldots}^5(y_\beta)].\}
		\end{align*}
		
		The coefficient of $X_\beta$ at $(\mathbf{1} - \delta_\beta)$ is $y_\beta$. Hence, using \cref{cor:linear-combo-2} and \cref{lem:reflection-equality}, the coefficient at $X_\beta$ is $\kappa_\beta^1$.	
	}		
\end{ex}

\begin{rem}\label{rem:zero}
	Let $(R,F)$ be a formal group law such that, for all $\alpha\in\Sigma$, $x_\alpha+x_{-\alpha}-qx_\alpha x_{-\alpha}=0$,  where $q\in R$ is fixed.
	By direct computation and using the method of \cref{ex:special-cases}, we compute $\kappa_{\alpha,\beta}^{(1)}=\kappa_{\beta,\alpha}^{(1)}$ when $m=5$. We also compute $\kappa_{\alpha,\beta}^{(3)}=\kappa_{\beta,\alpha}^{(3)}$, $\kappa_{\alpha,\beta}^{(2)}=\kappa_{\beta,\alpha}^{(2)}=0$, and $\kappa_{\alpha,\beta}^{(1)}=\kappa_{\beta,\alpha}^{(1)}$ when $m=7$.
\end{rem}

\begin{rem}
	Suppose $m\geq 3$ is \textit{odd}. Let $(R,F)$ be a formal group law such that, for all $\alpha\in\Sigma$, $x_\alpha+x_{-\alpha}-qx_\alpha x_{-\alpha}=0$,  where $q\in R$ is fixed. In light of \cref{ex:special-cases} and \cref{rem:zero}, we conjecture that $\kappa_{\alpha,\beta}^{(i)}=\kappa_{\beta,\alpha}^{(i)}$ when $i$ is odd, and $\kappa_{\alpha,\beta}^{(i)}=\kappa_{\beta,\alpha}^{(i)}=0$ when $i$ is even.
\end{rem}

\begin{rem}{\label{rem:specialized-3}}
	Let $(\Sigma,\Delta)$ be the geometric realization of $W=I_2(3)$ given in \cref{ex:I2(3)}, and let $(R,F_a)$ be the additive formal group law over $R$. In this case, $\Delta=\{\alpha,\beta\}$. Using the formulas for the roots given in \cref{ex:I2(3)} and the formula for $\kappa_{\beta,\alpha}^{(1)}$ given in \cref{ex:lower-order-relations} \ref{itm:order-3}, we compute directly that $\kappa_{\beta,\alpha}^{(1)}=0$. Thus, we have given a combinatorial proof of \cref{rem:equality-in-additive-demazure-algebra} for this geometric realization of $I_2(3)$.
\end{rem}

\details{\begin{rem}{\label{rem:specialized-4}}
	Let $(\Sigma,\Delta)$ be the geometric realization of $W=I_2(4)$ given in \cref{ex:I2(4)}, and let $(R,F_a)$ be the additive formal group law over $R$. In this case, $\Delta=\{\alpha,\beta\}$. Using the formulas for the roots given in \cref{ex:I2(4)} and the formulas for $\kappa_{\beta,\alpha}^{(1)}$ and $\kappa_{\alpha,\beta}^{(2)}$ given in \cref{ex:lower-order-relations} \ref{itm:order-3}, we compute directly that $\kappa_{\beta,\alpha}^{(1)}=\kappa_{\alpha,\beta}^{(2)}=0$. Thus, we have given a combinatorial proof of \cref{rem:equality-in-additive-demazure-algebra} for this geometric realization of $I_2(4)$.
\end{rem}}

\begin{rem}\label{rem:specialized-5}
	Let $(\Sigma,\Delta)$ be the geometric realization of $W=I_2(5)$ given in \cref{ex:I2(5)}, and let $(R,F_a)$ be the additive formal group law over $R$. In this case, $\Delta=\{\alpha,\beta\}$. We already know that $\kappa_{\alpha,\beta}^{(2)}=0$ by \cref{ex:special-cases}. Moreover, using the formulas for the roots given in \cref{ex:I2(5)}, the formulas for the structure coefficients given in \cref{ex:order-5}, and the relation $\tau^2=\tau+1$, we compute directly that $\kappa_{\beta,\alpha}^{(1)}=\kappa_{\beta,\alpha}^{(3)}=0$ as well. Thus, we have given a combinatorial proof of \cref{rem:equality-in-additive-demazure-algebra} for this geometric realization of $I_2(5)$.
\end{rem}

\begin{rem}\label{rem:specialized-7}
	Let $(\Sigma,\Delta)$ be the geometric realization of $W=I_2(7)$ given in \cref{ex:I2(7)}, and let $(R,F_a)$ be the additive formal group law over $R$. In this case, $\Delta=\{\alpha,\beta\}$. We already know that $\kappa_{\alpha,\beta}^{(4)}=0$ by \cref{ex:special-cases}, and that $\kappa_{\alpha,\beta}^{(2)}=0$ by \cref{rem:zero}. Moreover, using the formulas for the roots given in \cref{ex:I2(7)}, the formulas for the structure coefficients given in \cref{ex:order-7}, and the relation $\zeta^3=\zeta^2+2\zeta-1$, we compute directly that $\kappa_{\beta,\alpha}^{(1)}=\kappa_{\beta,\alpha}^{(3)}=\kappa_{\beta,\alpha}^{(5)}=0$ as well. Thus, we have given a combinatorial proof of \cref{rem:equality-in-additive-demazure-algebra} for this geometric realization of $I_2(7)$.
\end{rem}

We will now consider the reflection groups $H_3$ and $H_4$.

\begin{ex}\label{ex:H3}
Let $(\Sigma,\Delta)$ be the geometric realization of $W=H_3$ given in \cref{section:generalized-root-lattice}. This geometric realization has simple roots $\Delta=\{\alpha_1,\alpha_2,\alpha_3\}$, where
	\begin{align*}
	m_{1,2}=5,\quad m_{1,3}=2,\quad m_{2,3}=3
	\end{align*}
	are the orders of $s_is_j$ in $W$. The coefficient $\kappa_{\alpha_2,\alpha_3}^{(1)}$ was computed in \cref{ex:lower-order-relations} \ref{itm:order-3}. In addition, $X_1X_3=X_3X_1$ by \cref{ex:lower-order-relations} \ref{itm:order-2}. Moreover, the coefficients $\kappa_{\alpha_1,\alpha_2}^{(1)}$, $\kappa_{\alpha_1,\alpha_2}^{(2)}$, $\kappa_{\alpha_1,\alpha_2}^{(3)}$ were computed in \cref{ex:order-5}. 
	
	When $(R,F_a)$ is the additive formal group law over $R$, one can verify directly, as was done in \cref{rem:specialized-3}, that the coefficient $\kappa_{\alpha_2,\alpha_3}^{(1)}=0$. One can also verify directly, as was done in \cref{rem:specialized-5}, that the coefficients  $\kappa_{\alpha_1,\alpha_2}^{(1)}$, $\kappa_{\alpha_1,\alpha_2}^{(2)}$, and $\kappa_{\alpha_1,\alpha_2}^{(3)}$ all equal $0$ in this case.
\end{ex}
\begin{ex}\label{ex:H4}
	Let $(\Sigma,\Delta)$ be the geometric realization of $W=H_4$ given in \cref{section:generalized-root-lattice}. This geometric realization has simple roots $\Delta=\{\alpha_1,\alpha_2,\alpha_3,\alpha_4\}$, where
	\begin{align*}
	m_{1,2}=5,\quad m_{1,3}=2,\quad m_{1,4}=2,\quad m_{2,3}=3,\quad m_{2,4}=2,\quad m_{3,4}=3
	\end{align*}
	are the orders of $s_is_j$ in $W$.
	The coefficients $\kappa_{\alpha_2,\alpha_3}^{(1)}$ and $\kappa_{\alpha_3,\alpha_4}^{(1)}$ were computed in \cref{ex:lower-order-relations} \ref{itm:order-3}. In addition, $X_1X_3=X_3X_1$, $X_1X_4=X_4X_1$, and $X_2X_4=X_4X_2$ by \cref{ex:lower-order-relations} \ref{itm:order-2}. Moreover, the coefficients $\kappa_{\alpha_1,\alpha_2}^{(1)}$, $\kappa_{\alpha_1,\alpha_2}^{(2)}$, $\kappa_{\alpha_1,\alpha_2}^{(3)}$ were computed in \cref{ex:order-5}. 
	
	When $(R,F_a)$ is the additive formal group law over $R$, one can verify directly, as was done in \cref{rem:specialized-3}, that the coefficients $\kappa_{\alpha_2,\alpha_3}^{(1)}$ and $\kappa_{\alpha_3,\alpha_4}^{(1)}$ equal $0$. One can also verify directly, as was done in \cref{rem:specialized-5}, that the coefficients  $\kappa_{\alpha_1,\alpha_2}^{(1)}$, $\kappa_{\alpha_1,\alpha_2}^{(2)}$, and $\kappa_{\alpha_1,\alpha_2}^{(3)}$ all equal $0$ in this case.
\end{ex}
\section{Appendix}\label{section:appendix}
In the present section, we provide computations for products of up to seven formal Demazure elements for all root systems.
\begin{calc}\label{calc:calculation}
	Let $W$ be a real finite reflection group, let $\Sigma$ be a root system of $W$, and let $\alpha,\beta\in\Sigma$. Below are explicit formulas for the products of $X_\alpha$ and $X_\beta$ up to seven elements. The formulas are written so that the coefficient to the left of a $\delta$-term in the expansion of the product appears after the colon.
	\begin{align*}
	&\underline{X_\beta} \\& (\mathbf{1}-\delta_\beta): y_\beta\\
	&\underline{X_\alpha X_\beta}\\ &(\mathbf{1}-\delta_\beta):y_\alpha y_\beta\\
	&(\delta_{\alpha\beta}-\delta_\alpha):y_\alpha s_\alpha (y_\beta)\\
	&\underline{X_\beta X_\alpha X_\beta}\\
 	&(\mathbf{1}-\delta_\beta):y_\beta\{S_{\beta,\alpha}^{(0,1)}(y_\alpha y_\beta)\}\\
	&(\delta_{\alpha\beta}-\delta_\alpha):y_\alpha y_\beta s_\alpha (y_\beta)\\
	&(\delta_{\beta\alpha} - \delta_{\beta \alpha \beta}):y_\beta s_\beta (y_\alpha) s_\beta s_\alpha (y_\beta)	\\
	&\underline{X_{\alpha,\beta,\ldots}^{(4)}} \\&(\mathbf{1}-\delta_\beta): y_\alpha y_\beta\{y_\alpha y_\beta + s_\beta(y_\alpha y_\beta) + s_\alpha(y_\alpha y_\beta)\}	\\
	&(\delta_{\alpha\beta}-\delta_\alpha): y_\alpha s_{\alpha}(y_\beta)\{S_{\alpha,\beta}^{(0,2)}(y_\alpha y_\beta)\}\\
	&(\delta_{\beta\alpha} - \delta_{\beta \alpha \beta}):y_\alpha y_\beta s_\beta(y_\alpha) s_\beta s_\alpha (y_\beta)\\
	&(\delta_{\alpha,\beta,\ldots}^4 - \delta_{\alpha \beta \alpha}):  y_\alpha s_\alpha (y_\beta) s_\alpha s_\beta (y_\alpha) s_\alpha s_\beta s_\alpha (y_\beta)\\
	&\underline{X_{\beta,\alpha,\ldots}^{(5)}}\\
	&(\mathbf{1}-\delta_\beta):y_\beta\{[S_{\beta,\alpha}^{(0,1)}(y_\alpha y_\beta)]^2  + y_\alpha y_\beta s_\alpha(y_\alpha y_\beta)  +s_\beta(y_\alpha y_\beta) s_\beta s_\alpha (y_\alpha y_\beta)\}\\
	&(\delta_{\alpha\beta}-\delta_\alpha): y_\alpha y_\beta s_\alpha (y_\beta)\{S_{\alpha,\beta}^{(0,2)}(y_\alpha y_\beta) + s_\beta (y_\alpha y_\beta)\}	 \\
	&(\delta_{\beta\alpha} - \delta_{\beta \alpha \beta}): y_\beta s_\beta (y_\alpha) s_\beta s_\alpha(y_\beta)\{S_{\beta,\alpha}^{(0,3)}(y_\alpha y_\beta)\}\\
	&(\delta_{\alpha,\beta,\ldots}^{(4)} - \delta_{\alpha\beta\alpha}): y_\alpha y_\beta s_\alpha (y_\beta) s_\alpha s_\beta (y_\alpha) s_\alpha s_\beta s_\alpha (y_\beta)	\\
	&(\delta_{\beta,\alpha,\ldots}^{(4)} - \delta_{\beta,\alpha,\ldots}^{(5)}): y_\beta s_\beta(y_\alpha) s_\beta s_\alpha(y_\beta) s_\beta s_\alpha s_\beta (y_\alpha) s_{\beta,\alpha,\ldots}^{(4)} (y_\beta)	\\
	&\underline{X_{\alpha,\beta,\ldots}^{(6)}}\\
	&(\mathbf{1}-\delta_\beta):y_\alpha y_\beta\{[S_{\beta,\alpha}^{(0,1)}(y_\alpha y_\beta)+s_{\alpha}(y_\alpha y_\beta)]^2 + s_\beta (y_\alpha y_\beta) s_\beta s_\alpha (y_\alpha y_\beta) \\&+ s_\alpha (y_\alpha y_\beta) s_\alpha s_\beta (y_\alpha y_\beta)- s_\alpha (y_\alpha y_\beta) s_\beta (y_\alpha y_\beta)\}\\
	&(\delta_{\alpha\beta} - \delta_\alpha):y_\alpha s_{\alpha}(y_\beta)\{[S_{\alpha,\beta}^{(0,2)}(y_\alpha y_\beta)]^2  + y_\alpha y_\beta s_\beta (y_\alpha y_\beta) \\&+ s_\alpha s_\beta (y_\alpha y_\beta) s_\alpha s_\beta s_\alpha (y_\alpha y_\beta)- y_\alpha y_\beta s_\alpha s_\beta (y_\alpha y_\beta)\}\\
	&(\delta_{\beta\alpha} - \delta_{\beta\alpha\beta}):y_\alpha y_\beta s_\beta (y_\alpha)s_\beta s_\alpha(y_\beta)\{S_{\beta,\alpha}^{(0,3)}(y_\alpha y_\beta) + s_\alpha (y_\alpha y_\beta)\}	\\
	&(\delta_{\alpha,\beta,\ldots}^{(4)} - \delta_{\alpha\beta\alpha}):y_\alpha s_\alpha (y_\beta) s_\alpha s_\beta (y_\alpha) s_\alpha s_\beta s_\alpha (y_\beta)\{S_{\alpha,\beta}^{(0,4)}(y_\alpha y_\beta)\}\\
	&(\delta_{\beta,\alpha,\ldots}^{(4)} - \delta_{\beta,\alpha,\ldots}^{(5)}):y_\alpha y_\beta s_\beta (y_\alpha) s_\beta s_\alpha (y_\beta) s_\beta s_\alpha s_\beta (y_\alpha) s_{\beta,\alpha,\ldots}^{(4)} (y_\beta)\\
	&(\delta_{\alpha,\beta,\ldots}^{(6)}- \delta_{\alpha,\beta,\ldots}^{(5)}):y_\alpha s_\alpha (y_\beta) s_\alpha s_\beta (y_\alpha) s_\alpha s_\beta s_\alpha (y_\beta) s_{\alpha,\beta,\ldots}^{(4)}(y_\alpha)s_{\alpha,\beta,\ldots}^{(5)}(y_\beta)\\
	&\underline{X_{\beta\ldots}^{(7)}}\\
	&(\mathbf{1}-\delta_\beta):y_\beta\{[S_{\beta,\alpha}^{(0,1)}(y_\alpha y_\beta)]^3 + 2(y_\alpha y_\beta)^2 s_\alpha(y_\alpha y_\beta)\\&+ 2y_\alpha y_\beta s_\alpha(y_\alpha y_\beta)s_\beta(y_\alpha y_\beta) +2y_\alpha y_\beta s_\beta(y_\alpha y_\beta) s_\beta s_\alpha (y_\alpha y_\beta) \\
	&+ 2[s_\beta(y_\alpha y_\beta)]^2 s_\beta s_\alpha (y_\alpha y_\beta) + y_\alpha y_\beta[s_\alpha(y_\alpha y_\beta)]^2 
	\end{align*}
	\begin{align*}&+ y_\alpha y_\beta s_\alpha (y_\alpha y_\beta) s_\alpha s_\beta (y_\alpha y_\beta) + s_\beta (y_\alpha y_\beta) [s_\beta s_\alpha(y_\alpha y_\beta)]^2 \\&+ s_\beta (y_\alpha y_\beta)s_\beta s_\alpha 	\\
	&(\delta_{\alpha\beta}-\delta_\alpha):y_\alpha y_\beta s_{\alpha}(y_\beta)\{[S_{\alpha,\beta}^{(0,2)}(y_\alpha y_\beta)+s_\beta(y_\alpha y_\beta)]^2 \\&+ s_\alpha s_\beta (y_\alpha y_\beta) s_\alpha s_\beta s_\alpha (y_\alpha y_\beta) +  s_\beta (y_\alpha y_\beta)s_\beta s_\alpha (y_\alpha y_\beta)- y_\alpha y_\beta s_\alpha s_\beta (y_\alpha y_\beta) \\&- s_\alpha (y_\alpha y_\beta) s_\beta (y_\alpha y_\beta) - s_\beta (y_\alpha y_\beta) s_\alpha s_\beta (y_\alpha y_\beta)\}
	(y_\alpha y_\beta) s_\beta s_\alpha s_\beta (y_\alpha y_\beta)\}	\\
	&(\delta_{\beta\alpha}-\delta_{\beta \alpha \beta}):y_\beta s_\beta (y_\alpha) s_\beta s_\alpha (y_\beta) \{[S_{\beta,\alpha}^{(0,3)}(y_\alpha y_\beta)]^2 + y_\alpha y_\beta s_\alpha (y_\alpha y_\beta) \\&+ s_\beta s_\alpha s_\beta (y_\alpha y_\beta)s_{\beta,\alpha,\ldots}^{(4)}y_\alpha y_\beta)-y_\alpha y_\beta s_\beta s_\alpha (y_\alpha y_\beta) - y_\alpha y_\beta s_\beta s_\alpha s_\beta (y_\alpha y_\beta) \\&- s_\beta (y_\alpha y_\beta) s_\beta s_\alpha s_\beta (y_\alpha y_\beta)\}\\
	&(\delta_{\alpha,\beta,\ldots}^{(4)} - \delta_{\alpha\beta\alpha}):y_\alpha y_\beta s_\alpha (y_\beta) s_\alpha s_\beta (y_\alpha) s_\alpha s_\beta s_\alpha (y_\beta) \{S_{\alpha,\beta}^{(0,4)}(y_\alpha y_\beta) + s_\beta(y_\alpha y_\beta)\}	\\
	&(\delta_{\beta,\alpha,\ldots}^{(4)} - \delta_{\beta,\alpha,\ldots}^{(5)}):y_\beta s_\beta (y_\alpha) s_\beta s_\alpha (y_\beta) s_\beta s_\alpha s_\beta (y_\alpha) s_{\beta,\alpha,\ldots}^{(4)}(y_\beta)\{S_{\beta,\alpha}^{(0,5)}(y_\alpha y_\beta)\}\\
	&(\delta_{\alpha,\beta,\ldots}^{(6)} - \delta_{\alpha,\beta,\ldots}^{(5)}):y_\alpha y_\beta s_\alpha(y_\beta) s_\alpha s_\beta (y_\alpha) s_\alpha s_\beta s_\alpha (y_\beta) s_{\alpha,\beta,\ldots}^{(4)}(y_\alpha) s_{\alpha,\beta,\ldots}^{(5)}(y_\beta)\\
	&(\delta_{\beta,\alpha,\ldots}^{(6)} - \delta_{\beta,\alpha,\ldots}^{(7)}):y_\beta s_\beta(y_\alpha) s_\beta s_\alpha (y_\beta) s_\beta s_\alpha s_\beta (y_\alpha) s_{\beta,\alpha,\ldots}^{(4)}(y_\beta) s_{\beta,\alpha,\ldots}^{(5)}(y_\alpha) s_{\beta,\alpha,\ldots}^{(6)}(y_\beta).
	\end{align*}
\end{calc}


\end{document}